\documentclass[12pt,a4paper]{amsart}
\usepackage{amsthm,amsfonts,amsmath,amssymb,latexsym}
\usepackage{epsfig,graphics,color}
\usepackage[all]{xy}
\usepackage[breaklinks=true]{hyperref}
\usepackage{mathrsfs}
\usepackage{stmaryrd}
\usepackage{verbatim}
\usepackage{bm}
\usepackage{mathabx}
\usepackage{enumitem}
\usepackage{pgf,amsmath,tikz,type1cm,fix-cm}
\usetikzlibrary{positioning}
 
\newlength{\XHeight}
\newlength{\XWidth}
\setlist[itemize,1]{leftmargin=\dimexpr 26pt-.1in}

\newtheorem{PARA}{}[section]
\newtheorem{theorem}[PARA]{Theorem}
\newtheorem{corollary}[PARA]{Corollary}
\newtheorem{lemma}[PARA]{Lemma}
\newtheorem{proposition}[PARA]{Proposition}
\newtheorem{definition}[PARA]{Definition}
\newtheorem{definition-proposition}[PARA]{Definition-Proposition}
\newtheorem{definition-lemma}[PARA]{Definition-Lemma}

\newtheorem{question}[PARA]{Question}
\theoremstyle{definition}
\newtheorem{remark}[PARA]{Remark}
\theoremstyle{theorem}
\newtheorem{example}[PARA]{Example}
\newtheorem{examples}[PARA]{Examples}
\newcommand{\para}{\begin{PARA}\rm}
\newcommand{\arap}{\end{PARA}\rm}
\newcommand{\dfn}{\begin{definition}\rm}
\newcommand{\nfd}{\end{definition}\rm}
\newcommand{\rmk}{\begin{remark}\rm}
\newcommand{\kmr}{\end{remark}\rm}
\newcommand{\xmpl}{\begin{example}\rm}
\newcommand{\lpmx}{\end{example}\rm}
\newcommand{\cA}{\mathcal{A}}

\newcommand{\cC}{\mathcal{C}}

\newcommand{\cM}{\mathcal{M}}

\DeclareMathOperator{\hatotimes}{{\hat{\otimes}}}

\newcommand{\C}{{\mathbb{C}}}

\renewcommand{\H}{{\mathbb{H}}}

\newcommand{\N}{{\mathbb{N}}}

\newcommand{\R}{{\mathbb{R}}}

\newcommand{\Z}{{\mathbb{Z}}}
\newcommand{\coker}{\mathrm{ coker }}  
\newcommand{\colim}{\mathrm{ colim}\, }  

\newcommand{\ev}{\mathrm{ev}}

\newcommand{\Top}{\mathrm{Top}}

\DeclareMathOperator{\Int}{Int}
\newcommand{\Hom}{\mathrm{Hom}}

\newcommand{\eps}{{\varepsilon}}
\newcommand{\om}{{\omega}}
\newcommand{\Om}{{\Omega}}

\def\NABLA#1{{\mathop{\nabla\kern-.5ex\lower1ex\hbox{$#1$}}}}
\def\Nabla#1{\nabla\kern-.5ex{}_{#1}}
\def\Tabla#1{\Tilde\nabla\kern-.5ex{}_{#1}}
\renewcommand{\Tilde}{\widetilde}

\newcommand{\p}{{\partial}}

\newcommand{\la}{\langle}
\newcommand{\ra}{\rangle}
\newcommand{\wh}{\widehat}
\newcommand{\ol}{\overline}

\newcommand{\wt}{\widetilde}

\newcommand{\into}{\hookrightarrow}

\newcommand{\bk}{\mathbf{k}}
\newcommand{\pt}{\mathrm{pt}}
\definecolor{darkgreen}{rgb}{0.12, 0.3, 0.17}
\definecolor{burntorange}{rgb}{0.8, 0.33, 0.0}
\definecolor{chromeyellow}{rgb}{1.0, 0.65, 0.0}
\definecolor{darkorange}{rgb}{1.0, 0.55, 0.0}
\definecolor{flame}{rgb}{0.89, 0.35, 0.13}
\definecolor{lightgray}{rgb}{0.75, 0.75, 0.75}
\usetikzlibrary{arrows}
\usetikzlibrary{shapes}

\newcommand{\boldbeta}{{\boldsymbol{\beta}}}
\newcommand{\boldlambda}{{\boldsymbol{\lambda}}}
\newcommand{\boldmu}{{\boldsymbol{\mu}}}
\newcommand{\boldeta}{{\boldsymbol{\eta}}}
\newcommand{\boldeps}{{\boldsymbol{\eps}}}
\newcommand{\boldc}{{\boldsymbol{c}}}
\newcommand{\boldp}{{\boldsymbol{p}}}
\newcommand{\boldB}{{\boldsymbol{B}}}
\newcommand{\boldS}{{\boldsymbol{S}}}
\newcommand{\boldzeta}{{\boldsymbol{\zeta}}}

\begin{document}

\title{Topological Frobenius algebras}
\author{Kai Cieliebak}
\address{Universit\"at Augsburg \newline Universit\"atsstrasse 14, D-86159 Augsburg, Germany}
\email{kai.cieliebak@math.uni-augsburg.de}
\author{Alexandru Oancea}
\address{
Universit\'e de Strasbourg \newline 
Institut de recherche math\'ematique avanc\'ee, IRMA \newline
Strasbourg, France}
\email{oancea@unistra.fr}
\date{\today}


\begin{abstract}
We define in a graded and topologized setting the notions of unital/counital/biunital infinitesimal anti-symmetric bialgebras and coFrobenius bialgebras and discuss their algebraic properties. We also define in this setting the notion of a 2D open-closed TQFT. We illustrate these structures with old and new examples of geometric origin that arise in contexts as varied as symplectic topology, string topology, manifold topology, and quantum homology. The underlying vector spaces, which are typically infinite dimensional, belong to a class of topological vector spaces known as Tate vector spaces. 
\end{abstract}

\maketitle

{\tableofcontents}

\section{Introduction}\label{sec:intro}
  
In this paper we introduce three new kinds of algebraic structures
and illustrate them with examples from
symplectic topology, string topology, and manifold topology.  
Along the way we give a survey of various old and new geometric instantiations of Frobenius algebras and $1+1$ dimensional Topological Quantum Field Theories (TQFT). We define these structures for graded linearly topologized vector spaces, possibly of infinite dimension, over discrete fields. This represents a two-fold generalization of the more classical setting which uses ungraded finite dimensional vector spaces. One of our goals is to introduce and advertise the framework of linearly topologized vector spaces as the natural one to accommodate infinite dimensional TQFTs of geometric origin.

This paper should be accessible and, we hope, interesting, for a broad array of readers. Classical and explicit computations from manifold topology connect to recent structure theorems in Floer theory and string topology under the unifying algebraic theme of Frobenius algebras and TQFTs. In addition, the framework of linearly topologized vector spaces may catch the attention of readers with a more functional analytic background.

Since our main examples are infinite dimensional, we work in the category of {\em ($\Z$-)graded Tate vector spaces} $A$ over a discrete field $\bk$. Tate vector spaces are a class of topological vector spaces originally introduced by Lefschetz under the name ``locally linearly compact vector spaces''~\cite{Lefschetz-book}. They are studied extensively in~\cite{CO-Tate}, and we recall their definition and basic properties in~\S\ref{sec:Tate}. Most relevant for our purposes are the following facts:
(1) If $A$ is a graded Tate vector space, then so is its topological dual $A^\vee$ and the canonical map $A\to A^{\vee\vee}$ is an isomorphism.
(2) Graded Tate vector spaces carry two natural tensor products $\otimes^*$ and $\otimes^!$ satisfying
$$
  (A\otimes^* B)^\vee\simeq A^\vee\otimes^!B^\vee\quad\text{and}\quad (A\otimes^! B)^\vee\simeq A^\vee\otimes^*B^\vee.
$$

The first algebraic structure that we present in this paper is that of {\em unital infinitesimal anti-sym\-metric bialgebra} $(A,\boldmu,\boldlambda,\boldeta)$ consisting of an associative product $\boldmu:A\otimes^* A\to A$ with unit $\boldeta\in A$ and a coassociative coproduct $\boldlambda:A\to A\otimes^! A$ satisfying the {\sc (unital infinitesimal relation)} and {\sc (unital anti-symmetry)} (Definition~\ref{defi:secondary-unital}). This structure appears on reduced symplectic homology and reduced loop homology~\cite{CHO-reducedSH}, where the {\sc (unital infinitesimal relation)} is the correct general form of a relation conjectured by D.~Sullivan in~\cite{Sullivan-open-closed}. Our {\sc (unital infinitesimal relation)} simultaneously generalizes the defining relation of infinitesimal bialgebras of Aguiar~\cite{Aguiar} and that of unital infinitesimal bialgebras of Loday-Ronco~\cite{Loday-Ronco}, see Remark~\ref{rmk:ncIIB-biblio} for a discussion of the bibliography. Its {\em counital} version $(A,\boldmu,\boldlambda,\boldeps)$ (where $\boldeps:A\to\bk$ is the counit, see Definition~\ref{defi:secondary-counital}) appears on reduced symplectic cohomology and reduced loop cohomology, where $\boldmu$ is a generalization of the loop cohomology product from~\cite{Goresky-Hingston}. Its {\em biunital} version $(A,\boldmu,\boldlambda,\boldeta,\boldeps)$ (Definition~\ref{defi:secondary-biunital}) appears on Rabinowitz Floer homology and Rabinowitz loop homology~\cite{CHO-PD} as a consequence of the biunital coFrobenius bialgebra structure that follows.  

The second algebraic structure is that of
\emph{biunital coFrobenius bialgebra} ($A,\boldmu,\boldlambda,\boldeta,\boldeps)$ consisting of an associative product $\boldmu$ with unit $\boldeta$ and a coassociative coproduct $\boldlambda$ with counit $\boldeps$ satisfying the {\sc (biunital coFrobenius relation)} and {\sc (symmetry)} (Definition~\ref{defi:biunital-coFrobenius-bialgebra}). This is a graded and infinite dimensional version of the notion of a symmetric Frobenius algebra (see~\cite{Abrams,Kock} for background on symmetric Frobenius algebras and their equivalence to $1+1$ dimensional TQFTs). This structure appears on the homology of a closed oriented manifold (\S\ref{sec:manifolds}), on quantum homology of a closed symplectic manifold (\S\ref{sec:quantum}), on Rabinowitz Floer homology of a Liouville fillable contact manifold, and on Rabinowitz loop homology of a closed manifold~\cite{CHO-PD}. The case of even dimensional closed oriented manifolds and that of closed symplectic manifolds is classical, our contribution in finite dimension is to define the graded structure and to give new examples if the dimension is odd, and also, more importantly, to introduce the formalism of Tate vector spaces which is able to accommodate our infinite dimensional examples.

We give several equivalent characterizations of a biunital coFrobenius bialgebra structure, show that it induces that of a biunital infinitesimal anti-symmetric bialgebra (Proposition~\ref{prop:biunital-coFrob-bialg}), and prove that it is isomorphic to the induced biunital coFrobenius bialgebra on its dual (Theorem~\ref{thm:algebraic_Poincare_duality}). This is an algebraic version of the Poincar\'e duality theorem for Rabinowitz Floer homology in~\cite{CHO-PD}. We also prove that biunital coFrobenius bialgebras can be equivalently characterized as Tate Frobenius algebras (Proposition~\ref{prop:cyclic-graded-algebra}). The latter are graded and topological counterparts of symmetric Frobenius algebras.  

The third algebraic structure is that of
{\em graded 2D open-closed TQFT}, which we define in~\S\ref{sec:TQFT} building on the notion of a biunital coFrobenius bialgebra. This structure appears on the cohomology of closed finite dimensional manifolds/submanifolds (\S\ref{sec:manifolds}), on quantum homology of monotone symplectic manifolds/Lagrangian submanifolds (\S\ref{sec:quantum}), and on Rabinowitz Floer homology of Liouville fillable contact manifolds/exact Lagrangian fillable Legendrian submanifolds~\cite{CHO-PD}. The ungraded case has been extensively studied, and our contribution is again to define the graded structure and to give new examples in odd dimensions and in infinite dimensions. 
One particular feature of a graded 2D open-closed TQFT is that the direct analogue of the so-called \emph{Cardy condition} only makes sense when the degree of the coproduct for closed strings is twice the degree of the coproduct for open strings, see Definition~\ref{defi:gradedTQFT}. Motivated by our main infinite dimensional example of Rabinowitz Floer homology, we  formulate a general Cardy condition for closed/open coproducts of arbitrary degrees.  

To our knowledge, the notion of a graded open-closed TQFT has not been considered before in the literature. This is somewhat surprising because, as explained in~\S\ref{sec:manifolds}, this structure appears naturally on the cohomology rings $(H^*M,H^*Z)$ of a closed oriented manifold $M$ with closed oriented submanifold $Z\subset M$. The graded open-closed TQFT in 
Theorem~\ref{thm:open-closed-RFH} associated to Rabinowitz Floer homology of a Liouville fillable contact manifold $\p V$ with exact oriented Lagrangian fillable Legendrian submanifold $\p L\subset \p V$ extends the graded open-closed TQFT on the manifold pair $\p L\subset\p V$. We view this as the contact analogue of the fact that the open-closed TQFT in quantum homology of a closed monotone symplectic manifold $M$ with closed oriented monotone Lagrangian submanifold $L$ extends the open-closed TQFT on the manifold pair $L\subset M$.

TQFT-like structures did appear previously in symplectic topology. An early landmark reference written from the perspective of closed Gromov-Witten invariants is Manin's book~\cite{Manin-book}. M.~Schwarz has constructed a Frobenius algebra structure on the Floer homology of a closed symplectic manifold~\cite{Schwarz95}, whereas P.~Seidel~\cite{Seidel07} and A.~Ritter~\cite{Ritter} have constructed a noncompact open-closed TQFT structure on the symplectic homology and wrapped Floer homology of a Liouville domain and an exact Lagrangian submanifold with Legendrian boundary. Note that the first case is finite dimensional, whereas the infinite dimensional second case carries only part of the TQFT operations. An open-closed TQFT structure on quantum homology of closed monotone Lagrangians is implicit in the work of Biran-Cornea~\cite{Biran-Cornea-rigidity-uniruling,Biran-Cornea-Lagrangian-topology}, as we explain in~\S\ref{sec:quantum}. Our Theorem~\ref{thm:open-closed-RFH} provides a large class of infinite dimensional examples carrying full (graded) open-closed TQFT structures. We refer to the paper of Moore-Segal~\cite{Moore-Segal} for a comprehensive list of references on open-closed TQFT structures outside of symplectic topology. In particular, we consider it interesting to study our structure in relation with the work of  Costello~\cite{Costello}, Godin~\cite{Godin} and Wahl-Westerland~\cite{Wahl-Westerland}. We refer to~\S\ref{sec:TQFT} for a more detailed comparison with these works.

The paper is organized as follows. In~\S\ref{sec:Tate} we 
survey the basic facts about Tate vector spaces and extend them to the graded setting. 
In~\S\ref{sec:inf-bialg}--\S\ref{sec:TQFT} we introduce and discuss the algebraic structures in the setting of graded Tate vector spaces. 
In~\S\ref{sec:manifolds}--\S\ref{sec:RFH} we 
explain the appearance of these structures 
in manifold homology, quantum homology, symplectic homology, and Rabinowitz Floer homology. 
In~\S\ref{sec:spheres} we illustrate them in the example of loop space homology of odd-dimensional spheres.
The Appendix contains a discussion of signs for compatibility of maps with operations in a graded setting.

{\bf Acknowledgments}.
This paper is a split-off from our collaboration with Nancy Hingston on Poincar\'e duality. Without her far reaching vision this could not have come to being. We have benefited from discussions with Jonathan Laurent Clivio, Pavel Safronov, and Nathalie Wahl. We would like to thank the anonymous referee for their comments and questions, which helped improve the paper. We would also like to thank the Institute for Advanced Study, Princeton, for its support in the academic year 2021/22, during which much of this work was completed. 
A.O. was partially funded by ANR grants ENUMGEOM 18-CE40-0009 and COSY 21-CE40-0002, and by a Fellowship of the University of Strasbourg Institute for Advanced Study within the French national programme ``Investment for the future" (IdEx-Unistra).

\section{Graded Tate vector spaces}\label{sec:Tate}

In this section we introduce graded Tate vector spaces, their duality theory, and their tensor products. We base our presentation on~\cite{CO-Tate}.

\subsection*{Tate vector spaces}
These spaces were introduced by Lefschetz~\cite[II]{Lefschetz-book} under the name of {\em locally linearly compact vector spaces}, and were from the onset designed as linear analogues of locally compact groups. They were used by Tate~\cite{Tate68} in his treatment of residues of differentials on curves, and later revived by Beilinson-Feigin-Mazur~\cite{BFM91} and Beilinson-Drinfeld~\cite{BD04} under the name of {\em Tate vector spaces} with a different, though equivalent, definition. See also Rojas~\cite{Rojas}. We will use the latter terminology.

We fix a discrete field $\bk$, i.e., a field equipped with the discrete topology. 
A {\em linearly topologized vector space} is a $\bk$-vector space $A$
with a Hausdorff topology which is translation invariant and has a
basis of open neighborhoods of $0$ consisting of linear subspaces. 
Linearly topologized vector spaces form a category $\Top$ whose morphisms are
continuous linear maps. The kernel and cokernel of a morphism $f:A\to B$ are defined as
$$
  \ker(f) = \{v\in A\mid f(v)=0\}, \qquad 
  \coker(f) = B/\ol{f(A)}. 
$$
In particular, quotients in $\Top$ should always be taken by closed
linear subspaces. Note that an open linear subspace $U$ is also closed
because its complement can be written as a union of translates of $U$.

A special and important class of linearly topologized vector spaces are the discrete ones, i.e., those for which $\{0\}$ is an open set. The continuity condition is automatically satisfied for linear maps between discrete vector spaces, and topological linear algebra coincides in this context with linear algebra. However, away from discrete vector spaces the two notions differ and this difference is essential for our purposes.  

\begin{definition} The {\em completion} of a linearly topologized vector space $A$ is
$$
   \wh A := \lim_{U\in\mathcal{U}} A/U,
$$
where $\mathcal{U}$ denotes the collection of open linear subspaces in
$A$. Here the inverse limit is topologized as a subset of the product $\prod_{U\in\mathcal{U}} A/U$, and each $A/U$ is discrete with respect to the quotient topology.
\end{definition}

The canonical map $A \to \wh A$ is always injective and its image is dense in $\wh A$.
The space $A$ is called {\em complete} if the canonical map $A \to \wh A$ is an
isomorphism. Any discrete vector space is complete because the inverse system
$A/U$, $U\in\mathcal{U}$, has a maximal element equal to $A$.

\begin{definition}[Beilinson-Feigin-Mazur]  
A linear subspace $L\subset A$ is called \emph{linearly bounded} if $\dim(L/L\cap U)<\infty$ for each open
  linear subspace $U\subset A$. 
\end{definition}

\begin{definition} \label{defi:Tate}
A linearly topologized vector space is \emph{linearly compact} if it is complete and linearly bounded. 
It is a \emph{Tate vector space} if it is complete and admits an open linearly compact subspace.
\end{definition}

\begin{examples} \label{ex:Tate}
(i) A linearly topologized vector space is finite dimensional if and only if it is discrete and linearly compact. 

(ii) The vector space $\bk[[t]]$ of formal power series, with a basis of neighborhoods of $0$ given by $t^n\bk[[t]]$, $n\ge 0$, is linearly compact.

(iii) The vector space $\bk[t^{-1},t]]$ of Laurent power series, with a basis of neighborhoods of $0$ given by $t^n\bk[[t]]$, $n\in\Z$, is Tate. It is neither discrete nor linearly compact. 
\end{examples}

\begin{proposition}
A linearly topologized vector space $A$ is Tate iff it is a topological direct sum $A= L\oplus D$, with $L$ linearly compact and $D$ discrete. 
\qed
\end{proposition}

\begin{remark}
Lefschetz gives in~\cite[(II.27.1)]{Lefschetz-book} a different equivalent definition of linear compactness in terms of the finite intersection property for linear varieties. His definition resonates well with other instances of compactness from general topology and provides by analogy a good intuition for linearly topologized vector spaces. The equivalence of the Lefschetz definition with Definition~\ref{defi:Tate} is proved in~\cite{CO-Tate}. The advantage of the latter is that it is more effective, e.g., it applies in a straightforward way to the examples~\ref{ex:Tate}.      
\end{remark}

\subsection*{Duality}
For linearly topologized vector spaces $A,B$ we denote by $\Hom(A,B)$ the space of continuous linear maps $A\to B$. Following Lefschetz~\cite[(II.28.1)]{Lefschetz-book}, we equip it with the {\em compact-open topology}, i.e., the linear topology whose neighborhood basis of the origin is given by the linear subspaces $\{f\in \Hom(A,B)\mid f(L)\subset V\}$ for $L\subset A$ linearly compact and $V\subset B$ linear open. 
An important special case of this is the topological dual $A^\vee=\Hom(A,\bk)$, whose neighborhood
basis of the origin is given by the linear subspaces
$$
   L^\perp=\{\alpha\in A^\vee\mid \alpha|_L=0\},\qquad
   L\subset A \text{ linearly compact}.
$$ 
The association $A\mapsto A^\vee$ is functorial: a continuous linear map $f:A\to B$ induces a continuous linear map $f^\vee:B^\vee\to A^\vee$, $f^\vee\beta=\beta\circ f$. 

\begin{theorem}[Lefschetz-Tate duality, Lefschetz~{\cite[(II.28.2-29.1)]{Lefschetz-book}}, Rojas~{\cite[Theorem~1.25]{Rojas}}]\label{thm:Tate-duality}\quad 

(a) If $A$ is discrete, then $A^\vee$ is linearly compact. 

(b) If $A$ is linearly compact, then $A^\vee$ is discrete.
  
(c) If $A$ is locally linearly compact, then so is $A^\vee$ and the canonical map $A\to
  A^{\vee\vee}$ is a topological isomorphism. \qed
\end{theorem}

The last reflexivity property of Tate vector spaces plays a fundamental role in phrasing the axioms of a TQFT in~\S\ref{sec:TQFT}. 

There is a canonical evaluation map $\ev:A^\vee\times A\to \bk$, $\ev(f,a)=f(a)$. This map is bilinear and continuous in each variable. It is continuous as a bilinear map if and only if $A$ is Tate~\cite[Proposition~4.12]{CO-Tate}. 

\subsection*{Tensor products}
Let $A,B$ be linearly topologized vector spaces over the discrete field $\bk$. Beilinson describes in~\cite[\S1.1]{Beilinson} two topologies on the algebraic tensor product $A\otimes B$, called the \emph{$*$ topology} and the \emph{$!$ topology}. We will subsequently consider the completed tensor products with respect to these topologies.\footnote{Beilinson also describes in~\cite[\S1.1]{Beilinson} an intermediate topology, called the $^\leftarrow$ topology, which we will not need here. We refer to~\cite{CO-Tate} for further details.} 

\begin{definition}\label{def:*top-alex}
A linear subspace $Q\subset A\otimes B$ is open in the {\em $*$
  topology} iff it satisfies the following conditions:
\begin{enumerate}[label=(\roman*)]
\item there exist open linear subspaces $U\subset A$ and $V\subset B$
  such that $U\otimes V\subset Q$;  
\item for each $a\in A$ there exists an open linear subspace $V\subset
  B$ such that $a\otimes V\subset Q$;
\item for each $b\in B$ there exists an open linear subspace $U\subset
  A$ such that $U\otimes b\subset Q$.
\end{enumerate}
\end{definition}

\begin{definition}\label{def:!top}
A linear subspace $Q\subset A\otimes B$ is open in the {\em $!$ topology} iff there exists open linear subspaces $U\subset A$ and $V\subset B$ such that $U\otimes B + A\otimes V\subset Q$.
\end{definition}

We denote $A\otimes^* B$ and $A\otimes^!B$ the tensor product $A\otimes B$ endowed with the $*$ topology, respectively the $!$ topology. We denote $A\hatotimes^*B$ and $A\hatotimes^!B$ their respective completions.

The $!$ topology is coarser than the $*$ topology. In particular, the identity map is continuous as a map $A\otimes^* B\to A\otimes^! B$, and it induces a continuous linear map 
$$
A\hatotimes^* B\to A\hatotimes^! B.
$$
The algebraic tensor product $A\otimes B$ (devoid of topology) is endowed with a canonical bilinear map $\pi:A\times B\to A\otimes B$ and is characterized by the following universal property: for any vector space $C$ and any bilinear map $\phi:A\times B\to C$ there is a unique linear map $\phi^\otimes:A\otimes B\to C$ such that $\phi=\phi^\otimes\circ\pi$. 
$$
\xymatrix{
A\times B\ar[r]^-\phi \ar[d]_-\pi & C \\
A\otimes B\ar[ur]_-{\phi^\otimes} &  
}
$$ 
The $*$ topology on $A\otimes B$ is designed to render this correspondence topological:

\begin{lemma}\label{lem:*topology-universal}
The $*$ topology on $A\otimes B$ is uniquely characterized by any one of the following two conditions. 

(a) It is the finest linear topology such that the canonical bilinear map $\pi:A\times B\to A\otimes B$
is continuous. 

(b) For each linearly topologized space $C$ the assignment $\phi^\otimes\mapsto \phi=\phi^\otimes\circ\pi$ defines a linear bijection
\begin{equation}\label{eq:lin-bilin}
  \pi^*:\Hom(A\otimes^*B,C) \to B(A\times B,C),
\end{equation}
with $B(A\times B,C)$ the space of continuous bilinear maps $A\times B\to C$.
\qed
\end{lemma}

The space of bilinear maps $B(A\times B,C)$ is equipped with the compact-open topology, i.e., the linear topology generated by the open linear subspaces $S_{K\times L,W}=\{f\in B(A\times B,C)\, : \, f(K\times L)\subset W\}$ for $K\subset A$, $L\subset B$ linearly compact and $W\subset C$ open. 
This topology is used in the following characterization of the $!$ topology on a tensor product, which is somewhat dual to that of the $*$ topology.

\begin{proposition}\label{prop:characterization!topologyTate}
If $A$ and $B$ are Tate, then the $!$ topology is the coarsest topology on $A\otimes B$ such that the linear map 
$$
\beta^\otimes:A\otimes B\to B(A^\vee\times B^\vee,\bk),\qquad a\otimes b \mapsto \big( (\varphi,\psi)\mapsto \varphi(a)\psi(b)\big)
$$
is continuous. \qed
\end{proposition}

Next, we collect the main properties of the two completed tensor products. We begin with commutativity and associativity. 

\begin{proposition} \label{prop:comm-ass}
For linearly topologized vector spaces $A,B,C$ we have canonical isomorphisms
\begin{gather*}
  A\hatotimes^*B\simeq B\hatotimes^*A, \qquad
  (A\hatotimes^*B)\hatotimes^*C \simeq A\hatotimes^*(B\hatotimes^*C), \cr
  A\hatotimes^!B\simeq B\hatotimes^!A, \qquad
  (A\hatotimes^!B)\hatotimes^!C \simeq A\hatotimes^!(B\hatotimes^!C).
\end{gather*}
Moreover, the identity induces a continuous linear map
$$
   A\hatotimes^*(B\hatotimes^!C) \to (A\hatotimes^*B)\hatotimes^!C.
$$
\qed
\end{proposition}

\begin{proposition} \label{prop:tensor_discrete}
(a) If $A$ and $B$ are discrete, then both topologies on $A\otimes B$ coincide and 
$$
A\otimes B = A\hatotimes^*B=A\hatotimes^!B  = \Hom(B^\vee,A) = \Hom(A^\vee,B)
$$
is discrete.

(b) If $A$ and $B$ are linearly compact, then both topologies on $A\otimes B$ coincide and 
$$
A\hatotimes^*B=A\hatotimes^!B=\Hom(B^\vee,A)=\Hom(A^\vee,B)
$$
is linearly compact. \qed
\end{proposition}

Note that $A\otimes^*B$ is in general not complete, hence not linearly compact. For example $\bk[[t]]\otimes^*\bk[[s]]$ is not complete, but $\bk[[t]]\hatotimes^*\bk[[s]]=\bk[[t,s]]$ is complete and linearly compact. 

\begin{proposition}[Lefschetz-Tate duality with monoidal structure~I]  \label{prop:duality*!}
Assume that $A$ and $B$ are both linearly compact, or both discrete. Then we have topological isomorphisms
$$
(A\hatotimes^* B)^\vee\simeq A^\vee\hatotimes^! B^\vee,\qquad (A\hatotimes^! B)^\vee\simeq A^\vee\hatotimes^*B^\vee.
$$
\qed
\end{proposition}

The vector spaces involved in this statement are either discrete or linearly compact. Remarkably, a much more general result holds. 

\begin{theorem}[Lefschetz-Tate duality with monoidal structure~II]   \label{thm:duality*!-Tate}
For Tate vector spaces $A,B$ we have topological isomorphisms
$$
(A\hatotimes^* B)^\vee\simeq A^\vee\hatotimes^! B^\vee,\qquad (A\hatotimes^! B)^\vee\simeq A^\vee\hatotimes^*B^\vee.
$$
\qed
\end{theorem}

\begin{remark}
Upgrading Proposition~\ref{prop:duality*!} to Theorem~\ref{thm:duality*!-Tate} requires to work outside the category of Tate vector spaces. The fundamental reason is that, if $A$ and $B$ are Tate, then $A\hatotimes^! B\simeq \Hom(A^\vee,B)$ need not be Tate, and $A\hatotimes^* B$ need not be Tate either, cf.~\cite{CO-Tate}. The correct setup is to consider the categories of ind-linearly compact spaces and pro-discrete spaces, which are respectively preserved under the $\hatotimes^*$ and the $\hatotimes^!$ tensor products, and which are monoidal dual to each other~\cite{CO-Tate,Esposito-Penkov22,Esposito-Penkov23}.
\end{remark}

\begin{remark}
The * and ! tensor product topologies are analogues of the projective and injective tensor product topologies for Banach spaces, see~\cite[\S2-3]{Ryan}. 
\end{remark}

\begin{remark} \label{rmk:ev}
If $A$ is Tate, then the evaluation map $\ev:A^\vee\times A\to\bk$ is continuous as a bilinear map and therefore induces a continuous linear map $\ev:A^\vee\otimes^* A\to\bk$. 
\end{remark}

\subsection*{Graded Tate vector spaces}
A {\em ($\Z$-)graded linearly topologized vector space} is a direct sum $A=\bigoplus_{i\in\Z}A_i$ of linearly topologized vector spaces. It is a \emph{graded Tate vector space} if $A_i$ is Tate for each $i\in\Z$. The morphism spaces in the category of graded linearly topologized vector spaces are by definition  
\begin{align*}
  \Hom(A,B) &= \bigoplus_{k\in\Z}\Hom_k(A,B),\quad
  \Hom_k(A,B) &= \prod_i \Hom(A_i,B_{i+k}).
\end{align*}
The degree of a homogeneous element $a\in A$ is denoted $|a|$. A linear map $\phi:A\to B$ is \emph{homogeneous of degree $k$} if $|\phi(a)|=|a|+k$ for homogeneous elements $a\in A$. A homogeneous map of degree $k$ is therefore a collection $\phi=(\phi_i)$ with $\phi_i\in\Hom(A_i,B_{i+k})$. The degree of a homogeneous linear map $\phi$ is denoted $|\phi|$. 

The \emph{dual} of a graded linearly topologized vector space is defined to be  
$$
  A^\vee = \bigoplus_{i\in\Z}A^\vee_i,\qquad A^\vee_i=\Hom(A_{-i},\bk).
$$
The $*$ and $!$ tensor products of graded linearly topologized vector spaces are the graded linearly topologized vector spaces 
\begin{gather*}
A\otimes^* B = \bigoplus_{k\in\Z}(A\otimes^* B)_k,\qquad 
  (A\otimes^* B)_k = \bigoplus_{i+j=k}A_i\hatotimes^* B_j, \cr
A\otimes^! B = \bigoplus_{k\in\Z}(A\otimes^! B)_k,\qquad 
  (A\otimes^! B)_k = \prod_{i+j=k}A_i\hatotimes^! B_j.
\end{gather*}
Note the degree-wise completions, and the appearance of a direct sum in $\otimes^*$ versus a direct product in $\otimes^!$. 

\begin{remark}\label{rmk:discrete-infinite-degrees}
While the * and ! tensor products coincide for ungraded discrete vector spaces, they may be different for graded discrete vector spaces. A typical example is the following: let $A$ and $B$ be finite dimensional in each degree, but supported in infinitely many positive and negative degrees. In this case $(A\otimes^* B)_k=\oplus_{i+j=k}A_i\otimes B_j$ and $(A\otimes^! B)_k=\prod_{i+j=k}A_i\otimes B_j$. This example will be relevant in~\S\ref{sec:spheres}.  
\end{remark}

The results of the previous three subsections carry over in an obvious way to the graded setting. 
To avoid duplications, we will write $\heartsuit$ for any of the symbols $*$ and $!$, and $\bar\heartsuit$ for the opposite symbol $!$ or $*$. For example, in this notation the graded analogue of Proposition~\ref{prop:comm-ass} reads

\begin{proposition} \label{prop:comm-ass-graded}
For graded linearly topologized vector spaces $A$, $B$, $C$ and $\heartsuit=*$ or $!$ we have canonical isomorphisms
\begin{gather*}
  A\otimes^\heartsuit B\simeq B\otimes^\heartsuit A, \qquad
  (A\otimes^\heartsuit B)\otimes^\heartsuit C \simeq A\otimes^\heartsuit(B\otimes^\heartsuit C).
\end{gather*}
Moreover, the identity induces a continuous linear map
$$
   A\otimes^*(B\otimes^!C) \to (A\otimes^*B)\otimes^!C.
$$
\qed
\end{proposition}

The graded analogue of Theorem~\ref{thm:duality*!-Tate} (derived using the fact that the dual of a topological direct sum is a topological direct product and vice versa~\cite{CO-Tate}) reads

\begin{theorem}[Graded Lefschetz-Tate duality with monoidal structure]   \label{thm:duality*!-Tate-graded}
For graded Tate vector spaces $A,B$ we have topological isomorphisms
$$
\iota:A^\vee\otimes^\heartsuit B^\vee\stackrel\simeq\longrightarrow (A\otimes^{\bar\heartsuit} B)^\vee, \quad \heartsuit=* \mbox{ or } !
$$
\qed
\end{theorem}

We will usually specify the kind of tensor product that we consider for vector spaces, e.g., $A\otimes^*B$, or $A\otimes^\heartsuit B$. In contrast, whenever we write a tensor product of maps we will simply use the symbol $\otimes$ without further decorations since the domain and target can be inferred from the context, e.g., $f\otimes g$.

\subsection*{Conventions} 
We describe in this section sign conventions in the context of graded linearly topologized vector spaces $A,B,C$. 

0. {\it Identity map.} The identity map of $A$ is denoted $1:A\to A$. 

1. {\it Twist.} The twist $\tau:A\otimes^\heartsuit A\to A\otimes^\heartsuit A$ is induced by the componentwise graded twist $\tau(a\otimes b)=(-1)^{|a||b|}b\otimes a$. 

2. {\it Product.} We call \emph{product} a continuous linear map 
$$
\boldmu:A\otimes^* A\to A.
$$ 
The degree of $\boldmu$ is denoted $|\boldmu|$. 
We say that $\boldmu$ is \emph{commutative} if 
$$
\boldmu \tau = (-1)^{|\boldmu|}\boldmu.
$$
We say that $\boldmu$ is \emph{associative} if 
$$
\boldmu(\boldmu\otimes 1) = (-1)^{|\boldmu|}\boldmu(1\otimes \boldmu).
$$
An element $\eta\in A$ is called a \emph{unit} for $\boldmu$ if
$$
(-1)^{|\boldmu|} \boldmu(\boldeta\otimes 1) = 1 = \boldmu(1\otimes \boldeta).
$$

3. {\it Coproduct.} We call \emph{coproduct} a continuous linear map 
$$
\boldlambda:A\to A\otimes^! A.
$$ 
The degree of $\boldlambda$ is denoted $|\boldlambda|$. 
We say that $\boldlambda$ is \emph{cocommutative} if 
$$
\tau \boldlambda = (-1)^{|\boldlambda|}\boldlambda. 
$$
We say that $\boldlambda$ is \emph{coassociative} if 
$$
(\boldlambda\otimes 1)\boldlambda = (-1)^{|\boldlambda|}(1\otimes \boldlambda)\boldlambda. 
$$
A map $\boldeps:A\to \bk$ is called a \emph{counit} for $\boldlambda$ if 
$$
(\boldeps\otimes 1)\boldlambda = 1 = (-1)^{|\boldlambda|}(1\otimes\boldeps)\boldlambda.
$$

4. {\it Duality.} We have a canonical degree $0$ linear continuous pairing 
$$
\langle \cdot,\cdot\rangle :A^\vee\otimes^* A\to \bk,\qquad \langle f,a\rangle = f(a).
$$ 
Given a graded map $A\stackrel\varphi\longrightarrow B$, the graded \emph{dual map} $B^\vee\stackrel{\varphi^\vee}\longrightarrow A^\vee$ is defined by $\langle \varphi^\vee g, a \rangle = (-1)^{|g|  |\varphi|}\langle g, \varphi(a)\rangle$. 
Given graded maps $A\stackrel \varphi\longrightarrow B \stackrel \psi\longrightarrow C$, we have 
$$
(\psi\circ \varphi)^\vee = (-1)^{|\varphi|  |\psi|}\varphi^\vee \circ \psi^\vee. 
$$
The canonical isomorphism 
$$
\iota:A^\vee\otimes^\heartsuit B^\vee\stackrel\simeq\longrightarrow (A\otimes^{\bar\heartsuit} B)^\vee
$$ 
is  
determined by the degree-wise map $\iota(f\otimes g)(a\otimes b)= (-1)^{|g|  |a|} f(a)g(b)$. 
Note that $\iota\tau=\tau^\vee\iota:A^\vee\otimes^\heartsuit A^\vee\to(A\otimes^{\bar\heartsuit} A)^\vee$. 

Given graded maps $A\stackrel \varphi \longrightarrow A'$ and $B\stackrel \psi\longrightarrow B'$, the graded map $\varphi\otimes \psi:A\otimes^\heartsuit B\to A'\otimes^\heartsuit B'$ is defined by $(\varphi\otimes \psi)(a\otimes b) = (-1)^{|\psi|  |a|} \varphi(a)\otimes \psi(b)$. 
We have 
$$
(\varphi\otimes \psi)^\vee = \varphi^\vee\otimes \psi^\vee. 
$$
In this formula, if the domain of $\varphi\otimes\psi$ is $A\otimes^\heartsuit B$, then the domain of $\varphi^\vee\otimes\psi^\vee$ is ${A'}^\vee\otimes^{\bar\heartsuit} {B'}^\vee$. 
The formula could be written more precisely $\iota^{-1}(\varphi\otimes \psi)^\vee\iota =  \varphi^\vee\otimes \psi^\vee$, but we omit for readability the mention of the canonical identification $\iota$.   

5. {\it Dualizing products and coproducts.} 
\renewcommand{\theenumi}{\roman{enumi}}
\begin{enumerate}
\item Given a (coassociative) (cocommutative) (counital) coproduct $\boldlambda:A\to A\otimes^! A$, the map $\boldlambda^\vee=A^\vee\otimes^* A^\vee\to A^\vee$ obtained by precomposing the dual of $\boldlambda$ with the canonical map $\iota$ is an (associative) (commutative) (unital) product on $A^\vee$. If $\boldeps$ is the counit of $\boldlambda$ then $\boldeps^\vee$ is the unit of $\boldlambda^\vee$. 
\item Given an (associative) (commutative) (unital) product $\boldmu:A\otimes^* A\to A$, the map $\boldmu^\vee:A^\vee\to A^\vee\otimes^! A^\vee$ obtained by postcomposing the dual of $\boldmu$ with the inverse $\iota^{-1}$ is a (coassociative) (cocommutative) (counital) coproduct on $A^\vee$. If $\boldeta$ is the unit of $\boldmu$ then $\boldeta^\vee$ is the counit of $\boldmu^\vee$. 
\item We have 
$$
\boldlambda^{\vee\vee}=\boldlambda,\qquad \boldmu^{\vee\vee}=\boldmu
$$
via the canonical isomorphism $A\to A^{\vee\vee}$, $a\mapsto a^{\vee\vee}$ defined on $f\in A^\vee$ by $\la a^{\vee\vee},f\ra=(-1)^{|a|}f(a)$.
\end{enumerate}

\begin{remark}
That products are defined on the $*$ tensor product is natural in view of Lemma~\ref{lem:*topology-universal} and the fact that they model bilinear 2-to-1 operations. That coproducts take values in the $!$ tensor product is motivated by Theorem~\ref{thm:duality*!-Tate} and the previous dualization scheme: any coproduct is dual to a product defined on the $*$ tensor product of the duals.    
\end{remark}

6. {\it Shift.} We denote $A[1]$ the graded linearly topologized vector space given by $A[1]_i=A_{i+1}$. We denote 
$$
s:A\longrightarrow A[1],\qquad \omega:A[1]\to A
$$
the canonical maps of degrees $|s|=-1$, $|\om|=1$ induced by the identity on $A$. These maps are inverse to each other. 

Given a linear map $\varphi:A^{\otimes^* k}\to A^{\otimes^! \ell}$, $k,\ell\ge 0$, we denote its \emph{shift} by 
$$
\overline \varphi : A[1]^{\otimes^* k}\to A[1]^{\otimes^! \ell}, \qquad \overline \varphi = s^{\otimes \ell}\circ \varphi\circ \omega^{\otimes k}.
$$

The notions of associativity, commutativity, coassociativity, cocommutativity are invariant under shifts. 

The existence of a unit and that of a counit is a property that is also invariant under shift: 

(i) If $\boldeta\in A$ is a unit for $\boldmu$, then 
$$
\overline \boldeta = (-1)^{|\mu|}s\boldeta
$$
is a unit for $\overline \boldmu$. 

(ii) If $\boldeps:A\to \bk$ is a counit for $\boldlambda$, then 
$$
\overline \boldeps = (-1)^{|\boldlambda|}\boldeps \omega
$$
is a counit for $\overline \boldlambda$. 

7. {\it Shifts and duals.} The shift and the dual commute up to a sign given by the degree of the operation. If $\boldmu$ is a product with unit $\boldeta$ and $\boldlambda$ is a coproduct with counit $\boldeps$, we have 
$$
{\overline \boldmu}^\vee = (-1)^{|\boldmu|} \overline{\boldmu^\vee},\qquad {\overline\boldeta}^\vee = (-1)^{|\boldeta|}\overline{\boldeta^\vee},
$$ 
$$
{\overline\boldlambda}^\vee = (-1)^{|\boldlambda|}\overline{\boldlambda^\vee},\qquad
{\overline \boldeps}^\vee=(-1)^{|\boldeps|}\overline{\boldeps^\vee}.
$$
Here we use that $s^\vee=\om$ and $\om^\vee=s$ under the canonical identification $A[1]^\vee=A^\vee[-1]$.

\section{Infinitesimal anti-symmetric bialgebras}\label{sec:inf-bialg}

\begin{definition} \label{defi:secondary-unital} 
A \emph{unital infinitesimal anti-symmetric bialgebra} 
is a graded Tate vector space $A$ endowed with a product $\boldmu:A\otimes^* A\to A$, a coproduct $\boldlambda:A\to A\otimes^! A$ and an element $\boldeta\in A$ which satisfy the following relations:
\begin{itemize}
\item {\sc (unit)} the element $\boldeta$ is the unit for the product $\boldmu$.
\item {\sc (associativity)} the product $\boldmu$ is associative. 
\item {\sc (coassociativity)} the coproduct $\boldlambda$ is coassociative. 
\item {\sc (unital infinitesimal relation)} 
\begin{align*}
\boldlambda\boldmu = (-1)^{|\boldlambda||\boldmu|} ((1\otimes\boldmu)(\boldlambda\otimes 1) & + (\boldmu\otimes 1)(1\otimes\boldlambda)) \\
& - (-1)^{|\boldmu|}(\boldmu\otimes \boldmu)(1\otimes \boldlambda\boldeta\otimes 1).
\end{align*}
\item {\sc (unital anti-symmetry)} 
\begin{align*}
&(-1)^{|\boldmu| (|\boldlambda|+1)} (1\otimes\boldmu) (\tau\boldlambda\otimes 1)  
+ (-1)^{|\boldlambda| (|\boldmu|+1)}(\boldmu\tau\otimes 1)(1\otimes\boldlambda) \\
& \hspace{5cm}- (-1)^{|\boldlambda|+|\boldmu|}(\boldmu\tau\otimes \boldmu)(1\otimes\boldlambda\boldeta\otimes 1) \\
&= (-1)^{|\boldlambda| |\boldmu|} \tau(1\otimes\boldmu\tau)(\boldlambda\otimes 1) 
- (-1)^{(|\boldlambda|+1)(|\boldmu|+1)} \tau (\boldmu\otimes 1)(1\otimes\tau\boldlambda) \\ 
& \hspace{5cm} - (-1)^{|\boldmu|}\tau(\boldmu\otimes \boldmu\tau)(1\otimes\boldlambda\boldeta\otimes 1).
\end{align*}
\end{itemize}
\end{definition}

The {\sc (unital infinitesimal)} and {\sc (unital anti-symmetry)} relations are to be understood as relations between maps $A\otimes^* A\to A\otimes^! A$. 
For example, the map $(\boldmu\otimes 1)(1\otimes\boldlambda)$ is the composition
$$
  A\otimes^* A \stackrel{1\otimes\boldlambda}\longrightarrow A\otimes^*(A\otimes^!A) \longrightarrow (A\otimes^*A)\otimes^!A \stackrel{\boldmu\otimes 1}\longrightarrow A\otimes^! A,
$$
where the middle map comes from Proposition~\ref{prop:comm-ass-graded}.

Evaluating the {\sc (unital anti-symmetry)} relation on $\boldeta\otimes \boldeta$ we obtain in particular
$$
\tau \boldlambda\boldeta = (-1)^{|\boldlambda|}\boldlambda\boldeta. 
$$

\begin{remark}\label{rmk:double} The product $\boldmu$ and the coproduct $\boldlambda$ define a product $\boldsymbol{m}$ on $A\oplus A^\vee$ such that $(A,\boldmu)$ and $(A^\vee,\boldlambda^\vee)$ embed as subalgebras into $(A\oplus A^\vee,\boldsymbol{m})$, much like in~\cite{CO-cones}.
Assuming that $\boldlambda\boldeta=0$, it is proved in~\cite[Appendix~A]{LO} that the above conditions are equivalent to associativity of the product $\boldsymbol{m}$ and the fact that the latter is unital with unit $(\boldeta,0)$. 
Zhelyabin~\cite{Zhelyabin1997} and Bai~\cite{Bai2010} prove this in the case where $|\boldmu|$ and $|\boldlambda|$ are even. In Zhelyabin's terminology, this is the notion of a graded associative D-bialgebra, or graded associative bialgebra in the sense of Drinfeld. The algebra $(A\oplus A^\vee,\boldsymbol{m})$ is called the \emph{Drinfeld double}.
\end{remark}

The notion of a unital infinitesimal anti-symmetric bialgebra is invariant under shift. 

\begin{remark}
The terminology {\sc (unital infinitesimal relation)} is motivated as follows: if $\boldlambda\boldeta=0$, then the resulting relation $\boldlambda\boldmu = (-1)^{|\boldlambda||\boldmu|}((1\otimes\boldmu)(\boldlambda\otimes 1)  + (\boldmu\otimes 1)(1\otimes\boldlambda))$ expresses the fact that the coproduct is a graded derivation for the product, hence the term ``infinitesimal relation'' coined by Joni-Rota~\cite{Joni-Rota}. See also \emph{loc.\,cit.} for an example involving formal derivatives of polynomials. The term ``unital'' expresses the fact that the ``infinitesimal relation'' is perturbed by a term involving the unit $\boldeta$, and is inspired by Loday-Ronco~\cite{Loday-Ronco} who study the special case $\boldlambda\boldeta=\boldeta\otimes\boldeta$. See also the bibliographical note below. 
\end{remark}

\begin{remark}
{\sc (unital anti-symmetry)} is equivalent to the antisymmetry
$$
   \tau\boldS\tau = -(-1)^{|\boldS|}\boldS
$$
of the degree $|\boldS|=|\boldmu|+|\boldlambda|$ operation $\boldS:A\otimes^* A\to A\otimes^! A$ defined by
$$
   \boldS = (\boldmu\otimes 1)(1\otimes\tau\boldlambda) - (-1)^{|\boldmu|}(1\otimes\boldmu)(\tau\boldlambda\otimes 1). 
$$
\end{remark}

\begin{remark}[Bibliographical note] \label{rmk:ncIIB-biblio} 
There are two main instances of the previous definition, given by the vanishing or non-vanishing of $\boldlambda\boldeta$. 
Some of them have previously appeared in the literature, though only ungraded and devoid of linear topologies.
 
Under the standing assumption $\boldlambda\boldeta=0$,  
this kind of structure was studied by Aguiar under the name ``balanced infinitesimal bialgebra"~\cite[Definition~4.1]{Aguiar2001}, following an early appearance in Zhelyabin's work~\cite{Zhelyabin1997}. In the absence of unital anti-symmetry it was studied by Aguiar under the name ``infinitesimal bialgebra"~\cite[Definition~2.1]{Aguiar}, following early appearances in the work of Joni-Rota~\cite{Joni-Rota} under the name of ``infinitesimal coalgebra", and in the work of Ehrenborg-Readdy~\cite{Ehrenborg-Readdy} under the name ``Newtonian coalgebra".
Without unital anti-symmetry this structure has also appeared more recently in~\cite{Merkulov-Vallette,Dotsenko-Shadrin-Vallette}, and with unital anti-symmetry in~\cite{Kaufmann07,Bai2010,Rivera-Wang}.

Under the standing assumption $\boldlambda\boldeta=\boldeta\otimes \boldeta$, and in the absence of unital anti-symmetry, this kind of structure was introduced under the name ``unital infinitesimal bialgebra" by Loday-Ronco, who proved a Cartier-Milnor-Moore rigidity theorem for such bialgebras~\cite[Theorem~2.6]{Loday-Ronco}. See also~\cite[\S5.2]{Livernet2006} and~\cite[\S5]{Foissy-Malvenuto-Patras}. 

We prove in~\cite{CHO-reducedSH} that reduced symplectic homology, which is defined for a large class of Weinstein domains including disc cotangent bundles, always carries the structure of a unital infinitesimal anti-symmetric bialgebra. In that case the topology is discrete and we stay strictly within the realm of linear algebra. 
\end{remark}

It is useful to give a graphical interpretation of the {\sc (unital infinitesimal relation)} and of {\sc (unital anti-symmetry)}. Let us represent $\boldmu$ and $\boldlambda$ in the form of {\sf Y}-shaped graphs, 
with the inputs depicted in clockwise order with respect to the output for $\boldmu$, and the outputs depicted in counterclockwise order with respect to the input for $\boldlambda$. See Figure~\ref{fig:mu-and-lambda}.
\begin{figure}
\begin{center}
\includegraphics[width=.7\textwidth]{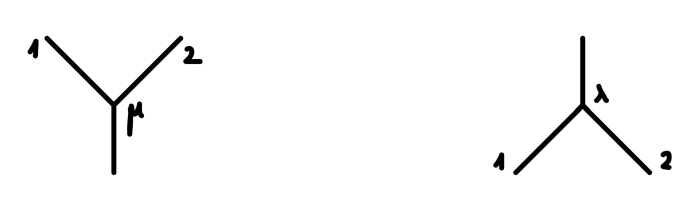}
\caption{The product $\boldmu$ and the coproduct $\boldlambda$.}
\label{fig:mu-and-lambda} 
\end{center}
\end{figure}
Then the {\sc (unital infinitesimal relation)} takes the form depicted in Figure~\ref{fig:infinitesimal-schematic}, and {\sc (unital anti-symmetry)} takes the form depicted in Figure~\ref{fig:4-term-new-schematic}.

\begin{figure}
\begin{center}
\includegraphics[width=\textwidth]{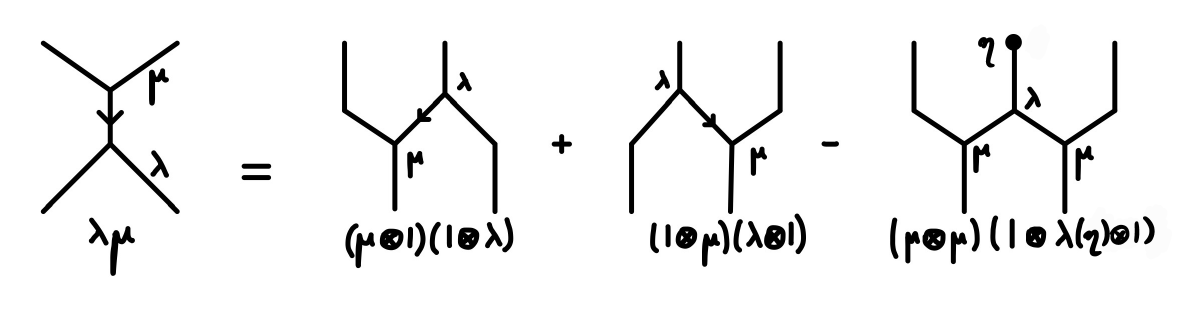}
\caption{The unital infinitesimal relation in the case $|\boldmu|=|\boldlambda|=0$.}
\label{fig:infinitesimal-schematic} 
\end{center}
\end{figure}

\begin{figure}
\begin{center}
\includegraphics[width=\textwidth]{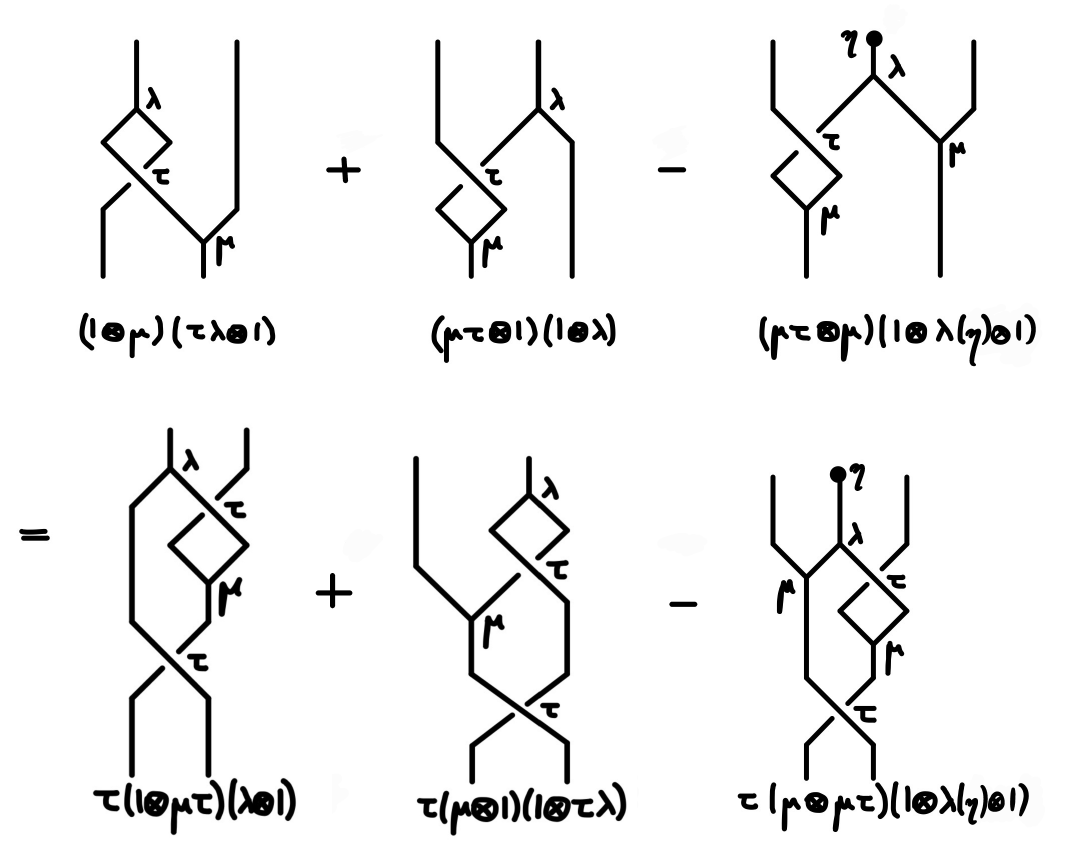}
\caption{The unital anti-symmetry relation in the case $|\boldmu|=|\boldlambda|=0$.}
\label{fig:4-term-new-schematic} 
\end{center}
\end{figure}

\begin{remark}[The commutative and cocommutative case]\label{rem:ccIIB}
If $\boldmu$ is commutative and $\boldlambda$ is cocommutative,
$$
   \boldmu\tau = (-1)^{|\boldmu|}\boldmu\quad\text{and}\quad
   \tau\boldlambda = (-1)^{|\boldlambda|}\boldlambda,
$$ 
then {\sc (unital anti-symmetry)} is a consequence of the {\sc (unital infinitesimal relation)}. To see this, simply observe that the unital infinitesimal relation transforms the left hand side of the unital anti-symmetry relation to $(-1)^{|\boldmu|+|\boldlambda|}\boldlambda\boldmu$ and the right hand side to $(-1)^{|\boldmu|}\tau\boldlambda\boldmu$, so the two sides are equal. 
\end{remark}

\begin{definition} \label{defi:secondary-counital} 
A \emph{counital infinitesimal anti-symmetric bialgebra} 
is a graded Tate vector space $A$ endowed with a product $\boldmu:A\otimes^* A\to A$, a coproduct $\boldlambda:A\to A\otimes^! A$ and an element $\boldeps\in A^\vee$ which satisfy the following relations:
\begin{itemize}
\item {\sc (counit)} the element $\boldeps$ is the counit for the coproduct $\boldlambda$.
\item {\sc (associativity)} the product $\boldmu$ is associative. 
\item {\sc (coassociativity)} the coproduct $\boldlambda$ is coassociative. 
\item {\sc (counital infinitesimal relation)} 
$$
\boldlambda\boldmu = (-1)^{|\boldlambda| |\boldmu|} ((1\otimes\boldmu)(\boldlambda\otimes 1) + (\boldmu\otimes 1)(1\otimes\boldlambda)) - (-1)^{|\boldlambda|}(1\otimes\boldeps\boldmu\otimes 1)(\boldlambda\otimes \boldlambda).
$$
\item {\sc (counital anti-symmetry relation)} 
\begin{align*}
&(-1)^{|\boldmu| (|\boldlambda|+1)} (1\otimes\boldmu) (\tau\boldlambda\otimes 1)  
+ (-1)^{|\boldlambda| (|\boldmu|+1)}(\boldmu\tau\otimes 1)(1\otimes\boldlambda) \\
& \hspace{5cm}- (-1)^{|\boldlambda|+|\boldmu|}(1\otimes\boldeps\boldmu\otimes 1)(\tau\boldlambda\otimes \boldlambda) \\
&= (-1)^{|\boldlambda| |\boldmu|} \tau(1\otimes\boldmu\tau)(\boldlambda\otimes 1) 
- (-1)^{(|\boldlambda|+1)(|\boldmu|+1)} \tau (\boldmu\otimes 1)(1\otimes\tau\boldlambda) \\ 
& \hspace{5cm} - (-1)^{|\boldlambda|}(1\otimes\boldeps\boldmu\otimes 1)(\boldlambda\otimes \tau\boldlambda)\tau.
\end{align*}
\end{itemize}
\end{definition}

Again, these last two relations are to be understood as relations between maps $A\otimes^* A\to A\otimes^! A$
using Proposition~\ref{prop:comm-ass-graded}.

Evaluating $\boldeps\otimes\boldeps$ on the {\sc (anti-symmetry)} relation we obtain in particular
$$
\boldeps\boldmu\tau = (-1)^{|\boldmu|}\boldeps\boldmu. 
$$

\begin{remark} \label{rmk:double-coalgebra} The product $\boldmu$ and the coproduct $\boldlambda$ define a coproduct $\boldsymbol{c}$ on $A\oplus A^\vee$ such that the projections to $(A,\boldlambda)$ and $(A^\vee,\boldmu^\vee)$ are coalgebra maps.
Assume that $\boldeps\boldmu=0$. The above conditions are equivalent to coassociativity of the coproduct $\boldsymbol{c}$ and the fact that it is counital with counit $\boldeps\circ \mathrm{proj}_1$. This is a consequence of Remark~\ref{rmk:double} together with the duality between the corresponding structures. We could rightfully call this ``graded coassociative D-bialgebra", or ``graded coassociative bialgebra in the sense of Drinfeld". The coalgebra $(A\oplus A^\vee,\boldsymbol{c})$ is called the \emph{Drinfeld double}.
\end{remark}

The notion of counital infinitesimal anti-symmetric bialgebra is invariant under shift. 

The notions of unital-, resp. counital infinitesimal anti-symmetric bialgebra are dual to each other: 

(i) If $(A,\boldmu,\boldlambda,\boldeta)$ is a unital infinitesimal anti-symmetric bialgebra, then its dual $(A^\vee,\boldlambda^\vee,\boldmu^\vee,\boldeta^\vee)$ is a counital infinitesimal anti-symmetric bialgebra.

(ii) If $(A,\boldmu,\boldlambda,\boldeps)$ is a counital infinitesimal anti-symmetric bialgebra, then its dual $(A^\vee,\boldlambda^\vee,\boldmu^\vee,\boldeps^\vee)$ is a unital infinitesimal anti-symmetric bialgebra. 

The verification is straightforward, using the identity
\begin{equation}\label{eq:move-tau}
  \tau(1\otimes\boldmu\tau)(\boldlambda\otimes 1)=(\boldmu\otimes 1)(1\otimes\tau\boldlambda)\tau. 
\end{equation}

\begin{definition} \label{defi:secondary-biunital} 
A \emph{biunital infinitesimal anti-symmetric bialgebra} 
is a graded Tate vector space $A$ endowed with a product $\boldmu:A\otimes^* A\to A$, a coproduct $\boldlambda:A\to A\otimes^! A$, and elements $\boldeta\in A$ and $\boldeps\in A^\vee$ which satisfy the following relations:
\begin{itemize}
\item {\sc (unit)}the element $\boldeta$ is the unit for the product $\boldmu$. 
\item {\sc (counit)} the element $\boldeps$ is the counit for the coproduct $\boldlambda$.
\item {\sc (associativity)} the product $\boldmu$ is associative. 
\item {\sc (coassociativity)} the coproduct $\boldlambda$ is coassociative. 
\item {\sc (biunital infinitesimal relation)} 
$$
\boldlambda\boldmu = (-1)^{|\boldlambda| |\boldmu|} ((1\otimes\boldmu)(\boldlambda\otimes 1) + (\boldmu\otimes 1)(1\otimes\boldlambda)) - (-1)^{|\boldlambda|}(1\otimes\boldeps\boldmu\otimes 1)(\boldlambda\otimes \boldlambda)
$$
and 
$$
(-1)^{|\boldlambda|}(1\otimes\boldeps\boldmu\otimes 1)(\boldlambda\otimes \boldlambda) = 
(-1)^{|\boldmu|}(\boldmu\otimes \boldmu)(1\otimes \boldlambda\boldeta\otimes 1).
$$
\item {\sc (biunital anti-symmetry relation)} 
\begin{align*}
&(-1)^{|\boldmu| (|\boldlambda|+1)} (1\otimes\boldmu) (\tau\boldlambda\otimes 1)  
+ (-1)^{|\boldlambda| (|\boldmu|+1)}(\boldmu\tau\otimes 1)(1\otimes\boldlambda) \\
& \hspace{5cm}- (-1)^{|\boldlambda|+|\boldmu|}(\boldmu\tau\otimes \boldmu)(1\otimes\boldlambda\boldeta\otimes 1) \\
&= (-1)^{|\boldlambda| |\boldmu|} \tau(1\otimes\boldmu\tau)(\boldlambda\otimes 1) 
- (-1)^{(|\boldlambda|+1)(|\boldmu|+1)} \tau (\boldmu\otimes 1)(1\otimes\tau\boldlambda) \\ 
& \hspace{5cm} - (-1)^{|\boldmu|}\tau(\boldmu\otimes \boldmu\tau)(1\otimes\boldlambda\boldeta\otimes 1).
\end{align*}
and 
$$
\left\{\begin{array}{rcl}
(\boldmu\tau\otimes \boldmu)(1\otimes\boldlambda\boldeta\otimes 1) & = & (1\otimes\boldeps\boldmu\otimes 1)(\tau\boldlambda\otimes\boldlambda),\\
(\boldmu\otimes \boldmu\tau)(1\otimes\boldlambda\boldeta\otimes 1) & = & (1\otimes\boldeps\boldmu\otimes 1)(\boldlambda\otimes\tau\boldlambda).
\end{array}\right.
$$
\end{itemize}
\end{definition}

The {\sc (biunital infinitesimal)} and {\sc (biunital anti-symmetry relation)} are again to be understood as relations between maps $A\otimes^* A\to A\otimes^! A$. 

This algebraic structure is invariant under shifts and is self-dual.
The additional conditions in the {\sc (infinitesimal)} and {\sc (anti-symmetry)} relations ensure that both the unital and the counital versions are satisfied.

\section{CoFrobenius bialgebras}\label{sec:coFrob}

\begin{definition} \label{defi:coFrobenius-unital-bialgebra} 
A \emph{unital coFrobenius bialgebra}
is a graded Tate vector space $A$ endowed with a product $\boldmu:A\otimes^* A\to A$, a coproduct $\boldlambda:A\to A\otimes^! A$, and an element $\boldeta\in A$ which satisfy the following relations:
\begin{itemize}
\item {\sc (unit)} the element $\boldeta$ is the unit for the product $\boldmu$.
\item {\sc (associativity)} the product $\boldmu$ is associative. 
\item {\sc (coassociativity)} the coproduct $\boldlambda$ is coassociative. 

\medskip 

\noindent Moreover, defining the \emph{copairing} by 
$$
\boldc = (-1)^{|\boldlambda||\boldmu| + |\boldmu|}\boldlambda\boldeta \in A\otimes^! A,
$$
we have: 

\item {\sc (unital coFrobenius)} 
$$
\boldlambda=(1\otimes \boldmu)(\boldc\otimes 1) = (-1)^{|\boldmu|}(\boldmu\otimes 1)(1\otimes \boldc).
$$
\item {\sc(symmetry)}
$$
\tau \boldc = (-1)^{|\boldlambda|} \boldc. 
$$
\end{itemize}
\end{definition}

The {\sc (coassociativity)} relation is actually implied by the {\sc (unital coFrobenius)} relation.

\begin{proposition}\label{prop:unital-coFrob-bialg}
Let $(A,\boldmu,\boldlambda,\boldeta)$ be a unital coFrobenius bialgebra. The following identities hold:
\begin{gather*}
  (1\otimes\boldmu\otimes 1)(\boldc\otimes\boldc) = (\boldlambda\otimes 1)\boldc
  = (-1)^{|\boldlambda|}(1\otimes\boldlambda)\boldc, \cr
  (1\otimes\boldmu)(\boldlambda\otimes 1) = (\boldmu\otimes 1)(1\otimes\boldlambda)
  = (\boldmu\otimes\boldmu)(1\otimes\boldc\otimes 1)
  = (-1)^{|\boldlambda||\boldmu|}\boldlambda\boldmu.
\end{gather*}
In particular, $(A,\boldmu,\boldlambda,\boldeta)$ is a unital infinitesimal anti-symmetric bialgebra. 
\end{proposition}

The first identity is to be understood as an identity in $A^{\otimes^! 3}$, and the second one as an identity of maps $A\otimes^* A\to A\otimes^! A$. 

\begin{proof}
Coassociativity of $\lambda$ and the other identities in the proposition follow by direct computation. By the second line of equalities and the definition of $\boldc$, all the terms in the {\sc (unital infinitesimal relation)} are equal and cancel pairwise. 
For {\sc (unital anti-symmetry)} one first checks the identities
$$
\tau\boldlambda = (-1)^{|\boldlambda|}(\boldmu\tau\otimes 1)(1\otimes\boldc)
= (-1)^{|\boldmu|+|\boldlambda|}(1\otimes\boldmu\tau)(\boldc\otimes 1). 
$$
These imply that the three terms on the left hand side of the {\sc (unital anti-symmetry)} relation are equal, as are the three terms on the right hand side. To show that the terms on the left hand side equal those on the right hand side, we first use associativity of $\boldmu$ to verify the identity
$$
  \boldmu(1\otimes\boldmu\tau)(\tau\otimes 1) = (-1)^{|\boldmu|}\boldmu\tau(1\otimes\boldmu).
$$
Using this and the {\sc (unital coFrobenius)} relation we compute
\begin{align*}
  (\boldmu\otimes 1) & (1\otimes\tau\boldlambda)
  = (-1)^{|\boldlambda|}(\boldmu\otimes 1)\Bigl(1\otimes[(\boldmu\tau\otimes 1)(1\otimes\boldc)]\Bigr) \cr
  &= (-1)^{|\boldlambda|}(\boldmu\otimes 1)(1\otimes\boldmu\tau\otimes 1)(1\otimes 1\otimes\boldc) \cr
  &= (-1)^{|\boldlambda|}\Bigl([\boldmu(1\otimes\boldmu\tau)]\otimes 1\Bigr)(1\otimes 1\otimes\boldc) \cr
  &= (-1)^{|\boldmu|+|\boldlambda|}\Bigl([\boldmu\tau(1\otimes\boldmu)(\tau\otimes 1)]\otimes 1\Bigr)(1\otimes 1\otimes\boldc) \cr
  &= (-1)^{|\boldmu|+|\boldlambda|}(\boldmu\tau\otimes 1)(1\otimes\boldmu\otimes 1)(\tau\otimes 1\otimes 1)(1\otimes 1\otimes\boldc) \cr
  &= (-1)^{|\boldmu|+|\boldlambda|}(\boldmu\tau\otimes 1)(1\otimes\boldmu\otimes 1)(1\otimes 1\otimes\boldc)\tau \cr
  &= (-1)^{|\boldlambda|}(\boldmu\tau\otimes 1)(1\otimes\boldlambda)\tau \cr
  &= (-1)^{|\boldlambda|}\tau(1\otimes\boldmu)(\tau\boldlambda\otimes 1),
\end{align*}
where the last equality follows from~\eqref{eq:move-tau}. This proves that the first term on the left hand side of the {\sc (unital anti-symmetry)} relation equals the second term on the right hand side, and thus concludes the proof of {\sc (unital anti-symmetry)}.
\end{proof}

The notion of a unital coFrobenius bialgebra is invariant under shifts. For the proof we need to determine the correct shift for the copairing $\boldc$. Given the shifted product $\ol\boldmu = s\boldmu (\omega\otimes\omega)$ and coproduct $\ol\boldlambda=(s\otimes s)\boldlambda \omega$, we recall that the unit for $\ol\boldmu$ is $\ol\boldeta=(-1)^{|\boldmu|}s\boldeta$. With this one computes 
$$
\ol\boldc = (-1)^{|\ol\boldlambda||\ol\boldmu|+|\ol\boldmu|}\ol\boldlambda\ol\boldeta = (-1)^{|\boldlambda|}(s\otimes s)\boldc.
$$ 
The invariance of the structure under shifts now follows by a straightforward computation.

\begin{definition} \label{defi:coFrobenius-counital-bialgebra} 
A \emph{counital coFrobenius bialgebra}
is a graded Tate vector space $A$ endowed with a product $\boldmu:A\otimes^* A\to A$, a coproduct $\boldlambda:A\to A\otimes^! A$, and an element $\boldeps\in A^\vee$ which satisfy the following relations:
\begin{itemize}
\item {\sc (counit)} the element $\boldeps$ is the counit for the coproduct $\boldlambda$.
\item {\sc (coassociativity)} the coproduct $\boldlambda$ is coassociative. 
\item {\sc (associativity)} the product $\boldmu$ is associative. 

\medskip 

\noindent Moreover, defining the \emph{pairing} by 
$$
\boldp = (-1)^{|\boldlambda|}\boldeps\boldmu:A\otimes^* A\to \bk,
$$
we have: 

\item {\sc (counital coFrobenius)} 
$$
\boldmu = (-1)^{|\boldmu|  |\boldlambda| + |\boldlambda|}(\boldp\otimes 1)(1\otimes \boldlambda) = (-1)^{|\boldmu|  |\boldlambda|}(1\otimes \boldp)(\boldlambda\otimes 1).
$$
\item {\sc (symmetry)}
$$
\boldp \tau  = (-1)^{|\boldmu|} \boldp. 
$$
\end{itemize}
\end{definition}

The {\sc (associativity)} relation is actually implied by the {\sc (counital coFrobenius)} relation.

Just like its unital counterpart, the notion of a counital coFrobenius bialgebra is invariant under shifts. Again, for the proof we need to determine the correct shift for the pairing $\boldp$. Given the shifted product $\ol\boldmu = s\boldmu (\omega\otimes\omega)$ and coproduct $\ol\boldlambda=(s\otimes s)\boldlambda \omega$, the counit for $\ol\boldlambda$ is $\ol\boldeps=(-1)^{|\boldlambda|}\boldeps\omega$. With this one computes 
$$
\ol\boldp = (-1)^{|\ol\boldlambda|}\ol\boldeps\, \ol\boldmu = (-1)^{|\boldlambda|+1}\boldp (\omega\otimes\omega).
$$ 
The invariance of the structure under shifts now follows by a straightforward computation. 

The notions of unital/counital coFrobenius bialgebras are dual to each other.
Dualizing Proposition~\ref{prop:unital-coFrob-bialg} yields

\begin{proposition}\label{prop:counital-coFrob-bialg}
Let $(A,\boldmu,\boldlambda,\boldeps)$ be a counital coFrobenius bialgebra. The following identities hold:
$$ 
(-1)^{|\boldp| |\boldmu|} (\boldp\otimes\boldp)(1\otimes\boldlambda\otimes 1) = \boldp(\boldmu\otimes 1) = 
 (-1)^{|\boldmu|} \boldp(1\otimes\boldmu), 
 $$
\begin{align*} 
(\boldmu\otimes 1)(1\otimes\boldlambda) 
  & = (1\otimes\boldmu) (\boldlambda\otimes 1) \\
  & = (-1)^{|\boldlambda||\boldmu|} (1\otimes\boldp\otimes 1)(\boldlambda\otimes\boldlambda)
  = (-1)^{|\boldlambda||\boldmu|}\boldlambda\boldmu.
\end{align*}
In particular, $(A,\boldmu,\boldlambda,\boldeps)$ is a counital infinitesimal anti-symmetric bialgebra.
\qed
\end{proposition}

The first identity is to be understood as an identity of maps $A^{\otimes^* 3}\to \bk$, and the second one as an identity of maps $A\otimes^* A\to A\otimes^! A$.

\begin{definition} \label{defi:biunital-coFrobenius-bialgebra} 
A \emph{biunital coFrobenius bialgebra}  
is a graded Tate vector space $A$ endowed with a product $\boldmu:A\otimes^* A\to A$, a coproduct $\boldlambda:A\to A\otimes^! A$, and elements $\boldeta\in A$, $\boldeps\in A^\vee$ which satisfy the following relations:
\begin{itemize}
\item {\sc (unit)} the element $\boldeta$ is the unit for the product $\boldmu$.
\item {\sc (counit)} the element $\boldeps$ is the counit for the coproduct $\boldlambda$. 
\item {\sc (associativity)} the product $\boldmu$ is associative. 
\item {\sc (coassociativity)} the coproduct $\boldlambda$ is coassociative. 

\medskip 

\noindent Moreover, defining the \emph{copairing} by 
$$
\boldc = (-1)^{|\boldlambda||\boldmu| + |\boldmu|}\boldlambda\boldeta\in A\otimes^! A,
$$
and the \emph{pairing} by 
$$
\boldp = (-1)^{|\boldlambda|}\boldeps\boldmu:A\otimes^* A\to\bk,
$$
we have: 

\item {\sc (biunital coFrobenius)} 
$$
\boldlambda=(1\otimes \boldmu)(\boldc\otimes 1) = (-1)^{|\boldmu|}(\boldmu\otimes 1)(1\otimes \boldc)
$$
and
$$
\boldmu = (-1)^{|\boldmu|  |\boldlambda| + |\boldlambda|}(\boldp\otimes 1)(1\otimes \boldlambda) = (-1)^{|\boldmu|  |\boldlambda|}(1\otimes \boldp)(\boldlambda\otimes 1).
$$
\item {\sc (symmetry)}
$$
\tau \boldc = (-1)^{|\boldlambda|} \boldc
$$
and
$$
\boldp \tau  = (-1)^{|\boldmu|} \boldp. 
$$
\end{itemize}
\end{definition}

\begin{proposition}\label{prop:coFrob-criterion}
\quad 

(i) Let $(A,\boldmu,\boldlambda,\boldeta)$ be a unital coFrobenius bialgebra such that the coproduct $\boldlambda$ is counital with counit $\boldeps$. Then $(A,\boldmu,\boldlambda,\boldeta,\boldeps)$ is a biunital coFrobenius bialgebra.  

(ii) Let $(A,\boldmu,\boldlambda,\boldeps)$ be a counital coFrobenius bialgebra such that the product $\boldmu$ is unital with unit $\boldeta\in A$. Then $(A,\boldmu,\boldlambda,\boldeta,\boldeps)$ is a biunital coFrobenius bialgebra. 
\end{proposition} 

\begin{proof}
For (i) we need to prove the {\sc (counital coFrobenius)} relation.
Using $(\boldmu\otimes 1)(1\otimes\boldlambda)=(-1)^{|\boldmu||\boldlambda|}\boldlambda\boldmu$ from Proposition~\ref{prop:unital-coFrob-bialg}, we compute
\begin{align*}
  (\boldp\otimes 1)(1\otimes\boldlambda)
  &= (-1)^{|\boldlambda|}(\boldeps\boldmu\otimes 1)(1\otimes\boldlambda)
  = (-1)^{|\boldmu||\boldlambda|+|\boldlambda|}(\boldeps\otimes 1)\boldlambda\boldmu \cr
  &= (-1)^{|\boldmu||\boldlambda|+|\boldlambda|}\boldmu.
\end{align*}
This proves the first equality in the {\sc (counital coFrobenius)} relation, and the second one follows similarly.
The proof of (ii) is analogous, using Proposition~\ref{prop:counital-coFrob-bialg}.
\end{proof}

\begin{proposition}\label{prop:biunital-coFrob-bialg}
Let $(A,\boldmu,\boldlambda,\boldeta,\boldeps)$ be a biunital coFrobenius bialgebra.
Then in addition to the ones in Propositions~\ref{prop:unital-coFrob-bialg} and~\ref{prop:counital-coFrob-bialg} the following identities hold:
\begin{equation*}
\boldp(\boldmu\otimes 1) = (-1)^{|\boldmu|}\boldp(1\otimes\boldmu),\qquad
(\boldlambda\otimes 1)\boldc = (-1)^{|\boldlambda|}(1\otimes\boldlambda)\boldc,\qquad
\end{equation*}
\begin{align*}
  (-1)^{|\boldlambda|}\boldeps &= \boldp(1\otimes\boldeta) = (-1)^{|\boldmu|}\boldp(\boldeta\otimes 1), \cr
  (-1)^{|\boldlambda||\boldmu|+|\boldmu|}\boldeta &= (\boldeps\otimes 1)\boldc = (-1)^{|\boldlambda|}(1\otimes\boldeps)\boldc.
\end{align*}
\begin{equation*}
  (1\otimes\boldp)(\boldc\otimes 1)
  = (-1)^{|\boldlambda|+|\boldmu|}(\boldp\otimes 1)(1\otimes\boldc)=1.
\end{equation*}
In particular, $(A,\boldmu,\boldlambda,\boldeta,\boldeps)$ is a biunital infinitesimal anti-symmetric bialgebra. 
\qed
\end{proposition}

A referee has asked the following 

{\bf Question.} {\it Is it true that every biunital infinitesimal anti-symmetric bialgebra is also a biunital coFrobenius bialgebra?}

We are able to give a positive answer, under an additional assumption.

\begin{proposition} \label{prop:biu-inf-as_to_biu-coFrob-bialg}
Let $(A,\boldmu,\boldlambda,\boldeta,\boldeps)$ be a biunital infinitesimal anti-symmetric bialgebra. Let 
$$
\boldc= (-1)^{|\boldlambda||\boldmu| + |\boldmu|}\boldlambda\boldeta\in A\otimes^! A.
$$
and 
$$
\boldp= (-1)^{|\boldlambda|}\boldeps\boldmu:A\otimes^* A\to\bk.
$$

(a) The {\sc (symmetry)} relations hold: 
$$
\tau \boldc = (-1)^{|\boldlambda|} \boldc
$$
and
$$
\boldp \tau  = (-1)^{|\boldmu|} \boldp. 
$$

(b) Assume that 
$$
(1\otimes\boldmu)(\boldc\otimes 1)=(-1)^{|\boldmu|}(\boldmu\otimes 1)(1\otimes\boldc),
$$ 
or, equivalently, 
$$
(-1)^{|\boldlambda|}(\boldp\otimes 1)(1\otimes\boldlambda)=(1\otimes\boldp)(\boldlambda\otimes 1).
$$ 
Then $(A,\boldmu,\boldlambda,\boldeta,\boldeps)$ is a biunital coFrobenius bialgebra. 
\end{proposition}

\begin{proof}
(a) The {\sc (symmetry)} relations are consequences of the {\sc (unital anti-symmetry)} relation. Indeed, the equality $\tau \boldc = (-1)^{|\boldlambda|} \boldc$ follows by applying the {\sc (unital anti-symmetry)} relation to $\boldeta\otimes \boldeta$, and the equality $\boldp \tau  = (-1)^{|\boldmu|} \boldp$ follows by applying to the {\sc (unital anti-symmetry)} relation the map $\boldeps\otimes\boldeps$. 

Since, in the proof that follows, we do not give the details of all the computations, we will illustrate them using the following graphical language (inputs are at the top, outputs are at the bottom).

\begin{tikzpicture}[scale=.3]
\pgftransformyshift{-.7cm}
\draw (0,0) .. controls (0,2) and (3,2) .. (3,0); \node at (2,2) {\tiny $\boldc$};
\pgftransformxshift{6.5cm}
\pgftransformyshift{1.4cm}
\draw (0,0) .. controls (0,-2) and (3,-2) .. (3,0); \node at (2,-2) {\tiny $\boldp$};
\pgftransformxshift{6.5cm}
\pgftransformyshift{-.7cm}
\draw (1.5,0) -- (1.5,-1.5) -- (1.5,0) -- (0,1.5) -- (1.5,0) -- (3,1.5); \node at (2,-0.5) {\tiny $\boldmu$}; 
\pgftransformxshift{6.5cm}
\draw (1.5,0) -- (1.5,1.5) -- (1.5,0) -- (0,-1.5) -- (1.5,0) -- (3,-1.5); \node at (2,0.5) {\tiny $\boldlambda$}; 
\pgftransformxshift{6.5cm}
\draw (1.5,0) -- (1.5,-1.5); \node[circle,draw,fill ,inner sep=1pt] at (1.5,0) {}; \node at (1.5,0.7) {\tiny $\boldeta$}; 
\pgftransformxshift{6.5cm}
\draw (1.5,0) -- (1.5,1.5); \node[circle,draw,fill ,inner sep=1pt] at (1.5,0) {}; \node at (1.5,-0.7) {\tiny $\boldeps$}; 
\pgftransformxshift{6.5cm}
\draw (1.5,1.5) -- (1.5,-1.5); \node at (2,0) {\tiny $1$};
\end{tikzpicture}

The properties of the unit $\boldeta$ and of the counit $\boldeps$ imply
\begin{equation*} 
\vec \boldc(\boldeps)=\boldeta,\qquad \vec\boldp(\boldeta)=\boldeps.
\end{equation*}
Graphically \ 
\begin{tikzpicture}[scale=.2]
\draw (0,0) .. controls (0,2) and (3,2) .. (3,0); \node[circle,draw,fill ,inner sep=1pt] at (0,0) {}; 
\pgftransformxshift{4.5cm}
\pgftransformyshift{.5cm}
\node at (0,0) {\tiny $=$};
\pgftransformxshift{0cm}
\pgftransformyshift{.7cm}
\draw (1.5,0) -- (1.5,-1.5); \node[circle,draw,fill ,inner sep=1pt] at (1.5,0) {}; 
\end{tikzpicture}
\ and \   
\begin{tikzpicture}[scale=.2]
\draw (0,0) .. controls (0,-2) and (3,-2) .. (3,0); \node[circle,draw,fill ,inner sep=1pt] at (0,0) {}; 
\pgftransformxshift{4.5cm}
\pgftransformyshift{-1cm}
\node at (0,0) {\tiny $=$};
\pgftransformxshift{0cm}
\pgftransformyshift{-.5cm}
\draw (1.5,0) -- (1.5,1.5); \node[circle,draw,fill ,inner sep=1pt] at (1.5,0) {}; 
\end{tikzpicture}
\ . The first equality follows from the definition of $\boldc$ and $\vec\boldc$ by rewriting the equality $(\boldeps\otimes 1)\boldlambda\boldeta=\boldeta$, and the second equality follows from the definition of $\boldp$ and $\vec\boldp$ by rewriting the equality $\boldeps\boldmu(1\otimes\boldeta)=\boldeps$. In particular $\vec\boldp$ and $\vec\boldc$ are nonzero if $A$ is not the zero algebra. 

(b) Assume $(1\otimes\boldmu)(\boldc\otimes 1)=(-1)^{|\boldmu|}(\boldmu\otimes 1)(1\otimes\boldc)$. Graphically \ 
\begin{tikzpicture}[scale=.1]
\draw (0,0) .. controls (0,2) and (3,2) .. (3,0); 
\pgftransformxshift{3cm}
\pgftransformyshift{-1.7cm}
\draw (1.5,0) -- (1.5,-1.5) -- (1.5,0) -- (0,1.5) -- (1.5,0) -- (3,1.5); 
\pgftransformxshift{4cm}
\node at (0,0) {\tiny $=$};
\pgftransformxshift{1cm}
\draw (1.5,0) -- (1.5,-1.5) -- (1.5,0) -- (0,1.5) -- (1.5,0) -- (3,1.5); 
\pgftransformxshift{3cm}
\pgftransformyshift{1.7cm}
\draw (0,0) .. controls (0,2) and (3,2) .. (3,0); 
\end{tikzpicture}
\ . By applying $1\otimes\boldeps$ to both sides of this equality and using the previous relations on $\boldc$ we find $(1\otimes\boldp)(\boldc\otimes 1)=1$. 
Graphically \ 
\begin{tikzpicture}[scale=.1]
\draw (0,0) .. controls (0,2) and (3,2) .. (3,0); 
\pgftransformxshift{3cm}
\pgftransformyshift{-.2cm}
\draw (0,0) .. controls (0,-2) and (3,-2) .. (3,0); 
\pgftransformxshift{5cm}
\pgftransformyshift{.2cm}
\node at (0,0) {\tiny $=$};
\pgftransformxshift{1cm}
\draw (1.5,1.5) -- (1.5,-1.5); 
\end{tikzpicture}
\ . By applying $\boldeps\otimes 1$ to both sides of this equality we find $1=(-1)^{|\boldlambda|+|\boldmu|}(\boldp\otimes 1)(1\otimes\boldc)$. 
Graphically \ 
\begin{tikzpicture}[scale=.1]
\draw (0,0) .. controls (0,-2) and (3,-2) .. (3,0); 
\pgftransformxshift{3cm}
\pgftransformyshift{.2cm}
\draw (0,0) .. controls (0,2) and (3,2) .. (3,0); 
\pgftransformxshift{5cm}
\pgftransformyshift{-.2cm}
\node at (0,0) {\tiny $=$};
\pgftransformxshift{1cm}
\draw (1.5,1.5) -- (1.5,-1.5); 
\end{tikzpicture}
\ . Therefore, $\vec\boldc$ and $\vec\boldp$ are topological isomorphisms inverse to each other. 

We now consider the second {\sc (biunital infinitesimal relation)}
\begin{equation} \label{eq:relation-for-biunital-ias}
(-1)^{|\boldmu|}(\boldmu\otimes \boldmu)(1\otimes \boldlambda\boldeta\otimes 1)=(-1)^{|\boldlambda|}(1\otimes\boldeps\boldmu\otimes 1)(\boldlambda\otimes \boldlambda).
\end{equation}
Graphically \ 
\begin{tikzpicture}[scale=.1]
\draw (1.5,0) -- (1.5,-1.5) -- (1.5,0) -- (0,1.5) -- (1.5,0) -- (3,1.5); 
\pgftransformxshift{3cm}
\pgftransformyshift{1.7cm}
\draw (0,0) .. controls (0,2) and (3,2) .. (3,0); 
\pgftransformxshift{3cm}
\pgftransformyshift{-1.7cm}
\draw (1.5,0) -- (1.5,-1.5) -- (1.5,0) -- (0,1.5) -- (1.5,0) -- (3,1.5); 
\pgftransformxshift{5cm}
\pgftransformyshift{0cm}
\node at (0,0) {\tiny $=$};
\pgftransformxshift{2cm}
\pgftransformyshift{1.7cm}
\draw (1.5,0) -- (1.5,1.5) -- (1.5,0) -- (0,-1.5) -- (1.5,0) -- (3,-1.5); 
\pgftransformxshift{3cm}
\pgftransformyshift{-1.7cm}
\draw (0,0) .. controls (0,-2) and (3,-2) .. (3,0); 
\pgftransformxshift{3cm}
\pgftransformyshift{1.7cm}
\draw (1.5,0) -- (1.5,1.5) -- (1.5,0) -- (0,-1.5) -- (1.5,0) -- (3,-1.5); 
\end{tikzpicture}
\ . By applying to this equality $\boldeps\otimes 1$ we find 
$$
(-1)^{|\boldlambda||\boldmu|+|\boldlambda|}(\boldp\otimes 1)\big(1\otimes \big[(1\otimes\boldmu)(\boldc\otimes 1)\big]\big)=(-1)^{|\boldlambda|+|\boldlambda||\boldmu|}(\boldp\otimes 1)(1\otimes\boldlambda).
$$ 
Graphically \ 
\begin{tikzpicture}[scale=.1]
\draw (0,0) .. controls (0,-2) and (3,-2) .. (3,0); 
\pgftransformxshift{3cm}
\pgftransformyshift{.2cm}
\draw (0,0) .. controls (0,2) and (3,2) .. (3,0); 
\pgftransformxshift{3cm}
\pgftransformyshift{-1.7cm}
\draw (1.5,0) -- (1.5,-1.5) -- (1.5,0) -- (0,1.5) -- (1.5,0) -- (3,1.5); 
\pgftransformxshift{5cm}
\node at (0,0) {\tiny $=$};
\pgftransformxshift{2cm}
\draw (0,0) .. controls (0,-2) and (3,-2) .. (3,0); 
\pgftransformxshift{3cm}
\pgftransformyshift{1.7cm}
\draw (1.5,0) -- (1.5,1.5) -- (1.5,0) -- (0,-1.5) -- (1.5,0) -- (3,-1.5); 
\end{tikzpicture}
\ . Since $\vec\boldp$ is invertible, this implies $(1\otimes\boldmu)(\boldc\otimes 1)=\boldlambda$. 
Graphically \ 
\begin{tikzpicture}[scale=.1]
\draw (0,0) .. controls (0,2) and (3,2) .. (3,0); 
\pgftransformxshift{3cm}
\pgftransformyshift{-1.7cm}
\draw (1.5,0) -- (1.5,-1.5) -- (1.5,0) -- (0,1.5) -- (1.5,0) -- (3,1.5); 
\pgftransformxshift{5cm}
\node at (0,0) {\tiny $=$};
\pgftransformxshift{2cm}
\pgftransformyshift{.5cm}
\draw (1.5,0) -- (1.5,1.5) -- (1.5,0) -- (0,-1.5) -- (1.5,0) -- (3,-1.5); 
\end{tikzpicture}
\ . This is the first unital coFrobenius relation. The second unital coFrobenius relation follows by applying $1\otimes\boldeps$ to~\eqref{eq:relation-for-biunital-ias}. The counital coFrobenius relations follow in a similar way by applying~\eqref{eq:relation-for-biunital-ias} to $\boldeta\otimes 1$ or $1\otimes\boldeta$, or, alternatively, using Proposition~\eqref{prop:coFrob-criterion}. 

The proof of the coFrobenius relations assuming $(-1)^{|\boldlambda|}(\boldp\otimes 1)(1\otimes\boldlambda)=(1\otimes\boldp)(\boldlambda\otimes 1)$,  
or graphically \ 
\begin{tikzpicture}[scale=.1]
\draw (0,0) .. controls (0,-2) and (3,-2) .. (3,0); 
\pgftransformxshift{3cm}
\pgftransformyshift{1.7cm}
\draw (1.5,0) -- (1.5,1.5) -- (1.5,0) -- (0,-1.5) -- (1.5,0) -- (3,-1.5); 
\pgftransformxshift{5cm}
\pgftransformyshift{-1.7cm}
\node at (0,0) {\tiny $=$};
\pgftransformxshift{2cm}
\pgftransformyshift{1.7cm}
\draw (1.5,0) -- (1.5,1.5) -- (1.5,0) -- (0,-1.5) -- (1.5,0) -- (3,-1.5); 
\pgftransformxshift{3cm}
\pgftransformyshift{-1.7cm}
\draw (0,0) .. controls (0,-2) and (3,-2) .. (3,0); 
\end{tikzpicture}
\ , is similar: one first proves that $\vec\boldp$ and $\vec\boldc$ are topological isomorphisms inverse to each other by applying this relation to $\boldeta\otimes 1$ and $1\otimes\boldeta$. Then one deduces the coFrobenius relations from~\eqref{eq:relation-for-biunital-ias} as above. 
\end{proof}

It follows from the preceding discussion in this section that the notion of a biunital coFrobenius bialgebra is self-dual
and invariant under shifts. This is summarized in the following lemma, where part (iii) is immediate from the definitions. 

\begin{lemma}\label{lem:biunital-coFrob-invariance}
Let $(A,\boldmu,\boldlambda,\boldeta,\boldeps)$ be a biunital coFrobenius bialgebra with copairing $\boldc$ and pairing $\boldp$. Then:

(i) $(A^\vee,\boldlambda^\vee,\boldmu^\vee,\boldeps^\vee,\boldeta^\vee)$ is a biunital coFrobenius bialgebra with copairing $\boldp^\vee$ and pairing $\boldc^\vee$.

(ii) $(A[1],\ol\boldmu,\ol\boldlambda,\ol\boldeta,\ol\boldeps)$ is a biunital coFrobenius bialgebra with copairing $\ol\boldc$ and pairing $\ol\boldp$, where
\begin{alignat*}{3}
  &\ol\boldmu = s\boldmu (\omega\otimes\omega),\qquad 
  &&\ol\boldlambda = (s\otimes s)\boldlambda \omega,\cr
  &\ol\boldeta = (-1)^{|\boldmu|}s\boldeta,\qquad 
  &&\ol\boldeps = (-1)^{|\boldlambda|}\boldeps\omega,\cr
  &\ol\boldc = (-1)^{|\boldlambda|}(s\otimes s)\boldc,\qquad
  &&\ol\boldp = (-1)^{|\boldlambda|+1}\boldp (\omega\otimes\omega).
\end{alignat*}
(iii) For any $m,\ell\in\Z$, the tuple $(A,(-1)^m\boldmu,(-1)^\ell\boldlambda,(-1)^m\boldeta,\break (-1)^\ell\boldeps)$ is a biunital coFrobenius bialgebra with copairing $(-1)^{m+\ell}\boldc$ and pairing $(-1)^{m+\ell}\boldp$.
\qed
\end{lemma}

\begin{remark}[on terminology]
The term ``coFrobenius" refers to the fact that, in the unital case, the structure involves the product $\boldmu$ and the copairing $\boldc:\bk\to A\otimes^! A$, whereas the definition of Frobenius algebras classically involves a product and a pairing $A\otimes^* A\to \bk$. Counital coFrobenius bialgebras are the dual notion, and involve a coproduct $\boldlambda$ and a pairing $\boldp$. 

The term ``coFrobenius"  was introduced previously by Lin~\cite{I-peng-lin} in an ungraded setting involving a coproduct $\boldlambda$ and a pairing $\boldp$ assumed to be nondegenerate. See also Iovanov~\cite{Iovanov}. 

Our terminology ``bialgebras" instead of ``algebras" or ``coalgebras" emphasizes the role played by both product and coproduct. It also echoes the bialgebra terminology from~\S\ref{sec:inf-bialg}.

In the case that $|\boldmu|$ and $|\boldlambda|$ are both even, the notion of a biunital coFrobenius bialgebra agrees with that of a symmetric Frobenius algebra in the literature (see e.g.~\cite{Abrams,Kock} for background on symmetric Frobenius algebras and their equivalence to $1+1$ dimensional TQFTs).
\end{remark} 

\begin{remark}[Direct sums]\label{rem:direct-sum}
All the algebraic structures in this and the previous section admit obvious direct sum constructions. E.g. the direct sum of two biunital coFrobenius bialgebras $(A_i,\boldmu_i,\boldlambda_i,\boldeta_i,\boldeps_i)$, $i=1,2$ is $A=A_1\oplus A_2$ with the operations defined for $a_i,b_i\in A_i$ by
\begin{gather*}
  \boldmu\bigl((a_1,a_2)\otimes(b_1,b_2)\bigr) = \bigl(\boldmu_1(a_1\otimes b_1),\boldmu_2(a_2\otimes b_2)\bigr), \cr
  \boldlambda(a_1,a_2) = \boldlambda_1(a_1) + \boldlambda_2(a_2) \in A_1\otimes A_1 + A_2\otimes A_2\subset A\otimes A, \cr
  \boldeta = (\boldeta_1,\boldeta_2), \qquad
  \boldeps(a_1,a_2) = \boldeps_1(a_1) + \boldeps_2(a_2).
\end{gather*}
\end{remark}

\begin{remark}[Involutivity]\label{rem:involutivity}
All the algebraic structures considered in this paper involve a product $\boldmu:A\otimes^* A\to A$ and a coproduct $\boldlambda:A\to A\otimes^! A$. Let us call such a structure {\em involutive} if the canonical map $A\otimes^* A\to A\otimes^! A$ is a (topological) isomorphism and
$$
  \boldmu\boldlambda=0.
$$
(i) The map $A\otimes^* A\to A\otimes^! A$ is an isomorphism if $A$ is linearly compact or discrete (Proposition~\ref{prop:tensor_discrete}). 

(ii) If $2\neq 0$ in the field $\bk$ and the map $A\otimes^* A\to A\otimes^! A$ is an isomorphism, then involutivity holds if $\boldmu$ and $\boldlambda$ are commutative resp.~cocommutative of opposite parity. This applies for example to the structures on symplectic homology and symplectic cohomology in the closed string case (resp.~loop homology and loop cohomology for free loops). In the open string (resp.~based loop) case the structures need not be involutive, see Remark~\ref{rem:involutivity-spheres}.

(iii) For a unital coFrobenius bialgebra with $A\otimes^* A\to A\otimes^! A$ an isomorphism, involutivity is equivalent to $\boldmu\boldc=0$. (One implication follows from $\boldc=\pm\boldlambda\boldeta$, the other one from the unital coFrobenius relation and associativity of $\boldmu$). Under the same hypothesis, for a counital coFrobenius bialgebra, involutivity is equivalent to $\boldp\boldlambda=0$. 

(iv) For Rabinowitz Floer homology, the map $A\otimes^* A\to A\otimes^! A$ is usually not a topological isomorphism, so the above definition of involutivity does not apply. However, the construction in~\cite[\S8]{CO-Tate} shows that involutivity holds for the operations restricted to suitable finite action intervals, so it holds in a ``filtered'' sense.
\end{remark}

\section{Algebraic Poincar\'e duality}\label{sec:algebraic-PD}
  
{\bf Permutation group action. }
The left action of the permutation group $S_n$ on $A^{\otimes n}$ is defined by
$$
   \rho(a_1\otimes\cdots\otimes a_n) = \eps(\rho,a)a_{\rho^{-1}(1)}\otimes\cdots\otimes a_{\rho^{-1}(n)},
$$
where $\eps(\rho,a)$ is the sign for reordering $a_1\cdots a_n$ to $a_{\rho^{-1}(1)}\cdots a_{\rho^{-1}(n)}$ according to the degrees in $A$. For example, the transposition $\tau=(12)$ and the cyclic permutation $\sigma=(123)$ act by
$$
   \tau(a\otimes b) = (-1)^{|a||b|}b\otimes a,\qquad
   \sigma(a\otimes b\otimes c) = (-1)^{(|a|+|b|)|c|}c\otimes a\otimes b. 
$$
Denoting general transpositions by $\tau_{ij}$, we have in $S_3$ the relations
\begin{equation}\label{eq:perm-relations}
    \sigma=\tau_{12}\tau_{23},\qquad \sigma^{-1}=\sigma^2=\tau_{23}\tau_{12}.
\end{equation}
The $S_n$ action extends canonically to $A^{\otimes^* n}$ and $A^{\otimes^! n}$.

\begin{lemma}
(i) In a unital coFrobenius bialgebra $(A,\boldmu,\boldlambda,\boldeta)$ with copairing $\boldc$, the operation 
$$
\boldbeta=(1\otimes \boldmu\otimes 1)(\boldc\otimes \boldc):\bk\to A\otimes^! A\otimes^! A
$$
is cyclically symmetric, i.e. 
$$
   \sigma\boldbeta=\boldbeta.
$$
(ii) In a counital coFrobenius bialgebra $(A,\boldmu,\boldlambda,\boldeps)$ with pairing $\boldp$, the operation 
$$
\boldB=(\boldp\otimes\boldp)(1\otimes\boldlambda\otimes 1):A\otimes^* A\otimes^* A\to \bk
$$
is cyclically symmetric, i.e. 
$$
  \boldB\sigma=\boldB.
$$
\end{lemma}
    
\begin{proof}
(i) A short computation yields
$$
   \sigma(\boldlambda\otimes 1) = (1\otimes\boldlambda)\tau.
$$
Using this together with $\tau\boldc=(-1)^{|\lambda|}\boldc$ and the relations
$$
   \boldbeta = (\boldlambda\otimes 1)\boldc = (-1)^{|\boldlambda|}(1\otimes\boldlambda)\boldc
$$
we find
$$
  \sigma\boldbeta
  = \sigma(\boldlambda\otimes 1)\boldc  
  = (1\otimes\boldlambda)\tau\boldc
  = (-1)^{|\boldlambda|}(1\otimes\boldlambda)\boldc
  = \boldbeta.
$$
The proof of (ii) is analogous.
\end{proof}

\begin{remark} 
(i) In a unital coFrobenius bialgebra\break $(A,\boldmu,\boldlambda,\boldeta)$ such that $\boldlambda$ is cocommutative, we have in addition
$$
   \tau_{12}\boldbeta = (-1)^{|\boldlambda|}\boldbeta.
$$
So in this case $\boldbeta:\bk\to A\otimes^! A\otimes^! A$ is equivariant with respect to $S_3$ acting on $\bk$ by the sign representation if $|\boldlambda|$ is odd, and by the trivial representation if $|\boldlambda|$ is even.

(ii) In a counital coFrobenius bialgebra $(A,\boldmu,\boldlambda,\boldeps)$ such that $\boldmu$ is commutative, we have in addition 
$$
   \boldB\tau_{12} = (-1)^{|\boldmu|}\boldB.
$$
So in this case $\boldB:A\otimes^* A\otimes^* A\to \bk$ is equivariant with respect to $S_3$ acting on $\bk$ by the sign representation if $|\boldmu|$ is odd, and by the trivial representation if $|\boldmu|$ is even.
\end{remark}

In the sequel we will only be concerned with biunital coFrobenius bialgebras. 

Recall the evaluation map $\ev:A^\vee\otimes^* A\to \bk$ from~\S\ref{sec:Tate}. We define
the \emph{coevaluation map} $\ev^\vee:\bk\to A\otimes^! A^\vee$ by dualizing $\ev$. 

(i) Given a copairing $\boldc:\bk\to A\otimes^! A$ we denote 
$$
\vec\boldc:A^\vee\to A
$$
the map $\vec\boldc=(\ev\otimes 1)(1\otimes\boldc)$, i.e. $\vec\boldc(f)=(-1)^{|f||\boldc|}(f\otimes 1)\boldc$. 
We say that \emph{$\boldc$ is perfect} if $\vec \boldc$ is a topological isomorphism. 

(ii) Given a pairing $\boldp:A\otimes^* A\to \bk$ we denote 
$$
\vec\boldp:A\to A^\vee
$$
the unique map such that $\boldp=\ev(\vec\boldp\otimes 1)$, i.e. $\vec\boldp(a)(b)=\boldp(a\otimes b)$. 
We say that \emph{$\boldp$ is perfect} if $\vec\boldp$ is a topological isomorphism. 

\begin{lemma} \label{lem:cp-perfect}
Let $\boldp$, $\boldc$ be a pairing and a copairing of opposite degrees $|\boldp|=-|\boldc|$ such that
$$
(1\otimes \boldp)(\boldc\otimes 1) = 1 = (-1)^{|\boldp||\boldc|}(\boldp\otimes 1)(1\otimes\boldc). 
$$ 
Then 
$$
\vec\boldc\vec\boldp=1, \qquad \vec\boldp\vec\boldc=1.
$$
In particular, $\boldc$ and $\boldp$ are both perfect. 
\qed
\end{lemma} 

In view of the last relation in Proposition~\ref{prop:biunital-coFrob-bialg}, this implies

\begin{corollary}\label{cor:perfect}  Given $(A,\boldmu,\boldlambda,\boldeta,\boldeps)$ a biunital coFrobenius bialgebra, the associated pairing and copairing are perfect. \qed
\end{corollary}

\begin{definition} \label{defi:compatibility_with_products}
Let $\varphi:A\to B$ be a graded linear map between graded Tate vector spaces, and let $\boldmu_A$, $\boldmu_B$ be products on $A$ and $B$, respectively $\boldlambda_A$, $\boldlambda_B$ be coproducts on $A$ and $B$. 

(i) We say that $\varphi$ \emph{intertwines the products} if 
$$
\varphi\boldmu_A = (-1)^{|\varphi||\boldmu_A|}\boldmu_B\varphi^{\otimes 2}.
$$ 

(ii) We say that $\varphi$ \emph{intertwines the coproducts} if
$$
\varphi^{\otimes 2} \boldlambda_A = (-1)^{|\varphi||\boldlambda_A|}\boldlambda_B\varphi.
$$ 
\end{definition} 

\begin{remark}
This definition is such that the composition of two graded linear maps which intertwine products/coproducts also intertwines products/coproducts, see Appendix~\ref{app:gradings}. Note, however, that the properties of intertwining products and coproducts \emph{are not} dual to each other (i.e., they do not dualize into each other.) 
The property of intertwining the products would be invariant under shifts in the source or target provided one redefined the shifts as 
$$
\ol\boldmu_A = (-1)^{|\boldmu_A|}s\boldmu_A (s\otimes s)^{-1}
$$
and 
$$
\ol\boldlambda_A=(-1)^{|\boldlambda_A|} (s\otimes s)\boldlambda_A s^{-1}.
$$
The unitality relation would then be shift invariant if we redefined it as
$$
  \boldmu(\boldeta\otimes 1) = 1 = (-1)^{|\boldmu|}\boldmu(1\otimes\boldeta),
$$
and the relation in Lemma~\ref{lem:intertwine-unitality} would become
$$
  \boldeta_B=\varphi\boldeta_A,
$$
in agreement with Appendix~\ref{app:gradings}. 
However, with these conventions, unitality would not dualize to counitality and the notion of a biunital coFrobenius bialgebra would not be self-dual. Since we want to keep this property, we will stick to our earlier conventions for shifts.
\end{remark}

\begin{lemma}\label{lem:intertwine-unitality}
Let $\varphi:A\to B$ be a map between graded Tate vector spaces, let $\boldmu_A$, $\boldmu_B$ be products, and let $\boldlambda_A$, $\boldlambda_B$ be coproducts on $A$ and $B$. Assume that $\varphi$ is a topological isomorphism.

(i) If $\varphi$ intertwines the products and $\boldmu_A$ is unital with unit $\boldeta_A$, then $\boldmu_B$ is unital with unit 
$$
\boldeta_B=(-1)^{|\varphi|}\varphi\boldeta_A.
$$

(ii) If $\varphi$ intertwines the coproducts and $\boldlambda_B$ is counital with counit $\boldeps_B$, then $\boldlambda_A$ is counital with counit 
$$
\boldeps_A=\boldeps_B\varphi. 
$$
\qed
\end{lemma} 

\begin{proposition}\label{prop:biunital-coFrob-transpose}
Let $(A,\boldmu,\boldlambda,\boldeta,\boldeps)$ be a biunital coFrobenius bialgebra
with associated pairing $\boldp$ and copairing $\boldc$.
Then $(A,\boldmu\tau,\tau\boldlambda,(-1)^{|\boldmu|}\boldeta,$ $(-1)^{|\boldlambda|}\boldeps)$ is a biunital coFrobenius bialgebra
with associated pairing $(-1)^{|\boldlambda|}\boldp\tau$ and copairing $(-1)^{|\boldmu|}\tau\boldc$.
\end{proposition}

\begin{proof}
For associativity, one first verifies the identities
\begin{equation}\label{eq:tau-sigma-mu}
  \tau(\boldmu\otimes 1)=(1\otimes\boldmu)\sigma,\qquad
  \tau(1\otimes\boldmu)=(\boldmu\otimes 1)\sigma^{-1}
\end{equation}
(where the second one is a formal consequence of the first one). Using these, we compute
\begin{align*}
  \boldmu\tau(\boldmu\tau\otimes 1)
  &= \boldmu\tau(\boldmu\otimes 1)(\tau\otimes 1) \cr
  &= (-1)^{|\boldmu|}\boldmu(\boldmu\otimes 1)\sigma(\tau\otimes 1) \cr
  &= (-1)^{|\boldmu|}\boldmu\tau(1\otimes\boldmu)\sigma^2(\tau\otimes 1) \cr
  &= (-1)^{|\boldmu|}\boldmu\tau(1\otimes\boldmu\tau)(1\otimes\tau)\sigma^2(\tau\otimes 1) \cr
  &= (-1)^{|\boldmu|}\boldmu\tau(1\otimes\boldmu\tau),
\end{align*}
where the second equality follows from associativity of $\boldmu$, and the last equality follows from the last relation in~\eqref{eq:perm-relations}.   
For unitality, one verifies
$$
  \tau(1\otimes\boldeta)=\boldeta\otimes 1
$$
and computes
\begin{align*}
\boldmu\tau(1\otimes(-1)^{|\boldmu|}\boldeta) &= (-1)^{|\boldmu|}\boldmu(\boldeta\otimes 1) = 1, \cr
\boldmu\tau((-1)^{|\boldmu|}\boldeta\otimes 1) &= (-1)^{|\boldmu|}\boldmu(1\otimes\boldeta) = (-1)^{|\boldmu|}.
\end{align*}
The proofs of coassociativity and counitality are analogous, using the identity
\begin{equation}\label{eq:tau-sigma-lambda}
  \sigma(\boldlambda\otimes 1)=(1\otimes\boldlambda)\tau.
\end{equation}
The copairing $\boldc$ is determined by
$$
  \boldc = (-1)^{|\boldlambda||\boldmu|+|\boldmu|}\boldlambda\boldeta,
$$
and applying $(-1)^{|\boldmu|}\tau$ yields the new copairing
$$
  (-1)^{|\boldmu|}\tau\boldc = (-1)^{|\boldlambda||\boldmu|+|\boldmu|}(\tau\boldlambda)\bigl((-1)^{|\boldmu|}\boldeta\bigr).
$$
The coFrobenius relation for $\boldlambda$ is
$$
  \boldlambda = (1\otimes\boldmu)(\boldc\otimes 1) = (-1)^{|\boldmu|}(\boldmu\otimes 1)(1\otimes\boldc).
$$
Applying $\tau$ to its second half yields
\begin{align*}
  \tau\boldlambda
  &= (-1)^{|\boldmu|}\tau(\boldmu\otimes 1)(1\otimes\boldc) \cr
  &= (-1)^{|\boldmu|}\tau(\boldmu\otimes 1)\sigma(\boldc\otimes 1) \cr
  &= (-1)^{|\boldmu|}(1\otimes\boldmu)\sigma^2(\boldc\otimes 1) \cr
  &= (-1)^{|\boldmu|}(1\otimes\boldmu)(1\otimes\tau)(\tau\otimes 1)(\boldc\otimes 1) \cr
  &= (1\otimes\boldmu\tau)\bigl((-1)^{|\boldmu|}\tau\boldc\otimes 1\bigr).
\end{align*}
Here the second equality follows from the identity
$$
  \sigma(\boldc\otimes 1)=1\otimes\boldc,
$$ 
the third one from~\eqref{eq:tau-sigma-mu}, and the fourth one from~\eqref{eq:perm-relations}.
This proves the first half of the coFrobenius relation for $\tau\boldlambda$, and the second half follows analogously.
The pairing $\boldp$ is determined by
$$
  \boldp = (-1)^{|\boldlambda|}\boldeps\boldmu,
$$
and applying $(-1)^{|\boldlambda|}\tau$ yields the new pairing
$$
  (-1)^{|\boldlambda|}\boldp\tau = (-1)^{|\boldlambda|}\bigl((-1)^{|\boldlambda|}\boldeps\bigr)(\boldmu\tau).
$$
The coFrobenius relation for $\boldmu$ is
$$
  (-1)^{|\boldmu||\boldlambda|}\boldmu = (1\otimes\boldp)(\boldlambda\otimes 1)
  = (-1)^{|\boldlambda|}(\boldp\otimes 1)(1\otimes\boldlambda).
$$
Applying $\tau$ to its second half yields
\begin{align*}
  (-1)^{|\boldmu||\boldlambda|}\boldmu\tau
  &= (-1)^{|\boldlambda|}(\boldp\otimes 1)(1\otimes\boldlambda)\tau \cr
  &= (-1)^{|\boldlambda|}(1\otimes\boldp)\sigma(1\otimes\boldlambda)\tau \cr
  &= (-1)^{|\boldlambda|}(1\otimes\boldp)\sigma^2(\boldlambda\otimes 1) \cr
  &= (-1)^{|\boldlambda|}(1\otimes\boldp)(1\otimes\tau)(\tau\otimes 1)(\boldlambda\otimes 1) \cr
  &= \bigl(1\otimes(-1)^{|\boldlambda|}\boldp\tau\bigr)(\tau\boldlambda\otimes 1).
\end{align*}
Here the second equality follows from the identity
$$
  (1\otimes\boldp)\sigma = \boldp\otimes 1,
$$ 
the third one from~\eqref{eq:tau-sigma-lambda}, and the fourth one from~\eqref{eq:perm-relations}.
This proves the first half of the coFrobenius relation for $\boldmu\tau$, and the second half follows analogously.
The symmetry relation is obvious.
\end{proof}

\begin{theorem}[Algebraic Poincar\'e duality] \label{thm:algebraic_Poincare_duality}
Let $(A,\boldmu,\boldlambda,\boldeta,\boldeps)$ be a biunital coFrobenius bialgebra with pairing $\boldp$ and copairing $\boldc$.
Then the mutually inverse maps
$$
\vec\boldp:A\to A^\vee,\qquad \vec\boldc:A^\vee\to A
$$
realize isomorphisms of biunital coFrobenius bialgebras
\begin{align*}
(A,\boldmu,\boldlambda,& \boldeta,\boldeps) \\
& \simeq (A^\vee,(-1)^{|\boldlambda|}\boldlambda^\vee\tau, (-1)^{|\boldmu||\boldlambda|+|\boldlambda|}\tau\boldmu^\vee,\boldeps^\vee,(-1)^{|\boldmu||\boldlambda|+|\boldlambda|+|\boldmu|}\boldeta^\vee),
\end{align*}
i.e. they intertwine the products, they intertwine the coproducts, and they preserve the units and counits. 
\end{theorem} 

\begin{proof}
The proof is by direct computation. 
Note that, since $\vec\boldc$ is the inverse of $\vec\boldp$ and the property of intertwining the products/coproducts is compatible with compositions, it is enough to prove that $\vec\boldp$ is an isomorphism of biunital coFrobenius bialgebras in order to infer the same for $\vec\boldc$. 
That the operations in $A^\vee$ in the theorem indeed define the structure of a biunital coFrobenius bialgebra can be seen as follows.
First, by Lemma~\ref{lem:biunital-coFrob-invariance}(i), $(A^\vee,\boldlambda^\vee,\boldmu^\vee,\boldeps^\vee,\boldeta^\vee)$ is a biunital coFrobenius bialgebra.
Next, by Proposition~\ref{prop:biunital-coFrob-transpose}, $(A^\vee,\boldlambda^\vee\tau,\tau\boldmu^\vee,(-1)^{|\boldlambda|}\boldeps^\vee,(-1)^{|\boldmu|}\boldeta^\vee)$ is a biunital coFrobenius bialgebra. 
Finally, by Lemma~\ref{lem:biunital-coFrob-invariance}(iii),
$$(A^\vee,(-1)^{|\boldlambda|}\boldlambda^\vee\tau, (-1)^{|\boldmu||\boldlambda|+|\boldlambda|}\tau\boldmu^\vee,\boldeps^\vee,(-1)^{|\boldmu||\boldlambda|+|\boldlambda|+|\boldmu|}\boldeta^\vee)$$
is a biunital coFrobenius bialgebra.
\end{proof}

The following is a useful criterion for a biunital coFrobenius bialgebra. 

\begin{corollary}\label{cor:coFrob-criterion}
(i) Let $(A,\boldmu,\boldlambda,\boldeta)$ be a unital coFrobenius bialgebra such that the induced map $\vec\boldc:A^\vee\to A$ is a topological isomorphism. Then 
$$
\boldeps=(-1)^{|\boldmu||\boldlambda|+|\boldlambda|+|\boldmu|}\boldeta^\vee\vec\boldc\,^{-1}
$$ 
makes $(A,\boldmu,\boldlambda,\boldeta,\boldeps)$ a biunital coFrobenius bialgebra.  

(ii) Let $(A,\boldmu,\boldlambda,\boldeps)$ be a counital coFrobenius bialgebra such that the induced map $\vec\boldp:A\to A^\vee$ is a topological isomorphism. Then 
$$
\boldeta=(-1)^{|\boldp|}{\vec\boldp}\,^{-1}(\boldeps^\vee)
$$
makes $(A,\boldmu,\boldlambda,\boldeta,\boldeps)$ a biunital coFrobenius bialgebra. 
\end{corollary} 

\begin{proof}
(i) The proof of Theorem~\ref{thm:algebraic_Poincare_duality} shows that $\vec\boldc$ intertwines the coproducts $(-1)^{|\boldmu||\boldlambda|+|\boldlambda|}\tau\boldmu^\vee$ and $\boldlambda$. Since $\boldmu$ has unit $\boldeta$, the coproduct $\boldmu^\vee$ has counit $\boldeta^\vee$ and $(-1)^{|\boldmu||\boldlambda|+|\boldlambda|}\tau\boldmu^\vee$ has counit $(-1)^{|\boldmu||\boldlambda|+|\boldlambda|+|\boldmu|}\boldeta^\vee$, hence $\boldlambda$ has counit $\boldeps=(-1)^{|\boldmu||\boldlambda|+|\boldlambda|+|\boldmu|}\boldeta^\vee\vec\boldc\,^{-1}$. Now part (i) follows from Proposition~\ref{prop:coFrob-criterion}(i). The proof of part (ii) is analogous. 
\end{proof}

\begin{remark}[Dichotomy finite- vs. infinite dimensional] \qquad \label{rmk:dichotomy}

Let $(A,\boldmu,\boldlambda,\boldeta,\boldeps)$ be a biunital coFrobenius bialgebra. In view of Corollary~\ref{cor:perfect} we have an isomorphism $A\simeq A^\vee$ induced by the pairing or the copairing.
This presents us with the following dichotomy in relation to tensor products:

(i) If $A$ is finite dimensional, then the * and ! tensor products of $A$ are finite dimensional and coincide with the algebraic tensor product. In particular, the canonical map $A\otimes^* A\to A\otimes^! A$ is a topological isomorphism.   

(ii) If $A$ is infinite dimensional, then the canonical map $A\otimes^* A\to A\otimes^! A$ is \emph{never} a topological isomorphism.
Indeed, assume first that $A=\bigoplus_{i\in\Z}A_i$ is nontrivial in infinitely many degrees. Then by duality there exist infinitely many $i\geq 0$ such that $A_i$ and $A_{|\boldc|-i}$ are nontrivial, and the map $A\otimes^* A\to A\otimes^! A$ is not an isomorphism because $(A\otimes^* A)_{|\boldc|}$ is a direct sum and $(A\otimes^! A)_{|\boldc|}$ a direct product. Otherwise, $A$ is infinite dimensional in some degree, so after shifting we may assume without loss of generality that $A$ is concentrated in degree $0$. By~\cite[Proposition~3.15]{CO-Tate} we have $A=D\oplus D^\vee$ for an infinite dimensional discrete vector space $D$. 
Write $(D\oplus D^\vee)\otimes^*(D\oplus D^\vee)=(D\otimes^*D) \oplus (D\otimes^* D^\vee)\oplus (D^\vee\otimes^* D)\oplus (D^\vee\otimes^*D^\vee)$ and observe that, while the canonical maps $D\otimes^*D\to D\otimes^!D$ and $D^\vee\otimes^*D^\vee\to D^\vee\otimes^!D^\vee$ are isomorphisms (Proposition~\ref{prop:tensor_discrete}), the map $D^\vee\otimes^* D\to D^\vee\otimes^! D$ is not an isomorphism since it can be identified with the strict inclusion $D^\vee[X]\to D^\vee[[X]]$, where $X$ is a basis of $D$~\cite[Example~5.15]{CO-Tate}. Here $D^\vee[X]=\oplus_{x\in X} D^\vee$ and $D^\vee[[X]]=\{(a_x)\in\prod_{x\in X} {D^\vee} \mbox{ zero convergent}\}$, where $(a_x)$ is called zero convergent if $\{x\, : \, a_x\notin U\}$ is finite for each open linear subspace $U\subset D^\vee$.
\end{remark}

{\bf Equivalence to Tate Frobenius algebras. }
We conclude this section by yet another equivalent description of biunital coFrobenius bialgebras which will be useful in~\S\ref{sec:manifolds}. 

\begin{definition}\label{def:graded-symm-Frob}
A {\em Tate Frobenius algebra}\footnote{We opted for the terminology ``Tate Frobenius algebras'' in order to stress the topological character of the underlying vector spaces. Another option would have been ``graded symmetric Frobenius algebras'' because,
in the finite dimensional case, in view of Remark~\ref{rem:graded-symm-Frob} this corresponds to the notion of a symmetric Frobenius algebra (see e.g.~\cite{Kock}), as well as to the notion of a cyclic differential graded algebra with trivial differential in~\cite{CFL}.
If in addition the product is commutative it corresponds to the notion of a Poincar\'e duality algebra (see e.g.~\cite{Lambrechts-Stanley}).
Frobenius algebras go back to Frobenius~\cite{Frobenius} and the terminology was introduced by Brauer-Nesbitt~\cite{Brauer-Nesbitt}. Early references for Frobenius algebras in relation with TQFT are Dijkgraaf~\cite{Dijkgraaf},  Quinn~\cite{Quinn}, Abrams~\cite{Abrams}, Lauda-Pfeiffer~\cite{Lauda-Pfeiffer}.
}
$(A,\boldmu,\boldeta,\boldeps)$ consists of
\begin{itemize}
\item a graded Tate vector space $A$;
\item a degree zero associative product $\boldmu:A\otimes^* A\to A$ with unit $\boldeta$;
\item a linear map $\boldeps:A\to\bk$
\end{itemize}satisfying the following conditions:
\begin{enumerate}
\item the pairing $\boldp=(-1)^{|\boldeps|}\boldeps\boldmu:A\otimes^*A\to\bk$ is perfect;
\item the triple product $\boldp(\boldmu\otimes 1):A\otimes^*A\otimes^*A\to\bk$ is cyclically symmetric.
\end{enumerate}
\end{definition}

\begin{remark}\label{rem:graded-symm-Frob}
(a) The map $\boldeps$ can be recovered from $\boldp$ and $\boldeta$ via $\boldeps=(-1)^{|\boldp|}\boldp(\boldeta\otimes 1)$, so one can equivalently define a Tate Frobenius algebra in terms of the data $(A,\boldmu,\boldeta,\boldp)$. 

(b) By associativity of $\boldmu$, the triple product satisfies
$$
  \boldp(\boldmu\otimes 1)
  = (-1)^{|\boldeps|}\boldeps\boldmu(\boldmu\otimes 1) 
  = (-1)^{|\boldeps|}\boldeps\boldmu(1\otimes\boldmu) 
  = \boldp(1\otimes\boldmu).
$$
(c) The cyclic symmetry (ii) is equivalent to the symmetry $\boldp=\boldp\tau$. Indeed, (ii) implies
\begin{align*}
  \boldp\tau
  &= \boldp\tau(1\otimes\boldmu)(1\otimes\boldeta\otimes 1) 
  = \boldp(\boldmu\otimes 1)\sigma^{-1}(1\otimes\boldeta\otimes 1) \cr
  &= \boldp(\boldmu\otimes 1)(1\otimes\boldeta\otimes 1) 
  = \boldp
\end{align*}
where the first and last equalities follow from unitality of $\boldmu$, the second one from equation~\eqref{eq:tau-sigma-mu}, and the third one from (ii). Conversely, $\boldp=\boldp\tau$ implies 
$$
  \boldp(\boldmu\otimes 1)\sigma^{-1}
  =\boldp\tau(1\otimes\boldmu)
  =\boldp(1\otimes\boldmu)
  =\boldp(\boldmu\otimes 1).
$$
where the first equality follows from equation~\eqref{eq:tau-sigma-mu}, the second one from symmetry of $\boldp$, and the third one from (b). 
\end{remark}

\begin{proposition}\label{prop:cyclic-graded-algebra}
The structure of a Tate Frobenius algebra is equivalent to that of a biunital coFrobenius bialgebra with $|\boldmu|=0$. 
\end{proposition}

\begin{proof}
Suppose first that $(A,\boldmu,\boldlambda,\boldeta,\boldeps)$ is a biunital coFrobenius bialgebra with $|\boldmu|=0$. Then $|\boldeps|=-|\boldlambda|$, and $\boldp=(-1)^{|\boldeps|}\boldeps\boldmu=(-1)^{|\boldlambda|}\boldeps\boldmu$ is perfect by Lemma~\ref{lem:cp-perfect}. The cyclic symmetry condition (ii) follows from Remark~\ref{rem:graded-symm-Frob}(c).

Conversely, suppose that $(A,\boldmu,\boldeta,\boldeps)$ is a Tate Frobenius algebra. Remark~\ref{rem:graded-symm-Frob}(c) yields the symmetry $\boldp=\boldp\tau$. 
Perfectness of $\boldp$ means that the equation $\vec\boldp(a)(b)=\boldp(a,b)$ defines a topological isomorphism $\vec\boldp:A\to A^\vee$ of degree $|\vec\boldp|=|\boldp|=|\boldeps|$. Symmetry of $\boldp$ means that 
$$
  \vec\boldp(a)(b) = (-1)^{|a||b|}\vec\boldp(b)(a). 
$$
We define the coproduct
\begin{equation}\label{eq:lambda-via-mu}
   \boldlambda := (-1)^{|\boldeps|}\bigl(\vec\boldp\boldmu(\vec\boldp\otimes\vec\boldp)^{-1}\tau\bigr)^\vee: A\to A\otimes^!A
\end{equation}
of degree $|\boldlambda|=-|\boldeps|$. The definition is chosen so that $\vec\boldp$ intertwines the products $\boldmu$ and $(-1)^{|\boldlambda|}\boldlambda^\vee\tau$, 
$$
   \vec\boldp\boldmu = (-1)^{|\boldlambda|}\boldlambda^\vee\tau(\vec\boldp\otimes\vec\boldp),
$$
as it should according to Theorem~\ref{thm:algebraic_Poincare_duality}. Tensoring with the identity and applying the evaluation map we obtain
$$
  \ev(\vec\boldp\boldmu\otimes 1)
  = (-1)^{|\boldlambda|}\ev\bigl(\boldlambda^\vee\tau(\vec\boldp\otimes\vec\boldp)\otimes 1\bigr)
  = (\boldp\otimes \boldp)(1\otimes\boldlambda\otimes 1)\sigma,
$$
  where the last equality follows by an explicit computation evaluating both sides on a tensor $a\otimes b\otimes c$. By definition of $\vec\boldp$ and associativity, the left hand side equals $\ev(\vec\boldp\boldmu\otimes 1) = \boldp(\boldmu\otimes 1) = \boldp(1\otimes\boldmu)$. Since this is cyclically symmetric, we can drop $\sigma$ on the right hand side and get
$$
  \boldp(1\otimes\boldmu) = (\boldp\otimes \boldp)(1\otimes\boldlambda\otimes 1).
$$
This is equivalent (by tensoring with the identity, applying the pairing, and using that $\boldp$ is perfect and symmetric) to the biunital coFrobenius relation $\boldmu = (1\otimes \boldp)(\boldlambda\otimes 1)$. In more detail~: 
\begin{align*}
& \boldp(1\otimes\boldmu)=(\boldp\otimes\boldp)(1\otimes\boldlambda\otimes1) \\
\Leftrightarrow \quad & \boldp(1\otimes\boldmu)=\boldp(1\otimes (1\otimes \boldp)(\boldlambda\otimes1)) \\
\Leftrightarrow \quad & \boldp(\boldmu\otimes 1)\sigma^{-1}=\boldp((1\otimes \boldp)(\boldlambda\otimes1)\otimes 1)\sigma^{-1}\\
\Leftrightarrow \quad & \ev(\vec\boldp\boldmu\otimes 1)=\ev(\vec\boldp(1\otimes \boldp)(\boldlambda\otimes1)\otimes 1)\\
\Leftrightarrow \quad & \vec\boldp\boldmu=\vec\boldp(1\otimes \boldp)(\boldlambda\otimes1)\\
\Leftrightarrow \quad & \boldmu=(1\otimes \boldp)(\boldlambda\otimes1).
\end{align*}
The second equivalence uses that $\boldp=\boldp\tau$ and $\tau(1\otimes\boldmu)=(\boldmu\otimes 1)\sigma^{-1}$, as well as the analogous relation with $\boldmu$ replaced by $(1\otimes \boldp)(\boldlambda\otimes1)$. The third equivalence uses that $\boldp=\ev(\vec\boldp\otimes 1)$ and the fact that $\sigma^{-1}$ is an isomorphism. The fourth equivalence is a general property of $\ev$, and the last equivalence uses that $\vec\boldp$ is an isomorphism.    
 
Applying $(-1)^{|\boldeps|}\boldeps$ to both sides of the relation $\boldmu = (1\otimes \boldp)(\boldlambda\otimes 1)$ gives 
$$
  \boldp
  = (-1)^{|\boldeps|}\boldeps\boldmu
  = (-1)^{|\boldeps|}(\boldeps\otimes \boldp)(\boldlambda\otimes 1)
  = \boldp\bigl((\boldeps\otimes 1)\boldlambda\otimes 1\bigr),
$$
hence the counitality relation $(\boldeps\otimes 1)\boldlambda=1$. 
The other relations in Definition~\ref{defi:biunital-coFrobenius-bialgebra} follow from these ones by symmetry and duality.
\end{proof}

\begin{remark}
In the preceding proof, the isomorphism $\vec\boldp:A\to A^\vee$ allows us to transfer an operation on $A$ to and operation on $A^\vee$ and vice versa. Let us decorate the transfered operation by $\;\tilde{}\;$. Thus the product $\boldmu$ transfers to a product $\wt\boldmu=\vec\boldp\boldmu(\vec\boldp\otimes\vec\boldp)^{-1}$ on $A^\vee$, which in turn dualizes to a coproduct $(\wt\boldmu)^\vee$ on $A$. On the other hand, we can first dualize $\boldmu$ to a coproduct $\boldmu^\vee$ on $A^\vee$ and then transfer it to a coproduct $\wt{\boldmu^\vee}$ on $A$. One easily checks that
$$
  (\wt\boldmu)^\vee=\wt{\boldmu^\vee}.
$$
In this notation, the coproduct $\boldlambda$ takes the equivalent forms
$$
  (-1)^{|\boldeps|}\boldlambda = (\wt\boldmu\tau)^\vee = \tau(\wt\boldmu)^\vee = \tau\wt{\boldmu^\vee}.
$$
\end{remark}

Formula~\eqref{eq:lambda-via-mu} for $\boldlambda$ in terms of $\boldmu$ implies

\begin{corollary}
For a unital coFrobenius bialgebra, commutativity of $\boldmu$ is equivalent to cocommutativity of $\boldlambda$. \qed
\end{corollary}

We close this section with a discussion of

{\bf Morphisms of biunital coFrobenius bialgebras. }
The following definition generalizes the classical definition of a morphism of Frobenius algebras, as given in~\cite[2.4.4]{Kock} in the ungraded, finite-dimensional case. 

\begin{definition}
Let $(A,\boldmu_A,\boldeta_A,\boldlambda_A,\boldeps_A)$ and $(B,\boldmu_B,\boldeta_B,\boldlambda_B,\boldeps_B)$ be two biunital coFrobenius bialgebras. A morphism $f:A\to B$ is a degree 0 linear continuous map that intertwines the structures. 
\end{definition} 

The following proposition generalizes the classical fact that morphisms between ungraded finite-dimensional Frobenius algebras are always invertible~\cite[Lemma~2.4.5]{Kock}. The method of proof is borrowed from~\cite{Grosskopf-Vercruysse}. 

\begin{proposition} \label{prop:morphisms}
Any morphism between biunital coFrobenius bialgebras is a topological isomorphism. 
\end{proposition}

\begin{proof}
Let $\boldp_A$, $\boldp_B$, $\boldc_A$, $\boldc_B$ be the pairing and copairing on $A$ and $B$. Then $(f\otimes f)\boldc_A=\boldc_B$ and $\boldp_B(f\otimes f)=\boldp_A$. 

Let $g:B\to A$ be defined by 
$$
g=(1\otimes \boldp_B)(1\otimes f\otimes 1)(\boldc_A\otimes 1). 
$$  
Then $g$ is an inverse of $f$: it is linear continuous by definition, and  
\begin{align*}
f\circ g & =(f\otimes \boldp_B)(1\otimes f\otimes 1)(\boldc_A\otimes 1)\\
& =(1\otimes\boldp_B)((f\otimes f)\boldc_A\otimes 1)\\
& =(1\otimes\boldp_B)(\boldc_B\otimes 1)= 1_B,
\end{align*}
\begin{align*}
g\circ f & = (1\otimes\boldp_B)(1\otimes f\otimes 1)(\boldc_A\otimes f)\\
& = (1\otimes \boldp_B)(1\otimes f\otimes f)(\boldc_A\otimes 1)\\
& = (1\otimes \boldp_A)(\boldc_A\otimes 1)=1_A.
\end{align*}
\end{proof}

\section{Graded 2D open-closed TQFT}\label{sec:TQFT}

In this section we define the notion of a 2D open-closed TQFT in graded Tate vector spaces.

We begin by recalling the description of an ungraded 2D open-closed TQFT in terms of generators and relations from Lauda and Pfeiffer~\cite{Lauda-Pfeiffer}. It associates to the circle a finite dimensional $\bk$-vector space $C$ and to the closed interval a finite dimensional $\bk$-vector space $A$. Its generators are the following operations:
\begin{itemize}
\item the product $ \boldmu_C:C\otimes C\to C$ and the unit $\boldeta_C:\bk\to C$;
\item the coproduct $\boldlambda_C:C\to C\otimes C$ and the counit $\boldeps_C:C\to \bk$;
\item the product $ \boldmu_A:A\otimes A\to A$ and the unit $\boldeta_A:\bk\to A$;
\item the coproduct $\boldlambda_A:A\to A\otimes A$ and the counit $\boldeps_A:A\to \bk$;
\item the {\em zipper} (or closed-open map) $\boldzeta:C\to A$;
\item the {\em cozipper} (or open-closed map) $\boldzeta^*:A\to C$.
\end{itemize}
See Figure~\ref{fig:zipper-cozipper} for the zipper and cozipper, and Figure~\ref{fig:TQFT-operations} for the other operations. Inputs are denoted by $+$ (positive punctures), outputs are denoted by $-$ (negative punctures).
\begin{figure} [ht]
\centering
\includegraphics[width=.5\textwidth]{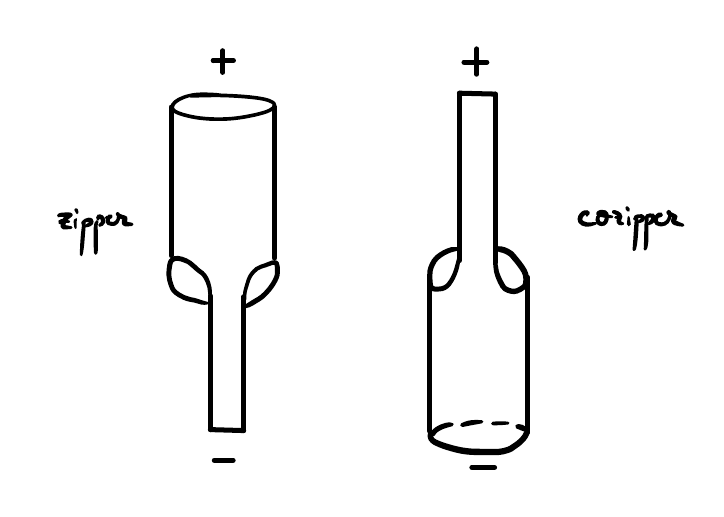}
\caption{The zipper and the cozipper.}
\label{fig:zipper-cozipper}
\end{figure}
\begin{figure} [ht]
\centering
\includegraphics[width=\textwidth]{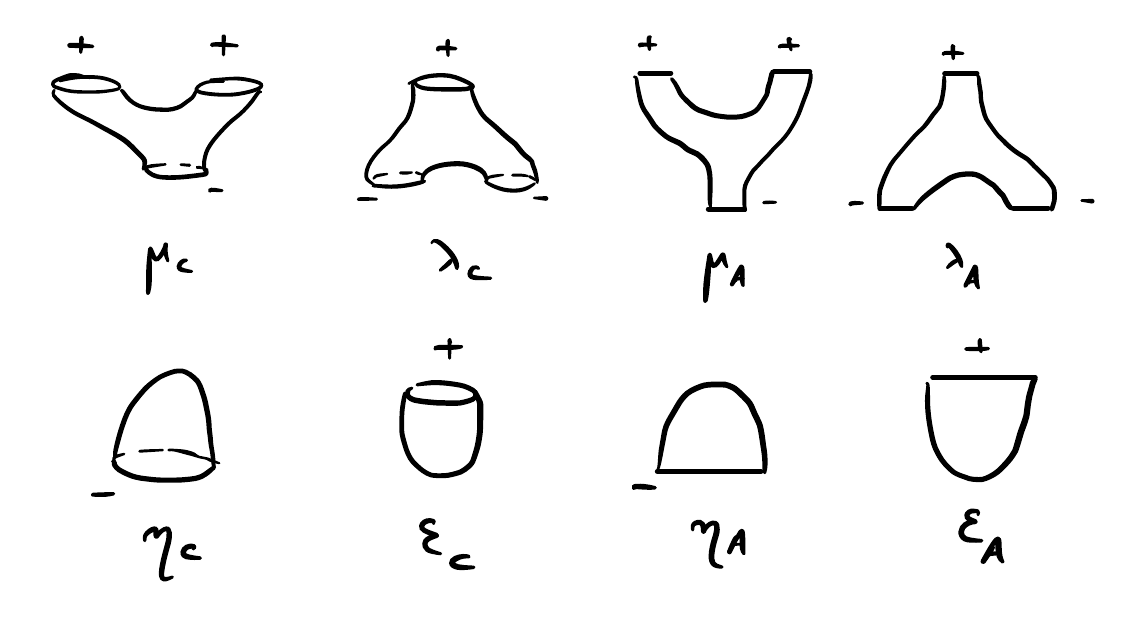}
\caption{The operations constituting a (graded) 2D open-closed TQFT.}
\label{fig:TQFT-operations}
\end{figure}

These satisfy the following relations (in our terminology):
\renewcommand{\theenumi}{\arabic{enumi}}
\begin{enumerate}
\item $(C, \boldmu_C,\boldeta_C,\boldlambda_C,\boldeps_C)$ is a commutative and cocommutative biunital coFrobenius bialgebra, called \emph{the closed sector}. 
\item $(A, \boldmu_A,\boldeta_A,\boldlambda_A,\boldeps_A)$ is a biunital coFrobenius bialgebra, called \emph{the open sector}. 
\item The zipper is an algebra homomorphism, 
$$
   \boldmu_A(\boldzeta\otimes\boldzeta) = \boldzeta\, \boldmu_C,\qquad \boldzeta\,\boldeta_C=\boldeta_A.   
$$
\item The zipper lands in the center of $ \boldmu_A$, 
$$
   \boldmu_A(\boldzeta\otimes 1) =  \boldmu_A\tau(\boldzeta\otimes 1).  
$$
\item The cozipper is dual to the zipper via the copairings $\boldc_C=\boldlambda_C\boldeta_C$ and $\boldc_A=\boldlambda_A\boldeta_A$, 
$$
  (1\otimes\boldzeta)\boldc_C = (\boldzeta^*\otimes 1)\boldc_A.
$$
\item The {\em Cardy condition}
$$
  \boldzeta\,\boldzeta^* =  \boldmu_A\tau\boldlambda_A. 
$$
\end{enumerate}
See Figure~\ref{fig:Gysin1-new} for relation (3). The middle term in Figure~\ref{fig:Gysin1-new} expresses the TQFT philosophy: the term $\mu_A(\zeta\otimes\zeta)$ on the left equals the term $\zeta\mu_C$ on the right because they both express the operation defined by a 2-disc with two interior positive punctures and one negative boundary puncture. The Cardy condition (6) is depicted in Figure~\ref{fig:Cardy}: the two terms $\zeta\zeta^*$ and $\mu_A\tau\lambda_A$ are equal because they both express the operation defined by an annulus with one positive puncture on one of its boundary components and one negative puncture on the other boundary component. We refer to Lauda-Pfeiffer~\cite{Lauda-Pfeiffer} for a complete set of figures representing all the TQFT relations. In~\cite{Lauda-Pfeiffer} condition (5) is stated as the equivalent condition in Lemma~\ref{lem:cozipper-dual-zipper} below (without signs), and it is proved that a 2D open-closed TQFT, seen as a monoidal functor from the category of 1+1 cobordisms to the category of finite dimensional $\bk$-vector spaces, is equivalent to the algebraic structure defined by these generators and relations. 
\begin{figure} [ht]
\centering
\includegraphics[width=\textwidth]{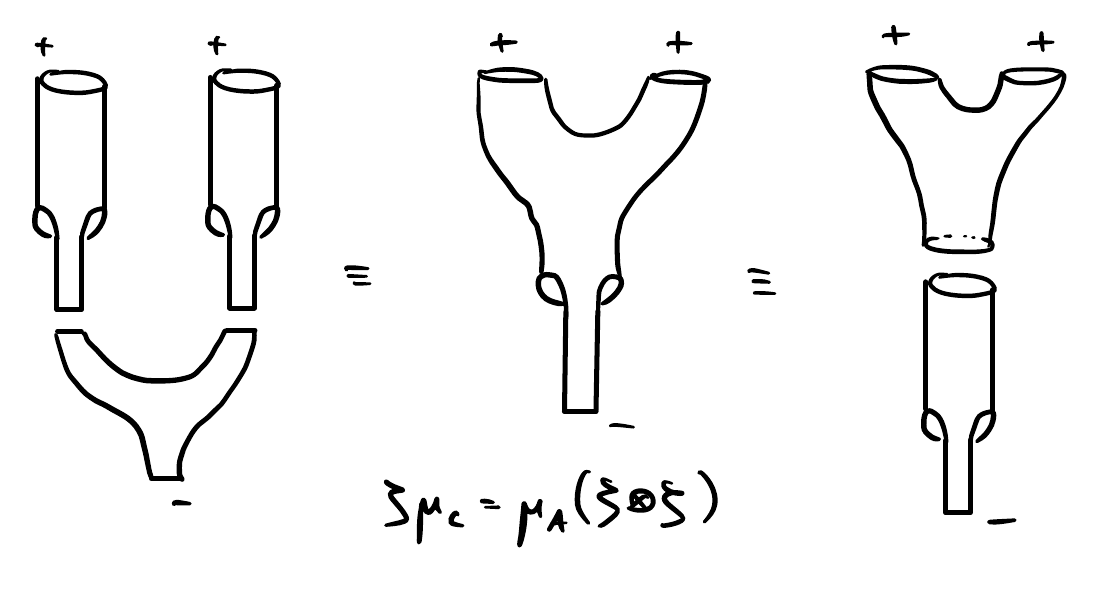}
\caption{The zipper is an algebra homomorphism.}
\label{fig:Gysin1-new}
\end{figure}
\begin{figure} [ht]
\centering
\includegraphics[width=.7\textwidth]{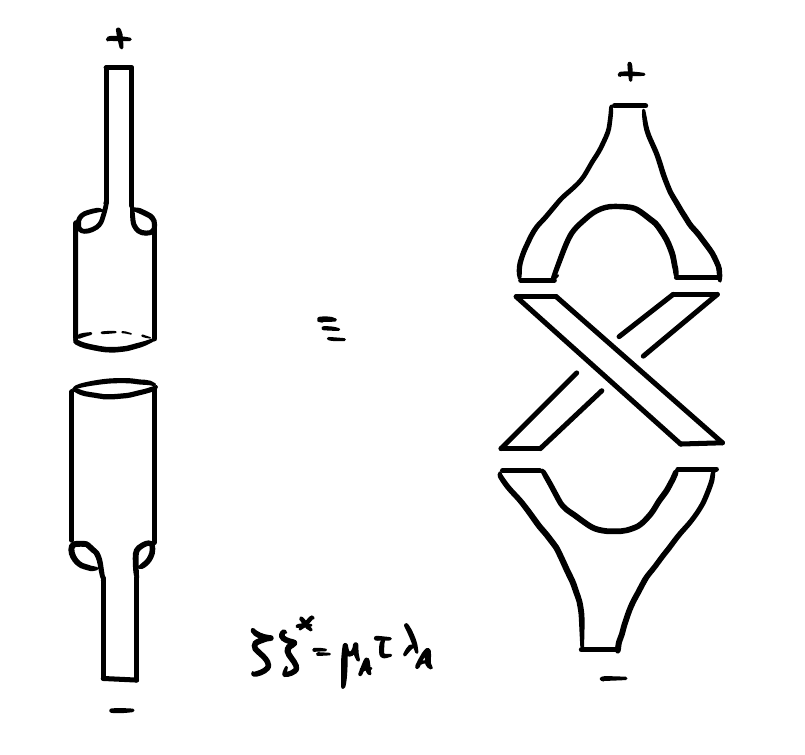}
\caption{The Cardy condition.}
\label{fig:Cardy}
\end{figure}

We now generalize this description of a 2D open-closed TQFT to graded Tate vector spaces. Compared to Lauda-Pfeiffer~\cite{Lauda-Pfeiffer}, this is a two-fold generalization: we use gradings, and we change the category of underlying vector spaces from finite dimensional to Tate. 

After shifting degrees, we can and will assume that the products $ \boldmu_C$ and $ \boldmu_A$ have degree $0$. Then relations (1)--(5) determine all the other degrees in terms of the degrees of $\boldlambda_C$ and $\boldlambda_A$: 
\begin{gather*}
  | \boldmu_C| = |\boldeta_C| = | \boldmu_A| = |\boldeta_A| = |\boldzeta| = 0, \cr
  |\boldlambda_C|=|\boldc_C| = -|\boldeps_C| = -|\boldp_C|, \qquad 
  |\boldlambda_A|=|\boldc_A| = -|\boldeps_A| = -|\boldp_A|, \cr
  |\boldzeta^*| = |\boldlambda_C| - |\boldlambda_A|. 
\end{gather*}
Relation~(6) presents a double constraint: \\
(i) It yields the additional condition $|\boldzeta^*|=|\boldlambda_A|$, which combined with the previous ones implies $|\boldlambda_C|=2|\boldlambda_A|$.
Therefore, relation~(6) makes sense only under this additional condition on the degrees. \\
(ii) The target of $\tau\boldlambda_A$ is $A\otimes^! A$, and the source of $\boldmu_A$ is $A\otimes^* A$. Therefore, the expression $\boldmu_A\tau\boldlambda_A$ in relation~(6) makes sense only if the canonical map $A\otimes^* A\to A\otimes^! A$ is a topological isomorphism. By Remark~\ref{rmk:dichotomy}, this is equivalent to $A$ being finite dimensional. 

\begin{definition} \label{defi:gradedTQFT}
A {\em graded 2D open-closed TQFT (in the category of Tate vector spaces)} consists of two graded Tate vector spaces $C,A$ and the following morphisms: 
\begin{itemize}
\item the degree $0$ product $ \boldmu_C:C\otimes^* C\to C$ and the unit $\boldeta_C:\bk\to C$;
\item the coproduct $\boldlambda_C:C\to C\otimes^! C$ and the counit $\boldeps_C:C\to \bk$;
\item the degree $0$ product $ \boldmu_A:A\otimes^* A\to A$ and the unit $\boldeta_A:\bk\to A$;
\item the coproduct $\boldlambda_A:A\to A\otimes^! A$ and the counit $\boldeps_A:A\to \bk$;
\item the {\em zipper} (or closed-open map) $\boldzeta:C\to A$;
\item the {\em cozipper} (or open-closed map) $\boldzeta^*:A\to C$.
\end{itemize} 
These are required to satisfy the following signed versions of relations (1)--(6) above:
\begin{enumerate}
\item $(C, \boldmu_C,\boldeta_C,\boldlambda_C,\boldeps_C)$ is a commutative and cocommutative biunital coFrobenius bialgebra, called \emph{the closed sector}. 
\item $(A, \boldmu_A,\boldeta_A,\boldlambda_A,\boldeps_A)$ is a biunital coFrobenius bialgebra, called \emph{the open sector}. 
\item The zipper is an algebra homomorphism, 
$$
   \boldmu_A(\boldzeta\otimes\boldzeta) = \boldzeta\, \boldmu_C,\qquad \boldzeta\,\boldeta_C=\boldeta_A.   
$$
\item The zipper lands in the center of $ \boldmu_A$, 
$$
  \boldmu_A(\boldzeta\otimes 1) = \boldmu_A\tau(\boldzeta\otimes 1).  
$$
\item The cozipper is dual to the zipper via the copairings $\boldc_C=\boldlambda_C\boldeta_C$ and $\boldc_A=\boldlambda_A\boldeta_A$, 
$$
  (1\otimes\boldzeta)\boldc_C = (\boldzeta^*\otimes 1)\boldc_A.
$$
\item The {\em graded Cardy condition} 
$$
\left\{\begin{array}{ll}
\boldzeta\,\boldzeta^*=(-1)^{|\lambda_A|} \boldmu_A\tau\boldlambda_A & 
\mbox{ if } |\boldlambda_C|=2|\boldlambda_A| \mbox{ and } \dim A<\infty, 
\\
\boldzeta\,\boldzeta^*= \boldmu_A\tau\boldlambda_A=0
& \mbox{ if } |\boldlambda_C|\neq 2|\boldlambda_A| \mbox{ and } \dim A<\infty, 
\\
\boldzeta\,\boldzeta^*=0
& \mbox{ if } |\boldlambda_C|\neq 2|\boldlambda_A|  \mbox{ and } \dim A=\infty. 
\end{array}\right.
$$
\end{enumerate}
\end{definition}

\begin{remark}[on the graded Cardy condition] \label{rmk:graded_Cardy}
Assume first $|\boldlambda_C|=2|\boldlambda_A|$. If $A$ is finite dimensional, then by Corollary~\ref{cor:muAlambdaA} the graded Cardy condition implies
\begin{equation} \label{eq:muAlambdaA}
\boldmu_A\boldlambda_A = \boldmu_A(1\otimes\boldzeta\boldzeta^*\boldeta_A) = (-1)^{|\boldlambda_A|}(1\otimes\boldeps_A\boldzeta\boldzeta^*)\boldlambda_A.
\end{equation}
If $A$ is infinite dimensional we take the graded Cardy condition as the \emph{definition} of the composition $\boldmu_A\tau\boldlambda_A$, and equation~\eqref{eq:muAlambdaA} as the \emph{definition} of the composition $\boldmu_A\boldlambda_A$. 

Assume now $|\boldlambda_C|\neq 2|\boldlambda_A|$. If $A$ is finite dimensional, then by Proposition~\ref{prop:muAlambdaA} the relation $\boldmu_A\tau\boldlambda_A=0$, which is half of the graded Cardy condition, implies $\boldmu_A\boldlambda_A=0$. 
If $A$ is infinite dimensional we \emph{define} 
$$
\boldmu_A\boldlambda_A = \boldmu_A\tau\boldlambda_A=0. 
$$
\end{remark}

\begin{remark}[on involutivity in the closed sector] \label{rmk:involutivity}
Let the closed sector of a graded 2D open-closed TQFT be $(C,\boldmu_C,\boldlambda_C,\boldeta_C,\boldeps_C)$.
Recall from Remark~\ref{rem:involutivity} and Remark~\ref{rmk:dichotomy} that it is involutive iff $\dim C<\infty$ and $\boldmu_C\boldlambda_C=0$.

Assume first that $\boldmu_C$ and $\boldlambda_C$ have different parities. If $C$ is finite dimensional and $2\neq 0$ in $\bk$, then it is involutive (Remark~\ref{rem:involutivity}). 
If $C$ is infinite dimensional we \emph{define} the composition $\boldmu_C\boldlambda_C$ to be $0$. 

Assume now that $\boldmu_C$ and $\boldlambda_C$ have the same parities. If $C$ is finite dimensional then the composition $\boldmu_C\boldlambda_C$ is well-defined but it might be nonzero. If $C$ is infinite dimensional we do not know how to meaningfully define $\boldmu_C\boldlambda_C$. 
\end{remark}
 
{\bf Relation to twisted 2D open-closed TQFT. }
Consider an ungraded 2D open-closed TQFT, viewed as a functor from the category of 1+1 oriented cobordisms to the category of finite dimensional $\bk$-vector spaces. By convention, given an oriented cobordism represented by a surface $S$, we call $S^{\mathrm{in}}$ (resp.~$S^{\mathrm{out}}$) the union of the boundary arcs or circles that correspond to inputs (resp.~outputs),
and $S^{\mathrm{free}}=\p S\setminus (S^{\mathrm{in}}\cup S^{\mathrm{out}})$ the free part of the boundary.
One can obtain another functor to the category of \emph{graded} finite dimensional vector spaces by shifting the degree of each operation associated to a surface $S$ by $c\chi(S,S^{\mathrm{in}})$, where $\chi(S,S^{\mathrm{in}})$ is the Euler characteristic of the pair $(S,S^{\mathrm{in}})$ and $c\in\Z$ is a fixed integer,
and inserting corresponding signs in the composition of operations.
Another admissible shift is given by $d\chi(S,S^{\mathrm{out}})$ for some fixed integer $d\in\Z$, and yet another one by combining these two. 
This {\em twisting} procedure is described by Godin~\cite{Godin} and Wahl-Westerland~\cite[\S6.3]{Wahl-Westerland}, and had also appeared in Costello~\cite{Costello}. 

By combining the twists $(S,S^{\mathrm{in}})$ and $(S,S^{\mathrm{out}})$ in the open sector one can achieve any parity of the open sector product and coproduct. Indeed, for the half pair of pants that defines the open sector coproduct we have $\chi(S,S^{\mathrm{in}})=0$, and for the half pair of pants that defines the open sector product we have $\chi(S,S^{\mathrm{in}})=-1$ (the values are exchanged if one considers instead $\chi(S,S^{\mathrm{out}})$). In contrast, this twisting procedure always preserves the difference of parity between the closed sector product and coproduct: for both pairs of pants that define the closed sector coproduct and product we have $\chi(S,S^{\mathrm{in}})=-1$. 

It was verified by Jonathan Laurent Clivio~\cite{Clivio}
that the signs in Definition~\ref{defi:biunital-coFrobenius-bialgebra} of a biunital coFrobenius bialgebra coincide with the signs obtained by interpreting it as a twisted {\em open} TQFT. 
This discussion implies the following for the signs in Definition~\ref{defi:gradedTQFT}.
\begin{itemize}
\item The signs in the open sector can be obtained by interpreting it as a twisted open TQFT.
\item The signs in the closed sector can be obtained by interpreting it as a twisted {\em open} TQFT, but {\em not} by interpreting it as a closed TQFT if $\boldmu_C$ and $\boldlambda_C$ have opposite parity.
\end{itemize}

As a consequence, although the twisting procedure can be applied simultaneously to the closed- and to the open sector of an open-closed TQFT, it does not allow to cover all possible parities of open/closed product/coproduct starting from an ungraded open-closed TQFT. In particular, this twisting procedure does not cover our main infinite dimensional example given by Rabinowitz Floer homology, where the closed sector product and coproduct have different parities. 
This leads us to formulate the following question:

\begin{question}
How to explain from the twisted TQFT perspective the signs of a graded open-closed 2D TQFT from Definition~\ref{defi:gradedTQFT} in the case where the closed sector product and coproduct have different parities?
\end{question}

\begin{remark}
Consider a graded open-closed 2D TQFT with closed sector $C$ and open sector $A$ over a field $\bk$ of characteristic $\neq 2$. Assume further that $\boldmu_C$ is even and $\boldlambda_C$ is odd, so that, in particular, $|\boldlambda_C|\neq2|\boldlambda_A|$.\footnote{This is the setting of Rabinowitz Floer homology in~\S\ref{sec:RFH}.} Then, following Remarks~\ref{rmk:graded_Cardy} and~\ref{rmk:involutivity}, we have 
$$
\boldmu_C\boldlambda_C=0,\qquad \boldmu_A\boldlambda_A=0,\qquad \boldmu_A\tau\boldlambda_A=0,\qquad \boldzeta\boldzeta^*=0, 
$$
where these identities are postulated in the case where either $C$ or $A$ are infinite dimensional. As a consequence, the TQFT operations vanish in genus $\ge 1$,
and also for genus $0$ surfaces that
have at least two boundary components with nonempty free parts. 
The example of Rabinowitz loop homology of odd-dimensional spheres from~\S\ref{sec:spheres} shows that the operation $\boldzeta^*\boldzeta$, which corresponds to a surface with one free boundary component, can be nonzero. 
\end{remark}
  
{\bf Further properties. }
In the remainder of this section we derive a number of identities that are either conceptually important, or that are used elsewhere.
  
The following lemma restates relation (5) in terms of the pairings. 
It asserts that, up to a sign, $\boldzeta^*$ is the adjoint of $\boldzeta$ with respect to the pairings $\boldp_C$ and $\boldp_A$.

\begin{lemma}\label{lem:cozipper-dual-zipper}
Assuming relations (1) and (2), relation (5) is equivalent to the following relation in terms of the pairings $\boldp_C=(-1)^{|\boldlambda_C|}\boldeps_C\boldmu_C$ and $\boldp_A=(-1)^{|\boldlambda_A|}\boldeps_A\boldmu_A$:
$$
  \boldp_C(1\otimes\boldzeta^*) =  (-1)^{|\boldlambda_A|+|\boldlambda_C|} \boldp_A(\boldzeta\otimes 1). 
$$
\end{lemma}

\begin{proof}
Tensoring this relation with the identity from both sides and applying it to $\boldc_C\otimes \boldc_A$ yields the relation~(5). Conversely, tensoring relation~(5) with the identity from both sides and applying to it $\boldp_C\otimes \boldp_A$ yields the relation from the statement. These computations use the fourth group of identities in Proposition~\ref{prop:biunital-coFrob-bialg}. 
\end{proof}

The following lemma shows that the cozipper satisfies a property dual to relation (3). 

\begin{lemma}\label{lem:cozipper-intertwines-coproducts}
Assuming relations (1), (2), (3) and (5), the cozipper $\boldzeta^*:A\to C$ is a coalgebra map, i.e.: 

(i) It intertwines the coproducts in the sense of Definition~\ref{defi:compatibility_with_products}(ii), 
$$
(\boldzeta^*\otimes\boldzeta^*)\boldlambda_A = (-1)^{|\boldzeta^*| |\boldlambda_A|} \boldlambda_C \boldzeta^*.
$$
(ii) It preserves the counits, 
$$
\boldeps_A=\boldeps_C\boldzeta^*.
$$
\end{lemma}

\begin{figure}
\begin{center}
\includegraphics[width=.8\textwidth]{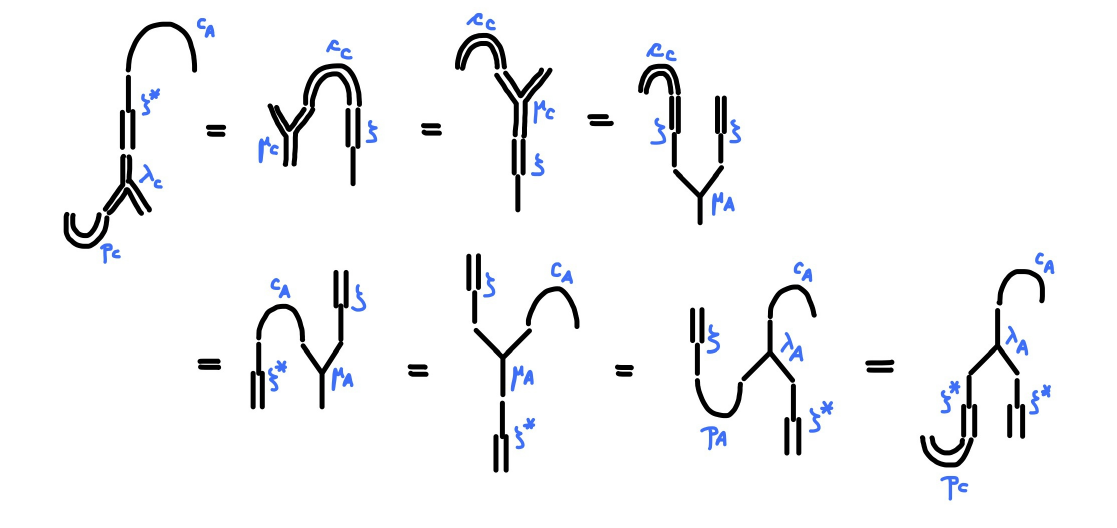}
\caption{Picture proof of Lemma~\ref{lem:cozipper-intertwines-coproducts}(i)}
\label{fig:cozipper-coproducts} 
\end{center}
\end{figure}

\begin{proof}
We give a picture proof of Lemma~\ref{lem:cozipper-intertwines-coproducts}(i) in Figure~\ref{fig:cozipper-coproducts}. In the operations depicted in that Figure  the closed string sector is represented by two lines, and the open string sector is represented by one line. Figure~\ref{fig:cozipper-coproducts} can serve as a red thread for the formal proof given below.  

To prove assertion (i) we show that $(\boldp_C\otimes 1\otimes 1)(1\otimes \boldlambda_C\boldzeta^*\otimes 1)(1\otimes \boldc_A) = (-1)^{|\boldzeta^*||\boldp_C|+|\boldzeta^*|} (\boldp_C\otimes 1\otimes 1)(1\otimes \boldzeta^*\otimes \boldzeta^*\otimes 1)(1\otimes \boldlambda_A\otimes 1)(1\otimes \boldc_A)$. Since $\boldp_C$ and $\boldc_A$ are perfect (Lemma~\ref{lem:cp-perfect}) and $|\boldzeta^*|=|\boldlambda_C|-|\boldlambda_A|$, $|\boldp_C|=-|\boldlambda_C|$, this is equivalent to the fact that $\boldzeta^*$ intertwines the coproducts. We compute
\begin{align*}
( & \boldp_C \otimes 1\otimes 1) (1\otimes \boldlambda_C\boldzeta^*\otimes 1)(1\otimes \boldc_A)\\
& = ((\boldp_C\otimes 1)(1\otimes \boldlambda_C)\otimes 1 )(1\otimes (\boldzeta^*\otimes 1)\boldc_A) \\
& = (-1)^{|\boldlambda_C|}(\boldmu_C\otimes 1)(1\otimes (1\otimes \boldzeta)\boldc_C) \\
& = (-1)^{|\boldlambda_C|} (1\otimes \boldzeta)(\boldmu_C\otimes 1)(1\otimes\boldc_C)\\
& = (-1)^{|\boldlambda_C|} (1\otimes \boldzeta)(1\otimes\boldmu_C)(\boldc_C\otimes 1) \\
& = (-1)^{|\boldlambda_C|} (1\otimes \boldzeta\boldmu_C)(\boldc_C\otimes 1) \\
& = (-1)^{|\boldlambda_C|} (1\otimes \boldmu_A(\boldzeta\otimes\boldzeta))(\boldc_C\otimes 1) \\
& = (-1)^{|\boldlambda_C|} (1\otimes\boldmu_A)((1\otimes \boldzeta)\boldc_C\otimes 1)\boldzeta \\
& = (-1)^{|\boldlambda_C|} (1\otimes\boldmu_A)((\boldzeta^*\otimes 1)\boldc_A\otimes 1)\boldzeta \\
& = (-1)^{|\boldlambda_C|} (\boldzeta^*\otimes 1)(1\otimes \boldmu_A)(\boldc_A\otimes 1)\boldzeta \\
& = (-1)^{|\boldlambda_C|} (\boldzeta^*\otimes 1)(\boldmu_A\otimes 1)(1\otimes \boldc_A)\boldzeta \\
& = (-1)^{|\boldlambda_C|+|\boldlambda_A|} (\boldzeta^*\otimes 1) ((\boldp_A\otimes1)(1\otimes\boldlambda_A)\otimes 1)(1\otimes \boldc_A)\boldzeta \\
& = (-1)^{|\boldlambda_C|+|\boldlambda_A|} (\boldzeta^*\otimes 1)(\boldp_A(\boldzeta\otimes 1)\otimes 1\otimes 1)(1\otimes\boldlambda_A\otimes 1)(1\otimes\boldc_A) \\
& = (\boldzeta^*\otimes 1)(\boldp_C(1\otimes\boldzeta^*)\otimes 1\otimes 1) (1\otimes\boldlambda_A\otimes 1)(1\otimes\boldc_A) \\
& = (-1)^{|\boldzeta^*||\boldp_C|+|\boldzeta^*|}(\boldp_C\otimes 1\otimes 1)(1\otimes\boldzeta^*\otimes\boldzeta^*\otimes1)(1\otimes\boldlambda_A\otimes 1)(1\otimes\boldc_A).
\end{align*}
Here in the 2nd and 4th equalities we use the biunital coFrobenius relations for $C$, and in the 2nd equality we also use condition~(5). In the 6th equality we use that $\boldzeta$ intertwines the products (relation~(3)). In the 8th equality we use relation~(5). In the 10th and 11th equality we use the biunital coFrobenius relation for $A$. In the 13th equality we use the reformulation of relation~(5) in terms of the pairings from Lemma~\ref{lem:cozipper-dual-zipper}. The sign in the last equality arises from the Koszul sign rule.

Assertion (ii) follows from Lemma~\ref{lem:intertwine-unitality}, though of course it can also be checked directly. 
\end{proof}

The previous result motivates our choice of sign $+1$ for the graded version of relation~(5). Our choice of sign for the graded version of the Cardy condition is motivated by Lemma~\ref{lem:normal-tangent} below. 

The next proposition and its corollary motivate the definition of the composition $\boldmu_A\boldlambda_A$ in Remark~\ref{rmk:graded_Cardy}.

\begin{proposition} \label{prop:muAlambdaA} Let $(A, \boldmu_A,\boldeta_A,\boldlambda_A,\boldeps_A)$ be a finite dimensional biunital coFrobenius bialgebra with $|\boldmu_A|=0$. Then 
$$
\boldmu_A\boldlambda_A = (-1)^{|\boldlambda_A|}\boldmu_A(1\otimes\boldmu_A\tau\boldlambda_A\boldeta_A) = (1\otimes\boldeps_A\boldmu_A\tau\boldlambda_A)\boldlambda_A.
$$
\end{proposition}

\begin{proof}
Using Definition~\ref{defi:coFrobenius-unital-bialgebra}, we compute 
\begin{align*}
(-1)^{|\boldlambda_A|} \boldmu_A(1\otimes\boldmu_A\tau\boldlambda_A\boldeta_A) 
& = (-1)^{|\boldlambda_A|} \boldmu_A (1\otimes\boldmu_A)(1\otimes \tau\boldlambda_A\boldeta_A) \\
& = (-1)^{|\boldlambda_A|} \boldmu_A (1\otimes\boldmu_A) (1\otimes \tau \boldc_A)\\
& = \boldmu_A (\boldmu_A\otimes 1) (1\otimes \boldc_A)\\
& = \boldmu_A\boldlambda_A.
\end{align*}
Here the second equality uses the definition of $\boldc_A$, the third one associativity of $\boldmu_A$ and symmetry of $\boldc_A$, and the last one the biunital coFrobenius relation for $\boldlambda_A$. 
Similarly, Definition~\ref{defi:coFrobenius-counital-bialgebra} yields
\begin{align*}
(1\otimes\boldeps_A\boldmu_A\tau\boldlambda_A)\boldlambda_A
& = (1\otimes\boldeps_A\boldmu_A\tau)(1\otimes\boldlambda_A)\boldlambda_A\\
& = (-1)^{|\boldlambda_A|}(1\otimes\boldp_A\tau)(1\otimes\boldlambda_A)\boldlambda_A\\
& = (1\otimes\boldp_A)(\boldlambda_A\otimes 1)\boldlambda_A \\
& = \boldmu_A\boldlambda_A.
\end{align*}
\end{proof}

\begin{corollary}\label{cor:muAlambdaA} Let $(A, \boldmu_A,\boldeta_A,\boldlambda_A,\boldeps_A)$ be a finite dimensional biunital coFrobenius bialgebra. Assume that $A$ is the open sector of a graded 2D open-closed TQFT such that the graded Cardy relation (6) holds with $|\boldlambda_C|=2|\boldlambda_A|$. Then
$$
\boldmu_A\boldlambda_A = \boldmu_A(1\otimes\boldzeta\boldzeta^*\boldeta_A) = (-1)^{|\boldlambda_A|}(1\otimes\boldeps_A\boldzeta\boldzeta^*)\boldlambda_A.
$$
\qed
\end{corollary}

We conclude this section with an algebraic consequence of Definition~\ref{defi:gradedTQFT} which is needed in~\cite{CHO-PD}.

\begin{lemma}
In a graded 2D open-closed TQFT the following additional relations hold:
\begin{align*}
  &(a)\qquad (\boldzeta\otimes 1)\boldc_C = (-1)^{|\boldlambda_C|+|\boldlambda_A|}(1\otimes\boldzeta^*)\boldc_A, \cr
  &(b)\qquad \boldmu_C(\boldzeta^*\otimes 1) = \boldzeta^*\boldmu_A(1\otimes\boldzeta). 
\end{align*}
\end{lemma}

Part (b) says that the cozipper $\boldzeta^*:A\to C$ intertwines the $C$-right module structures, where the right module structure $A$ is induced by the algebra map $\boldzeta:C\to A$. We depict this relation in Figure~\ref{fig:Gysin2-new}, where the middle term represents -- as in Figure~\ref{fig:Gysin1-new} -- the reason why the left and right terms coincide: they both describe the operation determined by an annulus with one interior positive puncture, one positive puncture on one boundary component and one negative puncture on the other boundary component. 

\begin{figure} [ht]
\centering
\includegraphics[width=\textwidth]{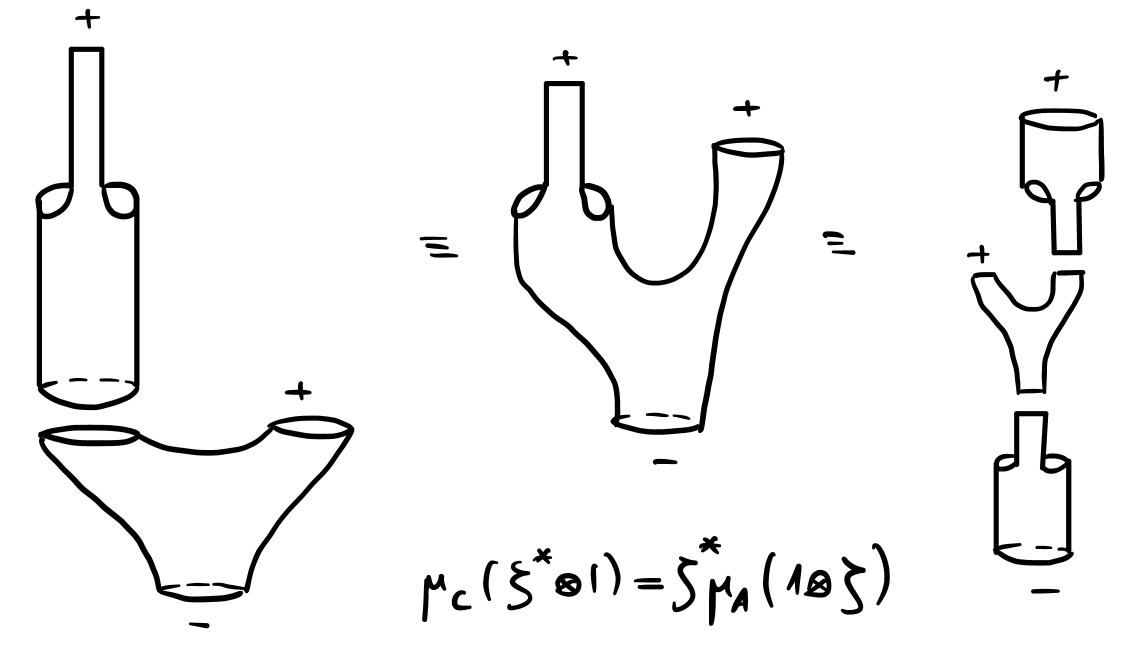}
\caption{The cozipper intertwines the $C$-right module structures.}
\label{fig:Gysin2-new}
\end{figure}

\begin{proof}
  Part (a) follows easily from relation (5) in Definition~\ref{defi:gradedTQFT} and the symmetry of $\boldc_C$ and $\boldc_A$. To prove (b) we tensor both sides with $1$ and apply them to $1\otimes\boldc_C$.
The left hand side gives
\begin{align*}
  \Bigl([\boldmu_C(\boldzeta^*\otimes 1)]\otimes 1\Bigr)(1\otimes\boldc_C) 
  &= (-1)^{|\boldc_C||\boldzeta^*|}(\boldmu_C\otimes 1)(1\otimes\boldc_C)\boldzeta^* \cr
  &= (-1)^{|\boldc_C||\boldzeta^*|}\boldlambda_C\boldzeta^* \cr
  &= (-1)^{|\boldc_C||\boldzeta^*|+|\boldlambda_A||\boldzeta^*|}(\boldzeta^*\otimes\boldzeta^*)\boldlambda_A,
\end{align*}
where the second equality follows from the {\sc (unital coFrobenius)} relation in $C$ and the third one from Lemma~\ref{lem:cozipper-intertwines-coproducts}. The right hand side gives 
\begin{align*}
  &\Bigl([\boldzeta^*\boldmu_A(1\otimes\boldzeta)]\otimes 1\Bigr)(1\otimes\boldc_C) \cr
  &= (\boldzeta^*\boldmu_A\otimes 1)(1\otimes\boldzeta\otimes 1)(1\otimes\boldc_C) \cr
  &= (-1)^{|\boldlambda_C|+|\boldlambda_A|}(\boldzeta^*\boldmu_A\otimes 1)(1\otimes 1\otimes\boldzeta^*)(1\otimes\boldc_A) \cr
  &= (-1)^{|\boldlambda_C|+|\boldlambda_A|}(\boldzeta^*\otimes\boldzeta^*)(\boldmu_A\otimes 1)(1\otimes\boldc_A) \cr
  &= (-1)^{|\boldlambda_C|+|\boldlambda_A|}(\boldzeta^*\otimes\boldzeta^*)\boldlambda_A,
\end{align*}
where the second equality follows from part (a) and the fourth one from the {\sc (unital coFrobenius)} relation in $A$. Because of $|\boldc_C|=|\boldlambda_C|$ and $|\boldzeta^*|=|\boldlambda_C|-|\boldlambda_A|$ the two sides agree, which in view of perfectness of $\boldc_C$ proves (b).
\end{proof}

{\bf Morphisms of graded 2D open-closed TQFT.} 
Viewing a 2D open-closed TQFT as a monoidal functor from the category of 1+1 cobordisms to the category of finite dimensional $\bk$-vector spaces as in~\cite{Lauda-Pfeiffer}, one can define a morphism to be a natural transformation of such functors. Spelling this out, a morphism between graded 2D open-closed TQFTs $(C,A)$ and $(C',A')$ consists of the data of two linear continuous maps $C\to C'$ and $A\to A'$ that intertwine all the structures, i.e., two morphisms between the Tate Frobenius algebras that underlie the closed sector, respectively the open sector. By Proposition~\ref{prop:morphisms}, these maps are topological isomorphisms. As a consequence, any such morphism between graded 2D open-closed TQFT is an isomorphism.

\section{Homology of manifolds}\label{sec:manifolds}

In this section we show how the algebraic structures introduced in the previous sections occur on the (co)homology of oriented manifolds. Using throughout coefficients in a discrete field $\bk$, we will prove

\begin{proposition}\label{prop:noncompact-manifold}
Let $M$ be a (possibly noncompact) oriented manifold of dimension $n$. Then $H_{n-*}(M)$ is discrete and carries the natural structure of a commutative and cocommutative counital infinitesimal anti-symmetric bialgebra. Its dual is the linearly compact space $H^*(M)$ with its natural structure of a commutative and cocommutative unital infinitesimal anti-symmetric bialgebra. 
\end{proposition}

\begin{remark}
  To view the homology of a noncompact manifold as a discrete space, and its cohomology as a linearly compact space, is a perspective that was pioneered by Lefschetz in his book~\cite{Lefschetz-book}. The previous two structures are related by topological duality. Inserting Poincaré duality into the picture we obtain the commutative diagram below. In this diagram $H^*_c(M)$ denotes \emph{cohomology with compact support}, generated by cocycles with compact support, and $H_*^{BM}(M)$ denotes \emph{Borel-Moore homology}, generated by locally finite cycles. This diagram is implicit in Lefschetz' duality theorem for star- or closure-finite complexes~\cite[III.8.(41.2)]{Lefschetz-book}, where it is established without any reference to algebraic structures.  
$$
\xymatrix{
H_*(M) \ar[d]_{PD}^\simeq \ar@{<->}[rr]^{\mathrm{topological}}_{\mathrm{duality}} & & H^*(M) \ar[d]_{PD}^\simeq \\
H^*_c(M) \ar@{<->}[rr]^{\mathrm{topological}}_{\mathrm{duality}} & & H_*^{BM}(M)
}
$$  
The spaces $H_*(M)$ and $H^*_c(M)$ in the left column are discrete counital infinitesimal anti-symmetric bialgebras, and the spaces $H^*(M)$ and $H_*^{BM}(M)$ in the right column are linearly compact unital infinitesimal anti-symmetric bialgebras (all commutative and cocommutative). Poincaré duality is a topological isomorphism that preserves the algebraic structures, whereas topological duality exchanges counital with unital, and discrete with linearly compact.  
\end{remark}

For closed oriented manifolds, both homology and cohomology are finite dimensional and the following statement may be more familiar. 

\begin{proposition}\label{prop:closed-manifold}
Let $M$ be a closed oriented $n$-dimensional manifold. Then $H_{n-*}(M)$ and $H^*(M)$ carry natural structures of commutative and cocommutative biunital coFrobenius bialgebras.
\end{proposition}

\begin{proposition}\label{prop:closed-manifold-pair}
Let $M$ be a closed oriented $n$-dimensional manifold and $Z\subset M$ a closed oriented $d$-dimensional submanifold. Assume in addition the following conditions: 
\begin{itemize}
\item if $n=2d$ the normal bundle $\nu Z$ and the tangent bundle $TZ$ have the same Euler class; 
\item if $n<2d$ the normal bundle $\nu Z$ and the tangent bundle $TZ$ have vanishing Euler classes; 
\item if $n>2d$ the tangent bundle $TZ$ has vanishing Euler class.
\end{itemize}
Then $C=H^*(M)$ and $A=H^*(Z)$ fit into a graded 2D open-closed TQFT. 
\end{proposition}

\begin{remark}
(a) In the case $n=2d$ the condition on the Euler classes in Proposition~\ref{prop:closed-manifold-pair} is satisfied if $M$ is symplectic and $Z\subset M$ is Lagrangian. Indeed, in that situation the normal bundle $\nu Z$ is isomorphic to the tangent bundle $TZ$. The equality of the Euler classes can be understood as a topological avatar of being a Lagrangian submanifold, and we find it remarkable that it arises from the algebra defining a graded 2D open-closed TQFT.    

(b) For $n>2d$ the Euler class of the normal bundle $\nu Z$ always vanishes for degree reasons. 
The following example of this case will be relevant later. Let $M=S^*Q$ be the unit cosphere bundle of a closed $m$-dimensional manifold $Q$, and $Z=S^*_qQ$ the unit sphere in the cotangent fiber at a point $q\in Q$. The Euler class of the tangent bundle $TZ$ vanishes with arbitrary coefficients if $m$ is even, and it vanishes with $\Z/2$-coefficients if $m$ is odd.
\end{remark}

\begin{remark} \label{rmk:TQFThomology}
By Poincaré duality $H^{-*}(M)\simeq H_{n+*}(M)$ and also $H^{-*}(Z)\simeq H_{d+*}(Z)$. Proposition~\ref{prop:closed-manifold-pair} induces a graded 2D open-closed TQFT structure on shifted homology $(H_{n+*}(M),H_{d+*}(Z))$ with products given by the intersection products on $M$ and $Z$ (see below).
\end{remark}

\subsection*{Bialgebra structures}
In this subsection we will first prove Proposition~\ref{prop:closed-manifold} and then derive from it Proposition~\ref{prop:noncompact-manifold}. 

Let $M$ be a closed oriented $n$-dimensional manifold. The cohomology $H^*(M)$ is finite dimensional and carries the following natural operations:
\begin{itemize}
\item the cup product $\boldmu:H^*(M)\otimes H^*(M)\to H^*(M)$, 
$$
\boldmu(\alpha\otimes\beta)=\alpha\cup\beta,
$$ 
which is associative and commutative of degree $0$ with unit 
$$
\boldeta=1\in H^0(M);
$$
\item the linear map $\boldeps:H^*(M)\to\bk$ of degree $-n$ given by evaluation on the fundamental class $[M]\in H_n(M)$, i.e., 
$$
\boldeps(\alpha)=\la\alpha,[M]\ra.
$$ 
\end{itemize}
The associated pairing $\boldp=(-1)^n\boldeps\boldmu$, 
$$ 
\boldp(\alpha,\beta)=\la[M],\alpha\cup\beta\ra
$$ 
is perfect by Poincar\'e duality, and the triple product $\boldp(\boldmu\otimes 1)(\alpha,\beta,\gamma)=\la[M],\alpha\cup\beta\cup\gamma\ra$ is cyclically (in fact, fully) symmetric. In other words, $H^*(M)$ is a commutative 
Tate Frobenius algebra in the sense of Definition~\ref{def:graded-symm-Frob}.

\begin{proof}[Proof of Proposition~\ref{prop:closed-manifold}] Since $H^*(M)$ is a commutative 
Tate Frobenius algebra, Proposition~\ref{prop:cyclic-graded-algebra} implies that the operations $\boldmu,\boldeta,\boldeps$ induce on the finite dimensional space $H^*(M)$ the structure of a commutative and cocommutative biunital coFrobenius bialgebra. 

The statement about $H_{n-*}(M)$ follows by algebraic duality. 
\end{proof}

The proof of Proposition~\ref{prop:noncompact-manifold} will use the {\em doubling construction}: for a compact oriented manifold $M$ with boundary, let $\wh M=M\cup_{\p M}M^{\rm op}$ be the closed oriented manifold obtained by gluing to $M$ another copy of $M$ with opposite orientation along $\p M$. 

\begin{lemma}\label{lem:doubling}
The map $\iota_*:H_*(M)\to H_*(\wh M)$ induced by the inclusion $\iota:M\into\wh M$ is injective. 
\end{lemma}

\begin{proof}
Represent $\alpha\in H_k(M)$ by a smooth singular $k$-cycle $a$ contained in the interior of $M$. Suppose that $\iota_*\alpha=0$, so $a=\p b$ for a singular $(k+1)$-chain $b$ in $\wh M$. We can choose $b$ to be smooth and transverse to $\p M$ (i.e., each simplex in $b$ is transverse to $\p M$). After subdivision, we can then write $b=b_1+b_2$ as the sum of two $(k+1)$-chains with $b_1\subset M$ and $b_2\subset M^{\rm op}$. Let $I:\wh M\to\wh M$ be the involution exchanging $M$ and $M^{\rm op}$. Then $I(b_2)\subset M$ and $\p I(b_2)=-\p b_2$, hence $\wt b=b_1-I(b_2)\subset M$ and $\p\wt b=\p b_1+\p b_2=a$, and therefore $\alpha=0$ in $H_*(M)$.  
\end{proof}

\begin{proof}[Proof of Proposition~\ref{prop:noncompact-manifold}]
Let $M$ be an oriented $n$-dimensional manifold which is not necessarily compact. 
It can be written as a union $M=\cup_{k\in\N}M_k$ of compact submanifolds with boundary with $M_k\subset \Int M_{k+1}$ for all $k$. Since each $H_*(M_k)$ is finite dimensional, $H_*(M)=\colim H_*(M_k)$ is discrete and its dual $H^*(M)=\lim H^*(M_k)$ is linearly compact. 

By the discussion earlier in this section, $H_*(M)$ carries the following natural operations:
\begin{itemize}
\item the intersection product $\boldmu:H_*(M)\otimes H_*(M)\to H_*(M)$, which is associative and commutative of degree $-n$;
\item the homology coproduct $\boldlambda=\Delta_*:H_*(M)\to H_*(M)\otimes H_*(M)$, which is coassociative and cocommutative of degree $0$ with counit $\boldeps:H_0(M)\to\bk$.
\end{itemize}
Here the intersection product can be defined in two equivalent ways: directly in terms of transverse intersections of smooth cycles, or indirectly in terms of the cup product and Poincar\'e duality applied to the $M_k$ in an exhaustion $M=\cup_{k\in\N}M_k$ as above.

We need to show that these operations satisfy the {\sc (counital infinitesimal relation)} from Definition~\ref{defi:secondary-counital}. For this, it suffices to prove that the relation holds when evaluated on each element $a\otimes b\in H_*(M)\otimes H_*(M)$. There exists $k$ such that $a,b$ can be represented by cycles in $M_k$ from the above exhaustion. Consider its double $\wh M_k$ defined above. By Remark~\ref{rmk:TQFThomology}, the operations $\boldmu,\boldlambda,\boldeps$ on $H_{*+n}(\wh M_k)$ are part of a biunital coFrobenius algebra structure and thus satisfy the {\sc (counital infinitesimal relation)}. By Lemma~\ref{lem:doubling}, the relation continues to hold on $H_{*+n}(M_k)$, and therefore on the element $a\otimes b$, because the operations $\boldmu$, $\boldlambda$, $\boldeps$ are of a geometric nature and leave invariant the image of $H_{*+n}(M_k)$ inside $H_*(M)$.
This proves the assertion on $H_{n-*}(M)$, and the one on $H^*(M)$ follows by duality. 
\end{proof}

\subsection*{Graded open-closed 2D TQFT structure} 
In order to prove\break Proposition~\ref{prop:closed-manifold-pair} we need to spell out in more detail the biunital coFrobenius bialgebra structure from Proposition~\ref{prop:closed-manifold}. Let $M$ be a closed oriented $n$-dimensional manifold.

We abbreviate $H^i = H^i(M)$ and $H_i = H_ i(M)$. 
We then have $H^i \simeq (H_i)^\vee$ and, since the cohomology groups are finite dimensional, we also have $H_i \simeq (H^i)^\vee$. Given a homology class $A$ and a cohomology class $\alpha$ of the same degree, we define $\langle A,\alpha\rangle = (-1)^{|A| |\alpha|} \langle \alpha,A\rangle$. We shall denote in the sequel by $\ev$ both the Kronecker pairing $H^i \otimes H_i \to \bk$ and its transposition $H_i\otimes H^ i\to \bk$. We denote $[M]\in H_n$ the fundamental class of $M$ and $\int_M\alpha = \langle \alpha,[M]\rangle$ if $\alpha$ has degree $n$. 

We use the following conventions for the cup product, cap product and Poincar\'e duality, mainly following~\cite{Dold}. The starting point is the \emph{homology coproduct}, given by the Alexander-Whitney diagonal $\Delta_*:H_* \to H_* \otimes H_*$~\cite[VI.12]{Dold}. It is a fundamental fact that $\Delta$ is cocommutative. 
\begin{itemize}
\item The \emph{cup product} $\cup:H^i \otimes H^j \to H^{i+j} $ is algebraically dual to $\Delta_*$, 
i.e. $\ev( \cup \otimes 1) = (\ev\otimes \ev) (1\otimes\tau\otimes 1)(1\otimes 1\otimes \Delta_*)$. See~\cite[VII.8]{Dold}.
\item The \emph{cap product} defines a bimodule structure on $H_*$ over $H^*$. The left module structure $\cap_{\mathrm{left}}:H^j\otimes H_i\to H_{i-j}$ is defined in~\cite[VII.12]{Dold} as $\cap_{\mathrm{left}}=(\ev\otimes 1)(1\otimes \tau\Delta_*)$. We define the right module structure $\cap_{\mathrm{right}}:H_i\otimes H^j\to H_{i-j}$ by transposing $\cap_{\mathrm{left}}$, i.e. $A\cap_{\mathrm{right}} \alpha = (-1)^{|A||\alpha|}\alpha\cap_{\mathrm{left}} A$, which makes sense because the cohomology ring is commutative. 

In the sequel we simply write $\cap$ for both $\cap_{\mathrm{right}}$ and $\cap_{\mathrm{left}}$, and the meaning will be clear from the context. If $A$ and $\alpha$ have the same degree then $\alpha\cap A=\langle \alpha,A\rangle[\pt]$ and 
$A\cap \alpha = \langle A,\alpha\rangle [\pt]$, where $[\pt]$ denotes the class of a point. 
\item We define the \emph{Poincar\'e duality isomorphism} $PD:H^i\stackrel\simeq\longrightarrow H_{n-i}$ as 
$$
PD(\alpha) = [M]\cap \alpha. 
$$
\end{itemize} 

We are now ready to describe the coFrobenius structure on $C=H^*$. We denote $C^\vee_i=H_{-i}$. 
\begin{itemize}
\item The product $\boldmu$, the unit $\boldeta$, the counit $\boldeps$ and the pairing $\boldp$ were described in the preamble of the proof of Proposition~\ref{prop:closed-manifold}. The degree $-n$ map $\vec\boldp:C\to C^\vee$ associated to $\boldp$ is   
$$
\vec\boldp=PD.
$$ 
\item \emph{The coproduct $\boldlambda$} is the \emph{Poincar\'e dual} of the homology coproduct $\Delta_*$, i.e. the diagram below is commutative (the degrees of the maps are indicated on the arrows) 
$$
\xymatrix
@C=50pt
{
H^k \ar[r]^-\boldlambda_-{n} \ar[d]_-{\vec\boldp=PD}^-{-n} & \oplus_{i+j=k+n} H^i\otimes H^j \ar[d]^{\vec\boldp\otimes \vec\boldp=PD\otimes PD}_{-2n} \\
H_{n-k}\ar[r]_-{\Delta_*}^-0& \oplus_{i+j=k+n} H_{n-i}\otimes H_{n-j}.
}
$$
Since $\Delta_*$ is cocommutative ($\tau\Delta_*=\Delta_*$) we obtain that $\boldlambda$ is also cocommutative ($\tau\boldlambda=(-1)^{|\boldlambda|}\boldlambda$). For the computation one uses that $\tau(\vec\boldp\otimes \vec\boldp)=(-1)^{|\vec\boldp|} (\vec\boldp\otimes \vec\boldp)\tau$, and $|\vec\boldp|=-|\boldlambda|=-n$. Also, coassociativity of $\Delta_*$ implies coassociativity of $\boldlambda$. The map $\boldeps$ is the counit for $\boldlambda$.  
\item The \emph{copairing $\boldc=\boldlambda\boldeta\in \oplus_{i+j=n} H^i\otimes H^j$} is, by definition of $\boldlambda$, the preimage via $PD\otimes PD$ of the Alexander-Whitney diagonal $\Delta_*([M])\in H_*\otimes H_*$. Viewed differently, it is the image of $1\in H^0$ by the shriek map $H^0\to H^n(M\times M)\cong \oplus_{i+j=n} H^i\otimes H^j$. 
\end{itemize}

Let now $Z\subset M$ be a closed oriented submanifold of dimension $d$. Proposition~\ref{prop:closed-manifold} provides two commutative and cocommutative biunital coFrobenius bialgebras with underlying vector spaces $C=H^*(M)$ and $A=H^*(Z)$, denoted $(C,\boldmu_C,\boldlambda_C,\boldeta_C,\boldeps_C)$ and $(A,\boldmu_A,\boldlambda_A,\boldeta_A,\boldeps_A)$. Here we have $|\boldlambda_C|=n$ and $|\boldlambda_A|=d$. 

Denote $incl:Z\hookrightarrow M$ the inclusion and $\vec\boldp_C=PD_C:C\to C^\vee$, $\vec\boldp_A=PD_A:A\to A^\vee$ the Poincar\'e duality isomorphisms for $M$ and $Z$ respectively. We define the zipper and the cozipper as follows: 
\begin{itemize}
\item The \emph{zipper} $\boldzeta:C\to A$ is defined as the map induced by $incl$ in cohomology, 
$$
\boldzeta=incl^*:H^*(M)\to H^*(Z).
$$
\item The \emph{cozipper} $\boldzeta^*:A\to C$ is defined as the shriek map induced by $incl$ in cohomology, multiplied by an extra sign $(-1)^{n-d}$, 
$$
\boldzeta^*=(-1)^{n-d} incl^!:H^*(Z)\to H^{*+n-d}(M).
$$
By definition, the shriek map $incl^!$ is determined by the identity $PD_C\circ incl^!= incl_*\circ PD_A$, so that $\boldzeta^*$ is such that the diagram below commutes up to a sign $(-1)^{n-d}$, 
$$
\xymatrix
@C=10pt
@R=10pt
{
H^*(Z)\ar[rr]^-{\boldzeta^*}\ar[dd]_-{PD_A} & & H^{*+n-d}(M) \ar[dd]^-{PD_C}\\
& (-1)^{n-d} & \\
H_{d-*}(Z) \ar[rr]_-{incl_*} & & H_{d-*}(M).
}
$$ 
The sign $(-1)^{n-d}$ is motivated by the proof of relation~(5) below. 
\end{itemize}

Proposition~\ref{prop:closed-manifold-pair} is a consequence of the following three lemmas.

\begin{lemma}
The above structure satisfies relations (1)-(5) from Definition~\ref{defi:gradedTQFT} of a graded 2D open-closed TQFT.
\end{lemma}

\begin{proof}
That $(C,\boldmu_C,\boldlambda_C,\boldeta_C,\boldeps_C)$ and $(A,\boldmu_A,\boldlambda_A,\boldeta_A,\boldeps_A)$ are biunital coFrobenius bialgebras has been already noted above. They are both commutative and cocommutative, so relation~(4) is void. That $\boldzeta=incl^*$ is an algebra homomorphism is standard, hence relation~(3) is satisfied. Finally, to check relation~(5) we compute 
$$
\boldp_A(1\otimes\boldzeta) = \ev(\vec\boldp_A\otimes incl^*) = \ev(incl_*\vec\boldp_A\otimes 1),
$$ 
$$
\boldp_C(1\otimes\boldzeta^*)=\boldp_C(\boldzeta^*\otimes 1)=\ev(\vec\boldp_C\boldzeta^*\otimes 1).
$$
Relation~(5) follows from $\vec\boldp_C\boldzeta^*=(-1)^{n-d}incl_*\vec\boldp_A$ and Lemma~\ref{lem:cozipper-dual-zipper}.
\end{proof}

\begin{lemma} \label{lem:normal-tangent}
Assume $n=2d$. Then the Cardy condition~(6) is equivalent to the equality between the Euler class of $TZ$ and the Euler class of the normal bundle of $Z$ in $M$.
\end{lemma}

The proof uses properties of the Thom class, which we single out. We follow the conventions of  Bredon~\cite{Br}.  
\begin{itemize}
\item \emph{Normal bundle}. We denote $\nu Z=TM|_Z/TZ$ the normal bundle to $Z$ in $M$. The orientations of $Z$ and $M$ determine an orientation of $\nu Z$ from the short exact sequence $0\to TZ\to TM|_Z\to \nu Z\to 0$, i.e. a positive basis of $TZ$, followed by a lift to $TM$ of a positive basis in $\nu Z$, is a positive basis of $TM$.  

\item \emph{Thom class} (cf. Bredon~\cite[VI.11]{Br}). Let $\pi:W\to B$ be an oriented $r$-disc bundle and choose a section $incl:B\to W$ which does not intersect the boundary. The Thom class $\tau\in H^r(W,\p W)$ is the unique class which restricts to the positive generator of $H^r(D^r,\p D^r)$ in the fiber. Assuming that $B$ is a closed oriented manifold, we orient $W$ so that the fiber orientation becomes the normal orientation of the section. The Thom class satisfies $\tau\cap [W]=incl_*[B]$, where $[W]$ is the fundamental class rel boundary for $W$. 

\item \emph{Euler class} (cf. Bredon~\cite[VI.12]{Br}). Given an oriented $r$-disc bundle $\pi:W\to B$ with a section $incl:B\to W$, the \emph{Euler class} is the pull-back of the Thom class, $e(W)=incl^*\tau$. For a closed oriented $n$-manifold $M$ we have $e(TM)=\chi(M)\omega$, with $\omega\in H^n(M)$ the orientation class (Bredon~\cite[VI.12.4]{Br}).    

\item \emph{Thom class and shriek map}. For a compact oriented $n$-manifold $M$ (possibly with boundary), let $\widetilde{PD}=\widetilde{PD}_M:H^*(M,\p M)\stackrel\simeq\longrightarrow H_{n-*}(M)$ be the isomorphism given by $\widetilde{PD}(\alpha)=\alpha\cap [M]$ (thus $\widetilde{PD}(\alpha)=(-1)^{n|\alpha|} PD(\alpha)$ with our previous convention for Poincar\'e duality). 

In the setup of an oriented $r$-disc bundle $\pi:W\to B$ over a closed oriented base $B$ with section $incl:B\to W$, we define ${\widetilde{incl}}^!=(\widetilde{PD}_W)^{-1} incl_* \widetilde{PD}_B:H^*(B)\to H^{*+r}(W,\p W)$. Then ${\widetilde{incl}}^!$ is an isomorphism and equals the composition $H^*(B)\stackrel{\pi^*}\longrightarrow H^*(W)\stackrel{\cup\tau}\longrightarrow H^{*+r}(W,\p W)$ (Bredon~\cite[VI.11.3]{Br}). Note that, on $H^i(B)$, we have ${\widetilde{incl}}^!=(-1)^{i \dim B + (i+r) (\dim B + r)} incl^!$ with our previous convention for $incl^!$.   
\end{itemize} 

\begin{proof}[Proof of Lemma~\ref{lem:normal-tangent}]
Since $\boldmu_A$ is commutative, the Cardy condition is equivalent to 
$$
\boldzeta \,\boldzeta^* = (-1)^{|\boldlambda_A|}\boldmu_A \boldlambda_A.
$$
Both these maps act from $H^*(Z)$ to $H^{*+d}(Z)$ and hence vanish in positive degrees. The Cardy condition is therefore equivalent to 
$$
\boldzeta \,\boldzeta^*(1) = (-1)^d\boldmu_A \boldlambda_A(1),
$$ 
with $1\in H^0(Z)$ the unit. 

On the one hand we have 
$$
\boldzeta \,\boldzeta^*(1)=(-1)^{n-d} incl^* \, incl^! (1) = (-1)^{n-d} incl^* \tau = (-1)^d e(\nu Z).
$$
The second equality uses that $incl^!={\widetilde{incl}}^!$ on $H^0$ if $n$ is even and ${\widetilde{incl}}^!(1)=\tau$, the Thom class of the normal bundle $\nu Z$ extended as a cohomology class in $M$. This was explained in the previous paragraph. The third equality uses that $n=2d$ and the definition of $e(\nu Z)$. 

On the other hand we claim that 
$$
\boldmu_A \boldlambda_A(1)= e(TZ).
$$ 
To see this we write $\boldmu_A=\Delta^*\times$, where $\Delta:Z\to Z\times Z$ is the diagonal and $\times:H^i(Z)\otimes H^j(Z)\to H^{i+j}(Z\times Z)$ is the cohomological cross-product (Bredon~\cite[VI.1-4]{Br}). These maps fit into the diagram below, where we denote $H^*=H^*(Z)$ and $H_*=H_*(Z)$ and $\times$ on the lower arrow denotes the homological cross product:
$$
\xymatrix
@C=30pt
{
H^k \ar[r]^-{\boldlambda_A} \ar[d]_{PD_Z} & \displaystyle \bigoplus_{i+j=k+d} H^i\otimes H^j \ar[d]^{PD_Z\otimes PD_Z} \ar[r]^-\times & H^{k+d}(Z\times Z) \ar[r]^-{\Delta^*}\ar[d]^{PD_{Z\times Z}} & H^{k+d} \\
H_{n-k}\ar[r]_-{\Delta_*}& \displaystyle \bigoplus_{i+j=k+d} H_{d-i}\otimes H_{d-j} \ar[r]_-\times & H_{d-k}(Z\times Z) & 
}
$$  
The first square is commutative by definition. To prove that the second square is commutative we use that $[Z]\times[Z]=[Z\times Z]$. On the upper and right sides we read $PD_{Z\times Z}(\alpha\times\beta)= [Z\times Z]\cap (\alpha\times\beta) = ([Z]\times [Z])\cap (\alpha\times\beta) = (-1)^{d|\alpha|}([Z]\cap \alpha)\times ([Z]\cap \beta)$. On the left and lower sides we read $\times((PD\otimes PD)(\alpha\otimes\beta))=(-1)^{d|\alpha|}PD(\alpha)\times PD(\beta)=(-1)^{d|\alpha|}([Z]\cap \alpha)\times ([Z]\cap \beta)$. 
We obtain 
$$
\boldmu_A\boldlambda_A(1) = \Delta^*PD_{Z\times Z}^{-1}[\Delta],
$$
where $[\Delta]\in H_d(Z\times Z)$ is the class of the diagonal. Let $\{\alpha\}$ be a basis for $H^*(Z)$ and $\{\alpha^o\}$ a dual basis, i.e. $\alpha^o\cup \beta=\delta_{\alpha,\beta}\omega$ with $\omega\in H^d(Z)$ the orientation class. Bredon~\cite[VI.12.4]{Br} shows that $\widetilde{PD}_{Z\times Z}^{-1}[\Delta]=\sum_\alpha (-1)^{|\alpha|}\alpha^o\times\alpha$. On the other hand $PD_{Z\times Z}=\widetilde{PD}_{Z\times Z}$ and we obtain $\Delta^*PD_{Z\times Z}^{-1}[\Delta] = \sum_\alpha (-1)^{|\alpha|}\alpha^o\cup \alpha = (\sum_\alpha (-1)^{|\alpha|})\omega = e(TZ)$. 

The Cardy condition is therefore equivalent to $e(\nu Z)=e(TZ)$. 
\end{proof}

\begin{lemma} \label{lem:normal-tangent-new}
(a) Assume $n<2d$. Then the Cardy condition~(6) is equivalent to the vanishing of the Euler classes of the tangent bundle $TZ$ and of the normal bundle $\nu Z$.

(b) Assume $n>2d$. Then the Cardy condition~(6) is equivalent to the vanishing of the Euler class of the tangent bundle $TZ$.
\end{lemma}

\begin{proof}
In both cases (a) and (b) the Cardy condition writes $\boldzeta\boldzeta^*=\boldmu_A\tau\boldlambda_A=0$. In view of the commutativity of $\boldmu_A$, this is equivalent to 
$$
\boldzeta\boldzeta^*=\boldmu_A\boldlambda_A=0.
$$

The condition $\boldmu_A\boldlambda_A=0$ is equivalent in both cases (a) and (b) to the vanishing of the Euler class of $TZ$. Indeed, we have seen in the proof of Lemma~\ref{lem:normal-tangent} that the only potentially nonvanishing component of $\boldmu_A\boldlambda_A:H^*(Z)\to H^{*+d}(Z)$ is $\boldmu_A\boldlambda_A(1)=e(TZ)$. 

The condition $\boldzeta\boldzeta^*=0$ is trivially satisfied in case (b) for dimensional reasons: the map acts as $\boldzeta\boldzeta^*:H^*(Z)\to H^{*+n-d}(Z)$, and the cohomology $H^*(Z)$ is supported in degrees $0,\dots,d$. 

The condition $\boldzeta\boldzeta^*=0$ is equivalent in case (a) to the vanishing of the Euler class of $\nu Z$. Indeed, by definition $(-1)^{n-d}\boldzeta\boldzeta^*=incl^*incl^!$ and we have seen in the proof of Lemma~\ref{lem:normal-tangent} that $incl^*incl^!$ acts on $H^i(Z)$ as $incl^*incl^!(\alpha)=\eps(i)\alpha\cup e(\nu Z)$, where $\eps(i)=(-1)^{id+(i+n-d)n}$. This map vanishes identically if and only if it vanishes on $1\in H^0(Z)$, where it is equal to $(-1)^{(n-d)n}e(\nu Z)$. 
\end{proof}

\section{Quantum homology}\label{sec:quantum}

This section is expository. Our purpose is to illustrate the graded 2D open-closed TQFT structure arising from quantum homology for pairs $(M,L)$ consisting of a closed symplectic manifold $M$ and a closed Lagrangian submanifold $L\subset M$. 
The algebras encountered in this section will be of finite type over the ground field. 
We follow closely the exposition of Biran-Cornea~\cite{Biran-Cornea-rigidity-uniruling,Biran-Cornea-Lagrangian-topology} and McDuff-Salamon~\cite[Chapter~11]{MS04}. 

Let $(M^{2n},\omega)$ be a closed symplectic manifold and $L^n\subset M^{2n}$ a closed Lagrangian submanifold. We assume the following:
\begin{itemize}
\item $L$ is \emph{monotone}, i.e., there exists $\tau>0$ such that, denoting $\omega:\pi_2(M,L)\to \R$ the symplectic area morphism and $\mu:\pi_2(M,L)\to\Z$ the Maslov index morphism, we have $\omega=\tau \mu$. 

\item The \emph{minimal Maslov number}
$$
N_L=\min \, \{\mu(\alpha)\, : \, \alpha\in\pi_2(M,L),\quad \mu(\alpha)>0\}
$$
satisfies $N_L\ge 2$. 

\item $L$ is orientable ($w_1(TL)=0$) and spin ($w_2(TL)=0$), with a \emph{fixed orientation and spin structure}. 
\end{itemize}

Since $\mu$ changes by $2c_1(TM)$ upon modifying a relative homotopy class by an element of $\pi_2(M)$, the manifold $M$ is necessarily monotone with monotonicity constant $2\tau$, i.e., $[\omega]|_{\pi_2(M)}=2\tau c_1(TM)|_{\pi_2(M)}$. Define the \emph{minimal Chern number} by
$$
C_M=\min \, \{c_1(\alpha)\, : \, \alpha\in\pi_2(M),\quad c_1(\alpha)>0\}.
$$
Then the minimal Maslov number $N_L$ divides $2C_M$.

\begin{remark} The assumptions that $L$ is monotone and $N_L\ge 2$ ensure that quantum homology of $M$, denoted $QH_*(M)$, and quantum homology of $L$, denoted $QH_*(L)$, are defined, the first over the ground ring $\Z$ and the second over the ground ring $\Z/2$. The assumption that $L$ is oriented and spin ensures that $QH_*(L)$ is defined over $\Z$. These assumptions are not the minimal ones that guarantee these properties~\cite{FOOO1,FOOO2}, but they are technically the easiest to work with. \end{remark}

In the sequel we use coefficients in a ground \emph{field} $\bk$. Lagrangian quantum homology $QH_*(L)$ is canonically a module over the graded ring $\Lambda=\bk[t,t^{-1}]$ with $|t|=-N_L$, whereas $QH_*(M)$ is canonically a module over the graded ring $\Gamma=\bk[s,s^{-1}]$ with $|s|=-2C_M$. Using the embedding of rings $\Gamma\hookrightarrow \Lambda$, $s\mapsto t^{2C_M/N_L}$ we view $\Lambda$ as a $\Gamma$-module and subsequently define $QH_*(M;\Lambda)=QH_*(M)\otimes_\Gamma\Lambda$. We further shift gradings to define 
$$
Q\H_*(M;\Lambda)=QH_{*+2n}(M;\Lambda),\qquad Q\H_*(L)=QH_{*+n}(L).
$$

\begin{theorem}[Biran-Cornea~\cite{Biran-Cornea-rigidity-uniruling}]
The quantum homologies 
$$
C=Q\H_*(M;\Lambda) \qquad \mbox{and} \qquad A=Q\H_*(L)
$$ 
fit into a graded 2D open-closed TQFT. \qed
\end{theorem}

This theorem is not stated as such by Biran-Cornea, but it essentially follows from their work as we now explain. 

The biunital coFrobenius bialgebra structure on $C$, also called Frobenius algebra structure in the literature, is determined by the product $\boldmu_C$ and the nondegenerate pairing $\boldp_C$. To describe them, note that additively we have 
$$
Q\H_*(M;\Lambda)=\H_*(M)\otimes_\bk \Lambda.
$$
The product $\boldmu_C$ is the \emph{quantum intersection product}, which extends the intersection product on $\H_*(M)$. The canonical $\bk$-linear map $\Lambda\mapsto \bk$ which associates to a Laurent polynomial in $t$ the coefficient of $1=t^0$ gives rise to a canonical nondegenerate pairing 
$$
\boldp_C:Q\H_*(M;\Lambda)\otimes_\bk Q\H_*(M;\Lambda)\to \bk
$$
that restricts to the canonical pairing on $\H_*(M)$. The key compatibility relation $\boldp_C(\boldmu_C\otimes 1)=\boldp_C(1\otimes\boldmu_C)$ is proved in~\cite[Proposition~11.1.9]{MS04}. 

Unlike the quantum homology of the ambient manifold, the quantum homology of a Lagrangian submanifold $L$ cannot be directly expressed in terms of the homology of $L$: there is a spectral sequence that starts at $\H_*(L)\otimes_\bk \Lambda$ and converges to $Q\H_*(L)$ (Oh~\cite{Oh96}, see also~\cite[Theorem~A(iv)]{Biran-Cornea-rigidity-uniruling}). This spectral sequence arises from the ``pearl complex'' of Biran and Cornea~\cite{Biran-Cornea-rigidity-uniruling}, which describes the differential for quantum homology as a perturbation of the Morse differential whose higher order terms involve pseudoholomorphic discs with boundary on $L$. 

The Frobenius algebra structure on $A$ is determined by the product $\boldmu_A$ and the nondegenerate pairing $\boldp_A$. The product $\boldmu_A$ is the \emph{quantum intersection product on $L$}, see~\cite[Theorem~A(ii)]{Biran-Cornea-rigidity-uniruling}. It is associative, but not necessarily commutative. The nondegenerate pairing 
$$
\boldp_A:Q\H_*(L)\otimes_\bk Q\H_*(L)\to \bk
$$
arises by applying to the quantum product a canonical augmentation, which is described at the level of Morse chains as follows: it associates to each minimum of the Morse function that underlies the pearl complex the coefficient of $1=t^0$ in its quantum coefficient, and it vanishes on critical points of positive Morse index. That this pairing is nondegenerate is proved in~\cite[Proposition~4.4.1]{Biran-Cornea-rigidity-uniruling}, see also~\cite[Remark~4.4.7.b]{Biran-Cornea-rigidity-uniruling}.

In the terminology of~\cite[Theorem~A(ii)]{Biran-Cornea-rigidity-uniruling}, the zipper $\boldzeta:Q\H_*(M;\Lambda)\to Q\H_*(L)$ is defined from the two-sided algebra structure of $Q\H_*(L)$ over $Q\H_*(M;\Lambda)$ by applying an element of $Q\H_*(M;\Lambda)$ to the unit of $Q\H_*(L)$. That the zipper is an algebra homomorphism (Definition~\ref{defi:gradedTQFT}(3)) and that its image lands in the center of $Q\H_*(L)$ (Definition~\ref{defi:gradedTQFT}(4)) are then consequences of this two-sided algebra structure. 

In the terminology of~\cite[Theorem~A(iv)]{Biran-Cornea-rigidity-uniruling}, the cozipper $\boldzeta^*:Q\H_*(L)\to Q\H_{*-n}(M;\Lambda)$ is called \emph{quantum inclusion}.  That the cozipper is dual to the zipper in the sense of Definition~\ref{defi:gradedTQFT}(5) follows from~\cite[Theorem~A(iii)]{Biran-Cornea-rigidity-uniruling}. 

The above results are proved by Biran and Cornea in~\cite{Biran-Cornea-rigidity-uniruling} with $\Z/2$-coefficients, and in~\cite[Appendix~A]{Biran-Cornea-Lagrangian-topology} with arbitrary coefficients. A related study of TQFT relations with arbitrary coefficients in the context of symplectic homology was carried out by Ritter~\cite{Ritter}. 

The only axiom of a graded 2D open-closed TQFT that is not explicitly proved in the work of Biran-Cornea is the Cardy condition from Definition~\ref{defi:gradedTQFT}(6). The proof involves an interpolation between annuli of modulus zero and annuli of modulus infinity in the moduli space of annuli with one marked point on each of the two boundary components. The actual implementation of this interpolation in Floer theory may vary depending on the context: in Hamiltonian Floer homology, or in symplectic homology, one interprets marked points as punctures and chooses Floer data in a coherent way over the various strata of this moduli space, see~\cite{Abouzaid2010a}; in quantum homology, one must add semi-infinite gradient trajectories of Morse functions at the marked points and upgrade the interpolation to such ``pearly'' configurations.

In the context of this paper, we will gloss over the Floer theoretic implementation and only describe the interpolation in the moduli space of annuli. See Figure~\ref{fig:Cardy-proof}.

\begin{figure} [ht]
\centering
\includegraphics[width=\textwidth]{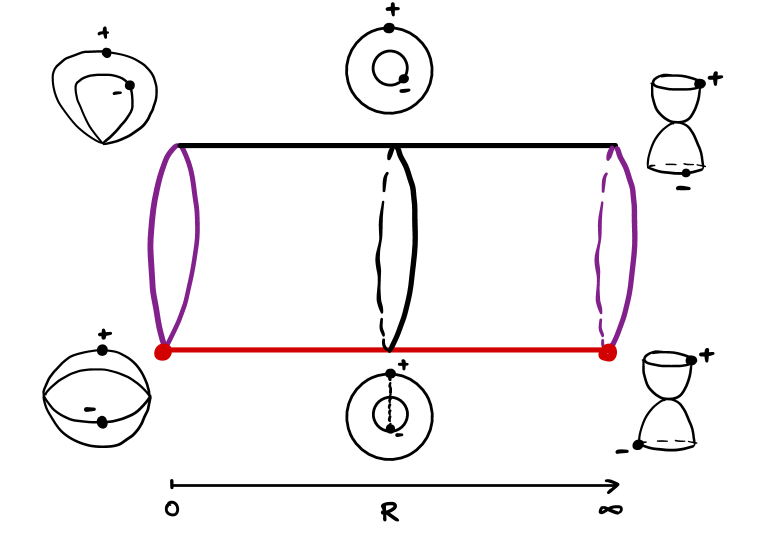}
\caption{The compactified moduli space of annuli with one marked point on each boundary component. The horizontal coordinate $\mu\in[0,\infty]$ is the modulus, and the values $\mu=0$ and $\mu=\infty$ correspond to nodal annuli. The top horizontal line describes annuli with aligned marked points, and the bottom horizontal line describes annuli with antipodal marked points. When implemented in Floer theory, the latter provides a chain homotopy between $\boldmu_A\tau\boldlambda_A$ (endpoint at $0$) and $\boldzeta\boldzeta^*$ (endpoint at $\infty$). On the right of the figure one can see depicted the family of lines that are used in the definition of the angle $\theta$ between the two marked points.}
\label{fig:Cardy-proof}
\end{figure}

Given an annulus $\cA$ with one marked point on each boundary component, we label the marked points $z_{1,2}$ and refer to them as ``the first'' and ``the second'' marked point. We write $\p\cA=\p_1\sqcup \p_2$ with $z_1\in\p_1$, $z_2\in\p_2$, and we refer to $\p_{1,2}$ as ``the first'' and ``the second'' boundary component. The moduli space $\cM$ of such annuli is diffeomorphic to an open cylinder $\cM\simeq S^1\times (0,\infty)$: the coordinate $\mu\in (0,\infty)$ is the modulus, and the coordinate $\theta\in S^1=\R/2\pi\Z$ is the relative position of the two marked points. To define the \emph{modulus} $\mu$, one uses that any annulus $\cA$ is biholomorphic to a standard annulus $\cA_{1,R}= \bar B(0,R)\setminus  B(0,1)\subset \C$ for a unique $R>1$, and we set $\mu=\log R$. To define $\theta$, one first defines the notion of the two marked points being ``aligned'', which corresponds to $\theta=0$: a biholomorphism $\cA\simeq \cA_{1,R}$ is unique up to postcomposing with a rotation in $\C$. The $S^1$-family of segments $\{re^{it}\, : \, 1\le r\le R\}_{t\in S^1}$ inside $\cA_{1,R}$ therefore determines by pull-back a canonical family of curves in $\cA$, and we say that two marked points on $\p \cA$ are \emph{aligned} if they are the endpoints of one of the curves in this family. The coordinate $\theta$ is then given by the angle between the second marked point $z_2$ and the point on $\p_2$ that is aligned with $z_1$, measured according to the boundary orientation of $\p_2$. We say that the two marked points are \emph{antipodal} if the corresponding configuration corresponds to $\theta=\pi\in S^1$. 

The moduli space $\cM$ admits a compactification $\bar\cM\simeq S^1\times [0,\infty]$ by nodal annuli. 
\begin{itemize}
\item The compactification at $+\infty$ is described as follows: an annulus of infinite modulus with one marked point on each boundary component is a curve of genus 0 with two boundary components, one interior node, and one marked point on each boundary component. Such a curve is the union of two discs, each with a marked point on the boundary and one marked point in the interior, together with an identification of the two interior points (the node), and also a complex anti-linear identification of the tangent spaces at these marked points. We denote as before the two marked points on the boundary $z_{1,2}$, and the corresponding boundary components $\p_{1,2}$. The complex anti-linear identification at the node allows to determine the marked point on $\p_2$ that is aligned with $z_1$, and therefore define a relative position $\theta\in S^1$ for $z_1$ and $z_2$. 
\item The compactification at $0$ is described as follows: for $\theta\neq \pi$, the underlying annulus is a genus $0$ curve with two boundary components and one boundary node. It can be realized explicitly as the one point compactification of the strip $\R\times [0,1]$ with one marked point on each of the two boundary components. The relative position $s\in\R$ of the two marked points determines $\theta=e^{2i\arctan s}\in S^1\setminus \{\pi\}$. The value $\theta=\pi$ corresponds to one of the marked points colliding with the node and forming a disc bubble. The resulting curve can be described by taking two copies of the strip $\R\times[0,1]$ compactified by a point at each of the ends, and identifying the compactifying points two by two in order to form boundary nodes. In the process, the marked points distribute as boundary marked points on each of the two strips.     
\end{itemize}

In Figure~\ref{fig:Cardy-proof} we have depicted the compactified moduli space $\bar \cM$. The modulus coordinate $\mu\in[0,\infty]$ is depicted horizontally. The line $\{\pi\}\times [0,\infty]$ is depicted at the bottom, and interpolates between the configuration at $(\pi,0)$ and the configuration at $(\pi,\infty)$. When implemented in Floer theory at chain level, the configuration at $(\pi,0)$ gives rise to the composition $\boldmu_A \tau \boldlambda_A$, the configuration at $(\pi,\infty)$ gives rise to the composition $\boldzeta\boldzeta^*$, and the homotopy $\{\pi\}\times[0,\infty]$ gives rise to a chain homotopy between these operations.

\section{Symplectic homology and Rabinowitz Floer homology}\label{sec:RFH}

Our main source of examples for the algebraic structures in this paper are symplectic homology and Rabinowitz Floer homology, which we describe in this section following~\cite{CHO-PD,CHO-reducedSH}. 

Consider a Liouville domain $V$ of dimension $2n$ and an exact oriented compact Lagrangian $L\subset V$ with Legendrian boundary $\p L\subset\p V$ such that $2c_1(V,L)=0$. The associated Floer homology and cohomology groups are taken with coefficients in a fixed discrete field $\bk$ and graded by Conley--Zehnder indices. The degree shifted symplectic homology $S\H_*(V)=SH_{*+n}(V)$, symplectic cohomology $S\H^*(V)=SH^{*+n}(V)$, and Rabinowitz Floer homology $S\H_*(\p V)=SH_{*+n}(\p V)$ fit into a long exact sequence
$$
  \cdots S\H^{-2n-*}(V) \stackrel{\eps}\longrightarrow S\H_*(V) \stackrel{\iota}\longrightarrow S\H_*(\p V) \stackrel{\pi}\longrightarrow S\H^{1-2n-*}(V)\cdots
$$
It follows that reduced symplectic homology $\ol{S\H}_*(V) = \coker\,\eps$ and cohomology $\ol{S\H}^*(V) = \ker\eps$ fit into the short exact sequence
$$
  0 \longrightarrow \ol{S\H}_*(V) \stackrel{\iota}\longrightarrow S\H_*(\p V) \stackrel{\pi}\longrightarrow \ol{S\H}^{1-2n-*}(V) \longrightarrow 0.
$$
Similarly, degree shifted symplectic homology $S\H_*(L)=SH_{*+n}(L)$, symplectic cohomology $S\H^*(L)=SH^{*+n}(L)$\footnote{
$S\H_*(L)$ and $S\H^*(L)$ are also known as wrapped Floer (co)homology.},
and Rabinowitz Floer homology $S\H_*(\p L)=SH_{*+n}(\p L)$ fit into a long exact sequence
$$
  \cdots S\H^{-n-*}(L) \stackrel{\eps}\longrightarrow S\H_*(L) \stackrel{\iota}\longrightarrow S\H_*(\p L) \stackrel{\pi}\longrightarrow S\H^{1-n-*}(L)\cdots
$$
It follows that reduced symplectic homology $\ol{S\H}_*(L) = \coker\,\eps$ and cohomology $\ol{S\H}^*(L) = \ker\eps$ fit into the short exact sequence
$$
  0 \stackrel{\eps}\longrightarrow \ol{S\H}_*(L) \stackrel{\iota}\longrightarrow S\H_*(\p L) \stackrel{\pi}\longrightarrow \ol{S\H}^{1-n-*}(L) \longrightarrow 0.
$$
The symplectic homologies $S\H_*(V)$ and $S\H_*(L)$ are colimits of finite dimensional vector spaces and thus discrete, so the symplectic cohomologies $S\H^*(V)$ and $S\H^*(L)$ are linearly compact. 
It is proved in~\cite{CO-Tate} that the Rabinowitz Floer homologies $S\H_*(\p V)$ and $S\H_*(\p L)$ are Tate vector spaces.

\begin{theorem}[{\cite{CHO-PD}}] \label{thm:open-closed-RFH}
The Rabinowitz Floer homologies $C=S\H_*(\p V)$ and $A=S\H_*(\p L)$ fit into a graded 2D open-closed TQFT. \qed
\end{theorem}

In order to get a good algebraic structure on reduced symplectic homology, we introduced in~\cite{CHO-reducedSH} the notions of a strongly essential Weinstein domain $V$ and a strongly essential Lagrangian $L\subset V$. Rather than repeating the definition here, let us just mention that it includes the disc cotangent bundles and disc cotangent fibres discussed below. 

\begin{theorem}[\cite{CHO-reducedSH}] \label{thm:ccuias}
(a) Let $V$ be a strongly essential Weinstein domain of dimension $2n\geq 6$. 
Then $\ol{S\H}_*(V)$ is a commutative cocommutative unital infinitesimal anti-symmetric bialgebra. 

(b) Let $L$ be a strongly essential Lagrangian of dimension $n\ge 2$. Assume in addition the following: $L$ is diffeomorphic to a disc if $n=2$, and $L$ has no homology in dimension $2$ if $n=4$. Then $\ol{S\H}_*(L)$ is a unital infinitesimal anti-symmetric bialgebra. \qed
\end{theorem}

By duality, it follows that under the hypotheses of the theorem $\ol{S\H}^*(V)$ and $\ol{S\H}^*(L)$ are counital infinitesimal anti-symmetric bialgebras. 

{\bf Loop homology. }
Consider now the particular case $V=D^*Q$ and $L=D^*_qQ$ where $Q$ is a closed  oriented manifold, $D^*Q$ its disc cotangent bundle, and $D^*_qQ\subset D^*Q$ its disc cotangent fibre at a fixed point $q\in Q$. Denote by $\Lambda Q$ the free loop space, and by $\Om Q$ the based loop space at $q$. Then we have isomorphisms~\cite{CHO-MorseFloerGH}\footnote{On the loop space side we need to twist by a suitable local system if $Q$ is not spin, see~\cite{CHO-MorseFloerGH}.}
\begin{gather*}
  S\H_*(\p D^*Q)=\wh \H_*(\Lambda Q), \qquad \qquad \qquad S\H_*(\p D^*_qQ)=\wh H_*(\Omega Q), \cr
  \ol{S\H}_*(D^*Q)=\ol{\H}_*(\Lambda Q), \qquad \ol{S\H}_*(D^*_qQ)=S\H_*(D^*_qQ)=H_*(\Omega Q),
\end{gather*}
where $\wh\H_*(\Lambda Q)=\wh H_{*+n}(\Lambda Q)$ is degree shifted Rabinowitz loop homology, $\wh H_*(\Omega Q)$ is based Rabinowitz loop homology, and $\ol{\H}_*(\Lambda Q)=H_{*+n}(\Lambda Q)/\chi(Q)[{\rm pt}]$ is degree shifted loop homology modulo Euler characteristic times the point class. Therefore, the preceding theorems have the following corollary.

\begin{corollary}\label{cor:loop}
Let $Q$ be a closed oriented $n$-dimensional manifold. 

(a) Reduced loop homology $\ol{\H}_*(\Lambda Q)$ is a commutative cocommutative unital infinitesimal anti-symmetric bialgebra if $n\ge 3$. 

(b) Based loop homology $H_*(\Om Q)$ is a unital infinitesimal\break anti-symmetric bialgebra. 

(c) The Rabinowitz loop homologies $\wh\H_*(\Lambda Q)$ and $\wh H_*(\Omega Q)$ fit into a graded 2D open-closed TQFT such that $\boldzeta\boldzeta^*=0$. 
\qed 
\end{corollary}

We refer to~\S\ref{sec:spheres} for explicit formulas of the relevant operations in the case of odd-dimensional spheres. Corollary~\ref{cor:loop}(b) does not follow in dimension $n=1$ from Theorem~\ref{thm:ccuias}(b), but rather from those explicit computations in the case $Q=S^1$.

Following~\cite{CHO-reducedSH}, parts (b) and (c) of Corollary~\ref{cor:loop} can be generalized by replacing $D^*_qQ$ by $D^*_ZQ$, the unit conormal bundle of a closed oriented submanifold $Z\subset Q$, and $\Omega Q$ by $\Om_ZQ$, the space of paths in $Q$ with endpoints on $Z$. 

\begin{corollary}\label{cor:loop-Z}
Let $Q$ be a closed oriented $n$-dimensional manifold and $Z\subset Q$ a closed oriented $d$-dimensional submanifold. 

(a) Assume $d\leq n/2$ and, in addition, $Z$ is a point if $n=2$ and $Z$ is a point or a circle if $n=4$. Then $H_*(\Om_ZQ)$ is a unital infinitesimal anti-symmetric bialgebra. 

(b) The Rabinowitz homologies $\wh\H_*(\Lambda Q)$ and $\wh H_*(\Omega_ZQ)$ fit into a graded 2D open-closed TQFT such that $\boldzeta\boldzeta^*=0$. \qed 
\end{corollary}

\begin{proof}
By~\cite{Abbondandolo-Portaluri-Schwarz} we have $S\H_*(D^*_ZQ)=H_{*+d}(\Omega_Z Q)$. The Lagrangian $D^*_ZQ$ is strongly essential if and only if $d\le n/2$, and part~(a) then follows from Theorem~\ref{thm:ccuias}(b). Part~(b) follows from Theorem~\ref{thm:open-closed-RFH}. 
\end{proof}

\begin{remark}
The graded 2D open-closed TQFT structure in Corollary~\ref{cor:loop-Z}(b) does not necessarily restrict to such a structure on the energy zero (constant path) sector. The reason is that on the energy zero sector all spaces are finite dimensional, so the Cardy condition also contains the requirement $\boldmu_A\boldlambda_A=0$, which is equivalent to the unit sphere conormal bundle $S^*_ZQ$ having vanishing Euler characteristic. In view of the product formula $\chi(S^*_ZQ)=\chi(S^{n-d-1})\chi(Z)$, this is equivalent to $n-d$ even or $\chi(Z)=0$. The same product formula shows that $\boldmu_A\boldlambda_A=0$ with coefficients in a field of characteristic $2$.
\end{remark}

\section{Loop space homology of odd-dimensional spheres}\label{sec:spheres}

In this section we illustrate the algebraic structures of this paper arising from loop spaces of odd-dimensional spheres $S^n$.\footnote{We do not discuss even-dimensional spheres because the full structure has not been explicitly computed in the literature. We refer to~\cite{COS-Cross} for the computation of the BV algebra structure for Rabinowitz loop homology, from which the full 2D open-closed TQFT structure can be derived.} 
Here we present only the results, referring to~\cite{CHO-PD,CHO-MorseFloerGH} for their derivation. 
For the free loop space $\Lambda S^n$, we denote by 
$$
   \H_*\Lambda S^n = H_{n+*}\Lambda S^n\quad\text{and}\quad
   \wh{\H}_*\Lambda S^n = \wh{H}_{n+*}\Lambda S^n
$$
the degree shifted loop homology and Rabinowitz loop homology, respectively. 
For the based loop space $\Om S^n$, we denote by
$$
   H_*\Om S^n \quad\text{and}\quad
   \wh{H}_*\Om S^n 
$$
the based loop homology and based Rabinowitz loop homology, respectively. 
We need to distinguish the cases $n\geq 3$ and $n=1$. 
\medskip

{\bf The case of odd $n\geq 3$. } 

{\it Rabinowitz loop homology $\wh{\H}_*\Lambda S^n$. }
As a ring with respect to $\boldmu$, the Rabinowitz loop homology is given by
$$
   \wh{\H}_*\Lambda S^n = \Lambda[A,U,U^{-1}],\qquad |U|=n-1,\ |A|=-n.
$$
Here $\Lambda$ denotes the exterior algebra and we denote generators in homology by capital letters. 
Thus $\boldmu$ has degree $0$, it is associative, commutative, and unital with unit $\boldeta=\bold1$.  
(We write $\bold1$ in boldface to distinguish it from the identity map $1$.)
Note that $\wh{\H}_*\Lambda S^n$ is finite dimensional in each degree and, as such, it is discrete. 
The coproduct is given by
\begin{align*}
   \boldlambda(AU^k) &= \sum_{i+j=k-1} AU^i\otimes AU^j,\cr
   \boldlambda(U^k) &= \sum_{i+j=k-1} (AU^i\otimes U^j-U^i\otimes AU^j).
\end{align*}
Note that $\boldlambda$ takes values in $\wh{\H}_*\Lambda S^n\otimes^! \wh{\H}_*\Lambda S^n$, see also Remark~\ref{rmk:discrete-infinite-degrees}. 
The coproduct $\boldlambda$ has odd degree $1-2n$, it is
coassociative, cocommuta\-tive, 
and counital with counit 
$$
  \boldeps(U^k)=0,\qquad \boldeps(AU^k)=\begin{cases}
  1, & k=-1,\cr 0, & \text{else}. \end{cases}
$$
One can verify by direct computation that $(\wh{\H}_*\Lambda S^n,\boldmu,\boldlambda,\boldeta,\boldeps)$ is a biunital coFrobenius bialgebra. 
The copairing is
$$
   \boldc = \boldlambda(\bold1) = \sum_{i+j=-1} (AU^i\otimes U^j-U^i\otimes AU^j),
$$
and the pairing $\boldp=-\boldeps\boldmu$ is given by
\begin{align*}
  \boldp(U^i\otimes U^j) = 0, \qquad \boldp (AU^i\otimes&  AU^j) = 0,\cr
  \boldp(AU^i\otimes U^j) = \boldp(U^i\otimes AU^j) &= \begin{cases}
  -1, & i+j=-1,\cr 0, & \text{else}. \end{cases}
\end{align*}
Denoting by $\{(U^k)^\vee, (AU^k)^\vee\, : \, k\in \Z\}$ the basis dual to the basis $\{U^k,AU^k \, : \, k\in\Z\}$, the Poincar\'e duality isomorphism is given by$$
  \vec\boldp(U^i) = -(AU^{-i-1})^\vee,\qquad
  \vec\boldp(AU^i) = -(U^{-i-1})^\vee.
$$

{\it Loop homology $\H_*\Lambda S^n$. }
As a ring with respect to the loop product $\boldmu$, ordinary loop homology is given by
$$
   \H_*\Lambda S^n = \Lambda[A,U],\qquad |U|=n-1,\ |A|=-n.
$$
Thus $\boldmu$ has degree $0$, it is associative, commutative, and unital with unit $\boldeta=\bold1$.  
The loop coproduct is given by
\begin{align*}
   \boldlambda(AU^k) &= \sum_{i+j=k-1\atop i,j\geq 0} AU^i\otimes AU^j,\cr
   \boldlambda(U^k) &= \sum_{i+j=k-1\atop i,j\geq 0} (AU^i\otimes U^j-U^i\otimes AU^j).
\end{align*}
Note that $\boldlambda$ takes values in the algebraic tensor product. 
The coproduct $\boldlambda$ has odd degree $1-2n$, it is coassociative and cocommutative, but it has no counit. 
One can verify by direct computation that $(\H_*\Lambda S^n,\boldmu,\boldlambda,\boldeta)$ is a unital infinitesimal anti-symmetric bialgebra.
Since
$$
   \boldlambda\boldeta = \boldlambda(\bold1) = 0,
$$
the {\sc (unital infinitesimal relation)} simplifies to Sullivan's relation
$$
  \boldlambda\boldmu = (1\otimes\boldmu)(\boldlambda\otimes 1) + (\boldmu\otimes 1)(1\otimes\boldlambda).
$$

{\it Based Rabinowitz loop homology $\wh{H}_*\Om S^n$. }
As a ring with respect to $\boldmu$, the based Rabinowitz loop homology is given by
$$
   \wh{H}_*\Om S^n = \Lambda[U,U^{-1}],\qquad |U|=n-1.
$$
Thus $\boldmu$ has degree $0$, it is associative, commutative, and unital with unit $\boldeta=\bold1$.  
Note that $\wh{H}_*\Om S^n$ is of finite type. 
The coproduct is given by
\begin{align*}
   \boldlambda(U^k) &= \sum_{i+j=k-1} U^i\otimes U^j.
\end{align*}
Note that $\boldlambda$ takes values in $\wh{H}_*\Om S^n\otimes^! \wh{H}_*\Om S^n$.
The coproduct $\boldlambda$ has even degree $1-n$, it is coassociative, cocommutative, and counital with counit 
$$
  \boldeps(U^k) = \begin{cases}
  1, & k=-1,\cr 0, & \text{else}. \end{cases}
$$
One can verify by direct computation that $(\wh{H}_*\Om S^n,\boldmu,\boldlambda,\boldeta,\boldeps)$ is a biunital coFrobenius bialgebra. 
The copairing is
$$
   \boldc = \boldlambda(\bold1) = \sum_{i+j=-1} U^i\otimes U^j,
$$
and the pairing $\boldp=\boldeps\boldmu$ is given by
\begin{align*}
  \boldp(U^i\otimes U^j) &= \begin{cases}
  1, & i+j=-1,\cr 0, & \text{else}. \end{cases}
\end{align*}
Denoting by $\{(U^k)^\vee\, : \, k\in\Z\}$ the basis dual to the basis $\{U^k\, : \, k\in\Z\}$, the Poincar\'e duality isomorphism is given by
$$
  \vec\boldp(U^i) = (U^{-i-1})^\vee.
$$

{\it Based loop homology $H_*\Om S^n$. }
As a ring with respect to the Pontrjagin product $\boldmu$, ordinary based loop homology is given by
$$
   H_*\Om S^n = \Lambda[U],\qquad |U|=n-1.
$$
Thus $\boldmu$ has degree $0$, it is associative, commutative, and unital with unit $\boldeta=\bold1$.  
The based loop coproduct is given by
\begin{align*}
   \boldlambda(U^k) &= \sum_{i+j=k-1\atop i,j\geq 0} U^i\otimes U^j.
\end{align*}
Here $\boldlambda$ takes values in the algebraic tensor product.
The coproduct $\boldlambda$ has even degree $1-n$, it is coassociative and cocommutative, but it has no counit. 
One can verify by direct computation that $(H_*\Om S^n,\boldmu,\boldlambda,\boldeta)$ is a unital infinitesimal anti-symmetric bialgebra. 
Since
$$
   \boldlambda\boldeta = \boldlambda(\bold1) = 0,
$$
the {\sc (unital infinitesimal relation)} simplifies to Sullivan's relation
$$
  \boldlambda\boldmu = (1\otimes\boldmu)(\boldlambda\otimes 1) + (\boldmu\otimes 1)(1\otimes\boldlambda).
$$

{\it Graded 2D open-closed TQFT structure on Rabinowitz (based) loop homology}.
We illustrate Theorem~\ref{thm:open-closed-RFH} with closed sector $\cC=\wh{\H}_*\Lambda S^n$ and open sector $\cA=\wh{H}_*\Om S^n$. We label the previous generators $U$ as $U_\cC$ in $\cC$, and $U_\cA$ in $\cA$. 
The zipper 
$$
\boldzeta: \wh{\H}_*\Lambda S^n\to \wh{H}_*\Om S^n
$$ 
acts by 
$$
\boldzeta(AU_\cC^k)=0,\qquad \boldzeta(U_\cC^k)=U_\cA^k, \ k\in\Z.
$$
This map has degree $0$ and is an algebra map. It can be understood geometrically as an extension to Rabinowitz loop homologies of the topological shriek map $\H_*\Lambda S^n\to H_*\Omega S^n$ induced by the inclusion $\Omega S^n\hookrightarrow \Lambda S^n$, see also~\cite{CHO-PD}. 
The cozipper 
$$
\boldzeta^*: \wh{H}_*\Om S^n\to \wh{\H}_{*-n}\Lambda S^n
$$
acts by
$$
\boldzeta^*U_\cA^k=AU_\cC^k,\ k\in\Z.
$$
This map has degree $-n$ and is a coalgebra map. In this particular case we have $|\boldzeta^*|=-n$ and $|\boldlambda_\cA|=1-n$, so the compatibility with the coproducts writes $(\boldzeta^*\otimes\boldzeta^*)\boldlambda_\cA=\boldlambda_\cC\boldzeta^*$. The cozipper can be understood geometrically as an extension to Rabinowitz loop homology of the pushforward map $H_*\Omega S^n\to H_*\Lambda S^n$ induced in homology by the inclusion $\Omega S^n\hookrightarrow \Lambda S^n$, see also~\cite{CHO-PD}. 
The verification of the compatibility conditions in Definition~\ref{defi:gradedTQFT} is straightforward. 

\bigskip 

{\bf The case $n=1$. } 

{\it Rabinowitz loop homology $\wh{\H}_*\Lambda S^1$. }
As a biunital coFrobenius bialgebra, the Rabinowitz loop homology of $S^1$ is a direct sum (see Remark~\ref{rem:direct-sum})
$$
   \wh{\H}_*\Lambda S^1 = \Lambda[A_+,U_+,U_+^{-1}]]\oplus \Lambda[A_-,U_-^{-1},U_-]],\quad |U_\pm|=0,\ |A_\pm|=-1.
$$
Here the two summands correspond to the two connected components of the unit sphere cotangent bundle $S^*S^1=\{+1,-1\}\times S^1$, and our convention is chosen such that a term $U_\pm^k$ has winding number $k\in\Z$ around the circle. Since both summands are identical under the transformation $U_+\to U_-^{-1}$, let us describe one of them, dropping the subscript $\pm$. The graded Tate vector space
$$
   \Lambda[A,U,U^{-1}]],\qquad |U|=0,\ |A|=-1
$$
consists of Laurent series
$$
   \sum_{i=-\infty}^N(a_iU^i+b_iAU^i),\qquad a_i,b_i\in \bk,\ N\in\N
$$
with product given by the usual multiplication. Thus $\boldmu$ has degree $0$, it is associative, commutative, and unital with unit $\boldeta=\bold1$.  
Note that this graded Tate vector space is supported in degrees $0$ and $-1$ and is infinite dimensional.  
The coproduct is given by
\begin{align*}
   \boldlambda(AU^k) &= \sum_{i+j=k} AU^i\otimes AU^j,\cr
   \boldlambda(U^k) &= \sum_{i+j=k} (AU^i\otimes U^j-U^i\otimes AU^j).
\end{align*}
Note that $\boldlambda$ takes values in the ! tensor product.  
Note also the similarity to the coproduct on $\wh{\H}_*\Lambda S^n$ for $n\geq 3$ odd, with the difference that the condition $i+j=k-1$ in the sums now becomes $i+j=k$. 
The coproduct $\boldlambda$ has odd degree $-1$, it is coassociative, cocommutative, and counital with counit 
$$
  \boldeps(U^k)=0,\qquad \boldeps(AU^k)=\begin{cases}
  1, & k=0,\cr 0, & \text{else}. \end{cases}
$$
One can verify by direct computation that $(\wh{\H}_*\Lambda S^1,\boldmu,\boldlambda,\boldeta,\boldeps)$ is a biunital coFrobenius bialgebra. 
The copairing is
$$
   \boldc = \boldlambda(\bold1) = \sum_{i+j=0} (AU^i\otimes U^j-U^i\otimes AU^j),
$$
and the pairing $\boldp=-\boldeps\boldmu$ is given by
\begin{align*}
  \boldp(U^i\otimes U^j) = 0,\qquad \boldp(AU^i\otimes & AU^j) = 0,\cr
  \boldp(AU^i\otimes U^j) = \boldp(U^i\otimes AU^j) &= \begin{cases}
  -1, & i+j=0,\cr 0, & \text{else}. \end{cases}
\end{align*}
Denoting by $\{(U^k)^\vee, (AU^k)^\vee \, : \, k\in\Z\}$ the basis dual to the basis $\{U^k,AU^k\, : \, k\in\Z\}$, 
the Poincar\'e duality isomorphism is given by
$$
  \vec\boldp(U^i) = -(AU^{-i})^\vee,\qquad
  \vec\boldp(AU^i) = -(U^{-i})^\vee.
$$

{\it Loop homology $\H_*\Lambda S^1$. }
As a ring with respect to the loop product $\boldmu$, ordinary loop homology is given by
$$
   \H_*\Lambda S^1 = \Lambda[A,U,U^{-1}],\qquad |U|=0,\ |A|=-1.
$$
Thus $\boldmu$ has degree $0$, it is associative, commutative, and unital with unit $\boldeta=\bold1$.  

The loop coproduct in this case depends on the choice of a nowhere vanishing vector field on $S^1$. 
Up to homotopy there are two such choices $v^\pm(x) = \pm 1$, giving rise to two coproducts 
\begin{align*}
  \boldlambda_+(AU^k) &= \begin{cases}
     \sum_{i=0}^{k}AU^i\otimes AU^{k-i} & k\geq 0, \\
     -\sum_{i=k+1}^{-1}AU^i\otimes AU^{k-i} & k<0
  \end{cases} \cr 
  \boldlambda_+(U^k) &= \begin{cases}
     \sum_{i=0}^{k}(AU^i\otimes U^{k-i} - U^i\otimes AU^{k-i}) & k\geq 0, \\
     -\sum_{i=k+1}^{-1}(AU^i\otimes U^{k-i} - U^i\otimes AU^{k-i}) & k<0
  \end{cases}  
\end{align*}
\begin{align*}
  \boldlambda_-(AU^k) &= \begin{cases}
     \sum_{i=1}^{k-1}AU^i\otimes AU^{k-i} & k > 0, \\
     -\sum_{i=k}^{0}AU^i\otimes AU^{k-i} & k\leq 0,
  \end{cases} \cr 
  \boldlambda_-(U^k) &= \begin{cases}
     \sum_{i=1}^{k-1}(AU^i\otimes U^{k-i} - U^i\otimes AU^{k-i}) & k> 0, \\
     -\sum_{i=k}^{0}(AU^i\otimes U^{k-i} - U^i\otimes AU^{k-i}) & k\leq 0.
  \end{cases}  
\end{align*}
Here $\boldlambda_\pm$ takes values in the algebraic tensor product.
The coproduct $\boldlambda_\pm$ has odd degree $-1$, it is coassociative and also cocommutative, but it has no counit. 
One can verify by direct computation that $(\H_*\Lambda S^1,\boldmu,\boldlambda_\pm,\boldeta)$ is a unital infinitesimal anti-symmetric bialgebra. 
Since
$$
   \boldlambda_\pm\boldeta = \boldlambda_\pm(\bold1) = \pm(A\otimes \bold1-\bold1\otimes A),
$$
the {\sc (unital infinitesimal relation)} and {\sc (unital anti-symm\-etry)} contain nontrivial terms involving $\boldlambda_\pm(\bold1)$. 

{\it Based Rabinowitz loop homology $\wh{H}_*\Om S^1$. }
As a biunital coFrobenius bialgebra, the based Rabinowitz loop homology of $S^1$ is a direct sum
$$
   \wh{H}_*\Om S^1 = \Lambda[U_+,U_+^{-1}]]\oplus \Lambda[U_-^{-1},U_-]],\qquad |U_\pm|=0.
$$
Let us again describe one summand, dropping the subscript $\pm$, 
$$
   \Lambda[U,U^{-1}]],\qquad |U|=0
$$
with product given by the usual multiplication. Thus $\boldmu$ has degree $0$, it is associative, commutative, and unital with unit $\boldeta=\bold1$.  
This graded Tate vector space is supported in degree 0 and coincides with the vector space of Laurent power series in $t=U^{-1}$ from Example~\ref{ex:Tate}. 
The coproduct is given by
\begin{align*}
   \boldlambda(U^k) &= \sum_{i+j=k} U^i\otimes U^j,
\end{align*}
taking values in $\wh{H}_*\Om S^1\otimes^! \wh{H}_*\Om S^1$. 
The coproduct $\boldlambda$ has even degree $0$, it is coassociative, cocommutative, and counital with counit
$$
  \boldeps(U^k)=\begin{cases}
  1, & k=0,\cr 0, & \text{else}. \end{cases}
$$
One can verify by direct computation that $(\wh{H}_*\Om S^1,\boldmu,\boldlambda,\boldeta,\boldeps)$ is a biunital coFrobenius bialgebra. 
The copairing is
$$
   \boldc = \boldlambda(\bold1) = \sum_{i+j=0} (U^i\otimes U^j),
$$
and the pairing $\boldp=\boldeps\boldmu$ is given by
\begin{align*}
  \boldp(U^i\otimes U^j) &= \begin{cases}
  1, & i+j=0,\cr 0, & \text{else}. \end{cases}
\end{align*}
Denoting by $\{(U^k)^\vee\, : \, k\in\Z\}$ the basis dual to the basis $\{U^k\, : \, k\in\Z\}$, 
the Poincar\'e duality isomorphism is given by
$$
  \vec\boldp(U^i) = (U^{-i})^\vee.
$$

{\it Based loop homology $H_*\Om S^1$. }
As a ring with respect to the Pontrjagin product $\boldmu$, ordinary based loop homology of $S^1$ is given by
$$
  H_*\Om S^1 = \Lambda[U,U^{-1}],\qquad |U|=0.
$$
Thus $\boldmu$ has degree $0$, it is associative, commutative, and unital with unit $\boldeta=\bold1$.  

Again, the two choices of nowhere vanishing vector fields $v^\pm(x) = \pm 1$ give rise to two coproducts 
\begin{align*}
  \boldlambda_+(U^k) &= \begin{cases}
     \sum_{i=0}^{k}U^i\otimes U^{k-i} & k\geq 0, \\
     -\sum_{i=k+1}^{-1}U^i\otimes U^{k-i} & k<0,
  \end{cases} \cr 
  \boldlambda_-(U^k) &= \begin{cases}
     \sum_{i=1}^{k-1}U^i\otimes U^{k-i} & k > 0, \\
     -\sum_{i=k}^{0}U^i\otimes U^{k-i} & k\leq 0
  \end{cases} 
\end{align*}
taking values in the algebraic tensor product. 
The coproduct $\boldlambda_\pm$ has even degree $0$, it is coassociative and cocommutative, but it has no counit. 
One can verify by direct computation that $(H_*\Om S^1,\boldmu,\boldlambda_\pm,\boldeta)$ is a unital infinitesimal anti-symmetric bialgebra. 
Since
$$
   \boldlambda_\pm\boldeta = \boldlambda_\pm(\bold1) = \pm(\bold1\otimes \bold1),
$$
the {\sc (unital infinitesimal relation)} simplifies to
$$
  \boldlambda_\pm\boldmu = (1\otimes\boldmu)(\boldlambda_\pm\otimes 1) + (\boldmu\otimes 1)(1\otimes\boldlambda_\pm) \mp \bold1\otimes \bold1.
$$
Thus $(H_*\Om S^1,\boldmu,\boldlambda_+,\boldeta)$ 
and $(H_*\Om S^1,\boldmu,-\boldlambda_-,\boldeta)$
are ``unital infinitesimal bialgebras'' in the sense of Loday--Ronco~\cite{Loday-Ronco}.

{\it Graded 2D open-closed TQFT structure on Rabinowitz (based) loop homology of the circle}. We indicate the zipper and cozipper maps which turn the pair $\cC=\wh{\H}_*\Lambda S^1$ and $\cA=\wh{H}_*\Om S^1$ into an open-closed TQFT. 
We label the previous generators $U_\pm$ as $U_{\pm,\cC}$ in $\cC$, and $U_{\pm,\cA}$ in $\cA$. The zipper 
$$
\boldzeta: \wh{\H}_*\Lambda S^1\to \wh{H}_*\Om S^1
$$ 
acts by 
$$
\boldzeta(A_\pm U_{\pm,\cC}^k)=0,\qquad \boldzeta(U_{\pm,\cC}^k)=U_{\pm,\cA}^k, \ k\in\Z.
$$
The cozipper 
$$
\boldzeta^*: \wh{H}_*\Om S^1\to \wh{\H}_{*-1}\Lambda S^1
$$
acts by
$$
\boldzeta^*U_{\pm,\cA}^k=A_\pm U_{\pm,\cC}^k,\ k\in\Z.
$$

As in the case of spheres of odd dimension $n\ge 3$, these maps can be understood geometrically as extensions to Rabinowitz loop homology of the shriek map, resp.~the pushforward map, induced in homology by the inclusion $\Omega S^1\hookrightarrow \Lambda S^1$. The verification of the compatibility conditions in Definition~\ref{defi:gradedTQFT} is straightforward.

\begin{remark}\label{rem:involutivity-spheres}
(i) The structures on loop homology $\H_*\Lambda S^n$ and based loop homology $H_*\Om S^n$ do {\em not} satisfy the {\sc (unital coFrobenius)} relation for any odd $n\geq 1$. For example, for $n\geq 3$ we have $\boldc=\boldlambda(\bold1) = 0$ but $\boldlambda\neq 0$.

(ii) In this section all products are commutative and all coproducts are cocommutative. On 
loop homology $\H_*\Lambda S^n$, the product and coproduct have opposite parity, in particular the structure is involutive (cf.~Remark~\ref{rem:involutivity}). On 
based loop homology $H_*\Om S^n$, the product and coproduct have the same parity and the structure is not involutive.  
\end{remark}

\appendix

\section{Multilinear operations and gradings}\label{app:gradings}

Fix a discrete field $\bk$. 
Consider $\Z_2$-graded linearly topologized $\bk$-vector spaces $A,B,C$ with multilinear operations
$$
  m_A:A^{\otimes^* k}\to A^{\otimes^!\ell},\qquad m_B:B^{\otimes^* k}\to B^{\otimes^!\ell},\qquad m_C:C^{\otimes^* k}\to C^{\otimes^!\ell}
$$
and linear maps
$$
  (A,m_A) \stackrel{\phi}\longrightarrow (B,m_B) \stackrel{\psi}\longrightarrow (C,m_C).
$$
We wish to define a notion of compatibility of maps with operations of the form
$$
  \phi^{\otimes\ell}m_A = (-1)^{\eps(|\phi|,|m_A|)}m_B\phi^{\otimes k},
$$
for some map $\eps:\Z_2\oplus\Z_2\to\Z_2$, which is preserved under composition. This means that we must also have
$$
  \psi^{\otimes\ell}m_B = (-1)^{\eps(|\psi|,|m_B|)}m_C\psi^{\otimes k}
$$
and
$$
  (\psi\phi)^{\otimes\ell}m_A = (-1)^{\eps(|\phi|+|\psi|,|m_A|)}m_C(\psi\phi)^{\otimes k}.
$$
The left had side of this equation is
$$
  (\psi\phi)^{\otimes\ell}m_A = (-1)^{\frac{\ell(\ell-1)}{2}|\phi||\psi|}\psi^{\otimes\ell}\phi^{\otimes\ell}m_A,
$$
and using the first two equations the term $m_C(\psi\phi)^{\otimes k}$ on the right hand side becomes
\begin{align*}
  m_C(\psi\phi)^{\otimes k}
  &= (-1)^{\frac{k(k-1)}{2}|\phi||\psi|}m_C\psi^{\otimes k}\phi^{\otimes k} \cr
  &= (-1)^{\frac{k(k-1)}{2}|\phi||\psi|+\eps(|\psi|,|m_B|)}\psi^{\otimes k}m_B\phi^{\otimes k} \cr
  &= (-1)^{\frac{k(k-1)}{2}|\phi||\psi|+\eps(|\psi|,|m_B|)+\eps(|\phi|,|m_A|)}\psi^{\otimes k}\phi^{\otimes k}m_A.
\end{align*}
So both sides are equal iff (mod $2$) we have
$$
  \eps(|\phi|+|\psi|,|m_A|) \equiv \eps(|\phi|,|m_A|) + \eps(|\psi|,|m_B|) + \frac{k(k-1)+\ell(\ell-1)}{2}|\phi||\psi|.
$$
Using $|m_B|\equiv |m_A|+(k+\ell)|\phi|$ and setting $x=|\phi|$, $y=|\psi|$, $z=|m_A|$, $a=k+\ell$, $b=\frac{k(k-1)+\ell(\ell-1)}{2}$ this becomes
$$
  \eps(x+y,z) \equiv \eps(x,z) + \eps(y,ax+z) + bxy.
$$
For $(x,y)=(0,0)$ this implies $\eps(0,z)=0$. With this, the equations for $(x,y)=(1,0)$ and for $(x,y)=(0,1)$ automatically hold, whereas the equation for $(x,y)=(1,1)$ becomes
$$
   \eps(1,z+a) \equiv \eps(1,z)+b.
$$
If $a\equiv b\equiv 0$ we can choose $\eps(1,z)$ arbitrarily, e.g.~$\eps(x,z)=0$ does the job. If $a\equiv 1$ the equation becomes $\eps(1,z+1)\equiv\eps(1,z)+b$, so we must have $\eps(1,z)=bz$ and $\eps(x,z)=bxz$. If $a\equiv 0$ and $b\equiv 1$ we do not get a solution; this occurs whenever $k+\ell\equiv 2$ mod $4$, e.g.~in the case $k=2$, $\ell=0$ of a pairing. 

{\bf Conclusion. }
For products $\mu$ or coproducts $\lambda$ we have $a\equiv b\equiv 1$, so we could define compatibility as in Definition~\ref{defi:compatibility_with_products} by
$$
  \phi\mu_A = (-1)^{|\phi||\mu_A|}\mu_B\phi^{\otimes 2},\qquad
  \phi^{\otimes 2}\lambda_A = (-1)^{|\phi||\lambda_A|}\lambda_B\phi.
$$
Applying this to the shift map $s:A\to A[1]$ would lead to
$$
   \ol\mu_A = (-1)^{|\mu_A|}s\mu_A(s\otimes s)^{-1},\qquad
   \ol\lambda_A = (-1)^{|\lambda_A|}(s\otimes s)\lambda_As^{-1},
$$
making compatibility with products/coproducts invariant under shifts.
For pairings or copairings there is no notion of compatibility which is invariant under composition.

\bibliographystyle{abbrv}
\bibliography{000_SHpair}

\end{document}